\documentclass[english]{article}
\usepackage[T1]{fontenc}
\usepackage[latin9]{inputenc}
\usepackage{geometry}
\usepackage{babel}
\usepackage{mathrsfs}
\usepackage{mathtools}
\usepackage{amsmath}
\usepackage{amsthm}
\usepackage{amssymb}
\usepackage[unicode=true,pdfusetitle,
 bookmarks=true,bookmarksnumbered=false,bookmarksopen=false,
 breaklinks=false,pdfborder={0 0 0},pdfborderstyle={},backref=false,colorlinks=false]
 {hyperref}

\makeatletter
\numberwithin{equation}{section}
\numberwithin{figure}{section}
\theoremstyle{plain}
\newtheorem{thm}{\protect\theoremname}[section]
\theoremstyle{definition}
\newtheorem{defn}[thm]{\protect\definitionname}
\theoremstyle{definition}
\newtheorem{example}[thm]{\protect\examplename}
\theoremstyle{plain}
\newtheorem{assumption}[thm]{\protect\assumptionname}
\theoremstyle{plain}
\newtheorem{prop}[thm]{\protect\propositionname}
\theoremstyle{remark}
\newtheorem{rem}[thm]{\protect\remarkname}
\theoremstyle{plain}
\newtheorem{cor}[thm]{\protect\corollaryname}
\theoremstyle{plain}
\newtheorem{lem}[thm]{\protect\lemmaname}

\allowdisplaybreaks

\makeatother

\providecommand{\assumptionname}{Assumption}
\providecommand{\corollaryname}{Corollary}
\providecommand{\definitionname}{Definition}
\providecommand{\examplename}{Example}
\providecommand{\lemmaname}{Lemma}
\providecommand{\propositionname}{Proposition}
\providecommand{\remarkname}{Remark}
\providecommand{\theoremname}{Theorem}

\begin{document}
\title{Risk aware minimum principle for optimal control of stochastic differential
equations}
\author{Jukka Isohätälä and William B. Haskell}
\maketitle
\begin{abstract}
We present a probabilistic formulation of risk aware optimal control
problems for stochastic differential equations. Risk awareness is
in our framework captured by objective functions in which the risk
neutral expectation is replaced by a risk function, a nonlinear functional
of random variables that account for the controller's risk preferences.
We state and prove a risk aware minimum principle that is a parsimonious
generalization of the well-known risk neutral, stochastic Pontryagin's
minimum principle. As our main results we give necessary and also
sufficient conditions for optimality of control processes taking values
on probability measures defined on a given action space. We show that
remarkably, going from the risk neutral to the risk aware case, the
minimum principle is simply modified by the introduction of one additional
real-valued stochastic process that acts as a risk adjustment factor
for given cost rate and terminal cost functions. This adjustment process
is explicitly given as the expectation, conditional on the filtration
at the given time, of an appropriately defined functional derivative
of the risk function evaluated at the random total cost. For our results
we rely on the Fréchet differentiability of the risk function, and
for completeness, we prove under mild assumptions the existence of
Fréchet derivatives of some common risk functions. We give a simple
application of the results for a portfolio allocation problem and
show that the risk awareness of the objective function gives rise
to a risk premium term that is characterized by the risk adjustment
process described above. This suggests uses of our results in e.g.
pricing of risk modeled by generic risk functions in financial applications.
\end{abstract}
\global\long\def\EE{\mathbb{E}}%

\global\long\def\PP{\mathbb{P}}%

\global\long\def\RR{\mathbb{R}}%

\global\long\def\ZZ{\mathbb{Z}}%

\global\long\def\NN{\mathbb{N}}%

\global\long\def\TT{\mathbb{T}}%
\global\long\def\UU{\mathbb{U}}%
\global\long\def\VV{\mathbb{V}}%
\global\long\def\WW{\mathbb{W}}%
\global\long\def\SS{\mathbb{S}}%
\global\long\def\CC{\mathbb{C}}%

\global\long\def\FF{\mathbb{F}}%
\global\long\def\XX{\mathbb{X}}%
\global\long\def\YY{\mathbb{Y}}%
\global\long\def\ZZ{\mathbb{Z}}%
\global\long\def\AA{\mathbb{A}}%
\global\long\def\DD{\mathbb{D}}%
\global\long\def\II{\mathbb{I}}%

\global\long\def\Ff{\mathcal{F}}%
\global\long\def\Hh{\mathcal{H}}%
\global\long\def\Vv{\mathcal{V}}%
\global\long\def\Cc{\mathcal{C}}%
\global\long\def\Ii{\mathcal{I}}%
\global\long\def\Nn{\mathcal{N}}%
\global\long\def\Uu{\mathcal{U}}%
\global\long\def\Ww{\mathcal{W}}%
\global\long\def\Aa{\mathcal{A}}%

\global\long\def\Ll{\mathcal{L}}%
\global\long\def\Pp{\mathcal{P}}%
\global\long\def\Oo{\mathcal{O}}%
\global\long\def\Mm{\mathcal{M}}%
\global\long\def\Gg{\mathcal{G}}%
\global\long\def\Ss{\mathcal{S}}%
\global\long\def\Ee{\mathcal{E}}%
\global\long\def\Kk{\mathcal{K}}%
\global\long\def\Rr{\mathcal{R}}%

\global\long\def\pP{\mathscr{P}}%
\global\long\def\lL{\mathscr{L}}%
\global\long\def\bB{\mathscr{B}}%
\global\long\def\dD{\mathscr{D}}%
\global\long\def\rR{\mathscr{R}}%

\global\long\def\RrR{\mathfrak{R}}%
\global\long\def\PpP{\mathfrak{P}}%
\global\long\def\XxX{\mathfrak{X}}%
\global\long\def\YyY{\mathfrak{Y}}%
\global\long\def\ZzZ{\mathfrak{Z}}%
\global\long\def\ddD{\mathfrak{d}}%
\global\long\def\mmM{\mathfrak{m}}%
\global\long\def\nnN{\mathfrak{n}}%
\global\long\def\TtT{\mathfrak{T}}%
\global\long\def\CcC{\mathfrak{C}}%
\global\long\def\EeE{\mathfrak{E}}%
\global\long\def\AaA{\mathfrak{A}}%
\global\long\def\BbB{\mathfrak{B}}%
\global\long\def\VvV{\mathfrak{V}}%

\global\long\def\iiI{\mathfrak{i}}%
\global\long\def\jjJ{\mathfrak{j}}%
\global\long\def\kkK{\mathfrak{k}}%
\global\long\def\llL{\mathfrak{l}}%
\global\long\def\mmM{\mathfrak{m}}%
\global\long\def\nnN{\mathfrak{n}}%

\global\long\def\D{\mathrm{d}}%

\global\long\def\E{\mathrm{e}}%

\global\long\def\Reals{\RR}%
\global\long\def\RealsNN{\RR_{\geq0}}%
\global\long\def\RealsP{\RR_{>0}}%

\global\long\def\deq{\coloneqq}%
\global\long\def\eqd{\eqqcolon}%
\global\long\def\dom{\mathop{\mathrm{dom}}}%
\global\long\def\argmax{\mathop{\mathrm{arg\,max}}}%
\global\long\def\argmin{\mathop{\mathrm{arg\,min}}}%
\global\long\def\supp{\mathop{\mathrm{supp}}}%

\global\long\def\transp{\top}%
\global\long\def\indic{\mathbb{\mathrm{1}}}%

\global\long\def\trace{\mathop{\mathrm{Tr}}}%
\global\long\def\CVaR{\mathop{\mathrm{CVaR}}}%

\global\long\def\DF{\mathrm{D}}%
\global\long\def\SD{\text{\ensuremath{\partial}}}%

\global\long\def\colplus{\oplus}%
\global\long\def\rowplus{\oplus}%
\global\long\def\diag{\mathop{\mathrm{diag}}}%

\global\long\def\cddot{\mathop{\cdot\cdot}}%

\global\long\def\deltaX{\hat{X}}%
\global\long\def\deltaY{\hat{Y}}%
\global\long\def\deltaZ{\hat{Z}}%
\global\long\def\deltaU{\hat{U}}%
\global\long\def\deltaB{\hat{B}}%
\global\long\def\deltaSigma{\hat{\Sigma}}%
\global\long\def\deltaF{\hat{F}}%
\global\long\def\deltaPhi{\hat{\Phi}}%

\global\long\def\deltaXast{\deltaX^{\ast}}%
\global\long\def\deltaYast{\deltaY^{\ast}}%
\global\long\def\deltaZast{\deltaZ^{\ast}}%
\global\long\def\deltaUast{\deltaU^{\ast}}%

\global\long\def\interior{\mathop{\mathrm{int}}}%
\global\long\def\sgn{\mathop{\mathrm{sgn}}}%
\global\long\def\cvar{\mathbb{ES}}%
\global\long\def\var{\mathrm{VaR}}%
\global\long\def\as{\text{-a.s.}}%

\global\long\def\warrow{\stackrel{\text{w}}{\longrightarrow}}%
\global\long\def\wsarrow{\stackrel{\text{w}^{\ast}}{\longrightarrow}}%

\global\long\def\cl{\mathop{\mathrm{cl}}}%
\global\long\def\diam{\mathop{\mathrm{diam}}}%
\global\long\def\ba{\mathrm{ba}}%

\global\long\def\rhoMD{\rho^{\mathrm{MD}}}%
\global\long\def\rhoMSD{\rho^{\mathrm{MD+}}}%
\global\long\def\rhoEpsMSD{\rho_{\epsilon}^{\mathrm{MD+}}}%
\global\long\def\rhoEntr{\rho^{\mathrm{Ent}}}%
\global\long\def\pbar{\bar{p}}%
\global\long\def\Pesssup{\PP\text{-}\mathop{\mathrm{ess\,sup}}}%
\global\long\def\qbar{\bar{q}}%
\global\long\def\phat{\hat{p}}%
\global\long\def\ptilde{\tilde{p}}%
\global\long\def\qtilde{\tilde{q}}%
\global\long\def\pzap{\mathring{p}}%
\global\long\def\qzap{\mathring{q}}%
\global\long\def\Snorm#1#2{\left\Vert #1\right\Vert _{\Ss_{\Ff}^{#2}}}%
\global\long\def\Hnorm#1#2{\left\Vert #1\right\Vert _{\Hh_{\Ff}^{#2}}}%
\global\long\def\Lnorm#1#2{\left\Vert #1\right\Vert _{#2}}%
\global\long\def\Lip{\mathrm{Lip}}%

\section{Introduction}

We consider the problem of optimal control of stochastic differential
equations of the form
\begin{gather}
x_{t}=\xi+\int_{0}^{t}b(s,x_{s},a_{s})\,\D s+\int_{0}^{t}\sigma(s,x_{s},a_{s})\,\D w_{s},\label{eq:intro_sde}
\end{gather}
over a finite time horizon, $t\in[0,T]\eqd\TT$, $0<T<\infty$, and
where $\xi$ is a random initial value, $x=(x_{t})_{t\in\TT}$ and
$a=(a_{t})_{t\in\TT}$ are the state and control processes, respectively,
taking values on spaces $\XX\deq\RR^{d_{x}}$ and $\AA\subset\RR^{d_{a}}$,
$d_{x},d_{a}\in\NN\deq\{1,2,\ldots\}$. The process $w=(w_{t})_{t\in\TT}$
is a standard $d_{w}$-dimensional Brownian motion, $d_{w}\in\NN$,
and $b$ and $\sigma$ are deterministic functions $b:\TT\times\XX\times\AA\to\RR^{d_{x}}$,
$\sigma:\TT\times\XX\times\AA\to\Reals^{d_{x}\times d_{w}}$.

Our focus here is on the problem of \emph{risk aware} control of the
diffusion process. The conventional optimal control theory of stochastic
processes considers risk neutral problems, understood here as the
minimization of expected costs accrued over the solution time interval,
\begin{align*}
\inf_{a=(a_{t})_{t\in\TT}} & \EE\biggl[\int_{0}^{T}c(t,x_{t}^{a},a_{t})\,\D t+g(x_{T}^{a})\biggr],
\end{align*}
where $c:\TT\times\XX\times\AA\to\RR$ is a cost rate function, and
$g:\XX\to\RR$ is a terminal cost function. In the \emph{risk aware}
control problems we consider here, the expectation in the objective
is supplanted by a risk function $\rho$ that describes controller's
preferences that are not sufficiently modeled by the expected value.
Formally, the risk aware problem is stated as
\begin{align}
\inf_{a=(a_{t})_{t\in\TT}} & \rho\biggl(\int_{0}^{T}c(t,x_{t}^{a},a_{t})\,\D t+g(x_{T}^{a})\biggr),\label{eq:intro_obj}
\end{align}
where we suppose that the risk function $\rho$ is some generic mapping
from random variables, representing total costs, to real values quantifying
the magnitude of the risk associated with a given random variable.
Convex or coherent risk measures form an important subset of the functions
$\rho$ that our results attempt to cover \cite{Artzner99,Follmer02,Frittelli02}. 

In the discrete time case, dynamic programming formulations of the
risk aware problem have proved elusive. This is intuitively unsurprising,
as the construction of the Bellman equation hinges on the linearity
of the expectation. Naturally this issue persists also in the continuous
time context. The continuous time setting, however, affords an alternative
to dynamic programming in the form of probabilistic formulations of
the control problem\footnote{To be clear, there appears to be a non-zero, though very small number
of works relating to probabilistic methods for \emph{discrete time}
stochastic optimal control; the only example that the present authors
are aware of is \cite{Blot2009}.}. Whereas in the dynamic programming world the control problem is
stated in terms of partial differential equations \cite{Lions83I,Lions83II,Soner06},
probabilistic formulations characterize the optimal controls in terms
of solutions to stochastic differential equations \cite{Yong1999}.
In this work, we specifically focus on the stochastic Pontryagin's
minimum principle and its generalization to the risk aware case\footnote{Throughout we use the term \emph{minimum principle}, since we phrase
our control problem as the minimization of costs. The term \emph{maximum
principle}, commonly used in the literature, should be seen as an
essentially synonymous term that is more appropriate when the problem
is stated as a maximization of rewards.}.

The risk neutral stochastic minimum principle, simply stated, asserts
that an optimal control minimizes, almost surely and at almost every
point in time, an appropriately defined Hamiltonian function that
in turn depends on adjoint processes satisfying a backward stochastic
differential equation. These necessary conditions for optimality derive
from variational equations describing the response of the cost functional
at the optimal control to an infinitesimal change in control. This
local nature of the minimum principle also provides a heuristic, \emph{a
priori} justification for preferring it over dynamic programming in
risk aware problems: Bellman's principle of optimality underlying
the dynamic programming method is a statement about the structure
of the objective of the control problem that relies on the linearity
and the tower property of the expectation. The minimum principle on
the other hand relates the optimal controls to the local behavior
of state space trajectories and the cost functions. As such, the minimum
principle does not impose requirements, here linearity, on the structure
of the risk function in the same way as dynamic programming does.
Instead, central to deriving a risk aware minimum principle is being
able to evaluate the response of a risk function to changes in its
input random variables.

\paragraph*{Literature review}

The stochastic minimum principle has a long history. Its early derivations
can be found in the works \cite{Kushner1972,Bismut1973,Bismut1978,Bensoussan1982},
with the modern version often being attributed to \cite{Peng1990}.
These results have spawned numerous refinements. Here we mention extensions
to probability measure valued controls given in \cite{Bahlali2002,Bahlali2006,Bahlali2008},
as this type of a control framework used in this paper. Generalizations
of the minimum principle to optimal control of continuous time partially
observed processes, a topic closely related to risk aware optimization,
have also been constructed \cite{Tang1998,Oksendal2007,Ahmed2013}.
The minimum principle has proven to be a viable alternative to dynamic
programming e.g. when the controls might not be Markov, in the sense
that they cannot be expressed as functions of the state variables
at any given point in time. This is the case for McKean-Vlasov problems,
for which probabilistic methods appear particularly well adapted \cite{Carmona2015_FBSDE_McKV,Carmona2018_I,Carmona2018_II}.
The minimum principle is extensively covered in \cite{Yong1999},
with numerous additional references. Dynamic programming and the minimum
principle are considered in parallel in \cite{Zhou1991}, and a comprehensive
review of the two methods can be found in \cite[Chapter 5]{Yong1999}.
For applications of the minimum principle, and backward stochastic
differential equations, we refer the reader to \cite{ElKaroui1997,Pham2009}.

Much of the recent work on the topic of control under uncertainty,
broadly understood as random variability not accounted for by an expectation
under full observations, has been done using dynamic risk measures
\cite{Acciaio2011} or nonlinear expectations such as Peng's $g$-expectation
\cite{Peng1997,Peng2004_NonlinE} and its generalization, the $G$-framework
\cite{Peng2008,Peng2010}. Compared to static risk functions, these
approaches impose additional structure, most notably time-consistency
that allows for the use of e.g. the dynamic programming principle.
While it is well-known that $g$-expectations give rise to convex
risk functions, the converse is generally true only for risk functions
that are time-consistent \cite{RosazzaGianin2006}. In our approach,
we consider objectives that are given in terms of static, law invariant
risk measures, and in particular we do not impose time-consistency
on the risk function. Moreover, since the risk function is not expressed
as a $g$-expectation, we do not need to consider forward-backward
stochastic differential equation as the starting point, as was done
in e.g. \cite{Oksendal2009_FBSDE} where a minimum principle was derived
for stochastic differential equations driven by Lévy processes. Optimality
conditions using a similar variational approach were given for forward-backward
differential equations in \cite{Yong2010}. Dynamic risk measures
were used in \cite{Barrieu2004_FBSDE}, where specifically the problem
of optimal derivatives design was considered. A dynamic programming
formulation for the $G$-framework has been developed in \cite{Hu2014,Hu2017}. 

Finally, we note that in addition to the probabilistic and dynamic
programming approaches, convex analytic and linear programming techniques
form a third, loose set of methods for both risk neutral and risk
aware control. For the risk neutral case, we refer the reader to \cite{Stockbridge90a,Stockbridge90b}
for early development and \cite{Bhatt96,Kurtz98,Kurtz01} for refinements.
A risk aware version has been developed in \cite{Isohatala2018_subm},
where a state space augmentation scheme, inherited from earlier discrete
time results \cite{Haskell15}, was used to construct a formulation
of the risk aware problem. While this approach leads to a tractable
computational method for solving the control problem, it does not
provide a useful characterization of the optimal control in the way
that the minimum principle does.

\paragraph*{Contributions and organization of the paper}

The contributions of this paper can be summed up as follows: (\emph{i})
We generalize the control problem informally stated in Eqs.~(\ref{eq:intro_sde},~\ref{eq:intro_obj})
to feature measure valued control processes. Albeit the control model
and the notion of a solution we utilize has been considered by some
authors under the name of relaxed controls, we opt for a new term
of vague controls. We justify the nomenclature by demonstrating key
differences between relaxed and vague controls, and further show why
the latter notion of a solution can be particularly useful. (\emph{ii})
We introduce law invariant risk functions into a framework that allows
a natural notion of functional differentiability that can subsequently
be applied in deriving variational conditions for optimality of controls.
(\emph{iii}) Using these results, we formulate and prove a risk aware
generalization of the stochastic Pontryagin's minimum principle, and
in doing so, we give a characterization of the optimal control of
a risk aware problem. We find that in comparison to the risk neutral
problem, the minimum principle is modified by a risk adjustment process
that is related to the functional derivative of the risk function,
evaluated at the terminal cost. Finally, (\emph{iv}), we demonstrate
by means of solving a simple example that in financial applications,
risk awareness creates non-trivial but intuitive risk pricing effects.

In the next section, we will describe the notations used in the paper,
and state the control problem we consider. Section~\ref{sec:risk-functions}
describes the risk functions that model the risk aware objectives.
We outline some necessary differentiability properties of the functions
that will subsequently be needed for the probabilistic formulation
of the problem that is given in the following Section~\ref{sec:probform}.
This section derives necessary and sufficient conditions for the optimality
of a control process. We present an application of the theory in Section~\ref{sec:exprob}
where we characterize the optimal controls of a simple portfolio allocation
problem. Section~\ref{sec:concl} concludes with discussion and some
remarks. Technical proofs are deferred to Appendix~\ref{sec:proofs}.

\section{\label{sec:model}Model}

Throughout the paper we will use the following notations and definitions:
For any probability space $(\Omega,\Sigma,\PP)$, a Banach space $(\VV,|\cdot|)$
and $p\geq0$, we denote $\Ll^{p}(\Omega,\Sigma,\PP;\VV)$, or $\Ll^{p}(\Omega;\VV)$
for short, as the set of random variables $q:\Omega\to\VV$ such that
$\EE_{\PP}[|q|{}^{p}]<\infty$, where $\EE_{\PP}$ stands for the
expectation with respect to the measure $\PP$. If $\PP$ is clear
from the context, we simply use the symbol $\EE$. In addition, $\Ll^{\infty}(\Omega;\VV)$
denotes the space of $\PP$-essentially bounded random variables.
We shall use $\Lnorm{\cdot}p$ to denote the norm on $\Ll^{p}(\Omega;\VV)$,
$p\in[1,\infty]$. For a real Banach space $\VV$, we use $\VV^{\ast}$
to denote its continuous dual, and $\langle\cdot,\cdot\rangle:\VV^{\ast}\times\VV\to\RR$
for the duality pairing.

Borel probability measures on a topological space $\VV$ are denoted
by $\Pp(\VV)$, and the Borel $\sigma$-algebra on $\VV$ is denoted
by $\bB(\VV)$. The Dirac measure centered at $x\in\VV$ is denoted
by $\delta_{x}$. By $\Pp^{p}(\VV)$, $p\in[1,\infty)$, we mean probability
measures $\mu\in\Pp(\VV)$ such that $\int d(v,v_{0}){}^{p}\mu(\D v)<\infty$
for all $v_{0}\in\VV$; $\Pp^{\infty}(\VV)$ denotes probability measures
with bounded support. The law or distribution of a random variable
$V\in\Ll^{p}(\Omega;\VV)$, $p\in[1,\infty]$, is denoted by $\lL_{\PP}(V)$,
that is, $\lL_{\PP}(V)(\Gamma)\deq\PP\circ V^{-1}(\Gamma)$ for all
$\Gamma\in\bB(\Reals)$; if the probability measure is clear from
the context, we use the symbol $\lL$ instead. The extended reals
will be denoted $\RR_{\infty}\deq\RR\cup\{\infty\}$ and elements
of $\RR^{n}$, $n\in\NN$, are by default interpreted as column vectors,
i.e. $\RR^{n}\deq\RR^{n\times1}$.

For a given filtered probability space $(\Omega,\Sigma,\Ff,\PP)$,
a normed space $(\VV,|\cdot|)$ and a $p\in[1,\infty)$, we shall
use $\Ss_{\Ff}^{p}(\Omega,\Sigma,\PP;\VV)$ or $\Ss_{\Ff}^{p}(\Omega;\VV)$
for short to denote $\VV$-valued $\Ff$-predictable continuous processes
on $\TT$ such that
\begin{gather*}
\Snorm xp\deq\EE\biggl[\sup_{t\in\TT}|x_{t}|^{p}\biggr]^{1/p}<\infty\quad\forall x\in\Ss_{\Ff}^{p}(\Omega;\VV).
\end{gather*}
In addition, $\Hh_{\Ff}^{p}(\Omega,\Sigma,\PP;\VV)=\Hh_{\Ff}^{p}(\Omega;\VV)$
denotes the space of $\Ff$-predictable processes on $(0,T)$ such
that
\begin{gather*}
\Hnorm zp\deq\EE\biggl[\biggl(\int_{0}^{T}|z_{t}|^{2}\,\D t\biggr)^{p/2}\biggr]^{1/p}<\infty\quad\forall z\in\Hh_{\Ff}^{p}(\Omega;\VV).
\end{gather*}
Two processes $z,z^{\prime}\in\Hh_{\Ff}^{p}(\Omega;\VV)$ are considered
equivalent if $\Hnorm{z-z^{\prime}}p=0$. Finally, we set $\Ss_{\Ff}^{\infty}(\Omega;\VV)\deq\cap_{p\in[1,\infty)}\Ss_{\Ff}^{p}(\Omega;\VV)$
and $\Hh_{\Ff}^{\infty}(\Omega;\VV)\deq\cap_{p\in[1,\infty)}\Hh_{\Ff}^{p}(\Omega;\VV)$.

Continuous functions from a topological space $\VV$ to a normed space
$\UU$ are denoted $\Cc(\VV,\UU)$, and we equip this space with the
usual supremum norm. If $\UU=\RR$, we abbreviate this by $\Cc(\VV)$.
The subspaces of bounded and compactly supported functions are denoted
$\Cc_{b}(\VV)$ and $\Cc_{c}(\VV)$, respectively. Superscripted function
spaces $\Cc^{(k)}(\RR^{n})$, $\Cc_{b}^{(k)}(\RR^{n})$, etc., $n\in\NN$,
denote spaces of $k\in\NN$ times continuously differentiable functions
with derivatives respectively in $\Cc$, $\Cc_{b}$, etc. For every
differentiable function $f:\RR^{n}\to\RR^{k}$, $n,k\in\NN$, the
Jacobian of $f$ is denoted $\nabla f$, so that $\nabla f\in\RR^{n}\to\RR^{k\times n}$
and $(\nabla f(x))_{ij}\deq\partial f_{i}(x)/\partial x_{j}$ for
all $i\in\{1,\ldots,k\}$, $j\in\{1,\ldots,n\}$; in particular, the
gradient of a real-valued function is a row vector. For multivariate
functions, we use $\nabla_{\UU}$ to indicate that the derivative
is taken with respect to the argument taking values in the space $\UU$.
For convenience, for all $A\in\RR^{n\times m}$ and $B\in\RR^{n\times\ell}$,
$n,m,\ell\in\NN$, we denote $A\cdot B\deq(A^{\transp}B)^{\transp}=B^{\transp}A\in\RR^{\ell\times n}$,
where $(\cdot)^{\transp}$ stands for the transpose.

We generalize Eq.~(\ref{eq:intro_sde}) to feature measure valued
controls. Instead of an adapted stochastic process $(a_{t})_{t\in\TT}$
taking values on an action space $\AA\subset\RR^{d_{a}}$, the controls
shall here in general be probability measure valued processes $(\pi_{t})_{t\in\TT}$,
$\pi_{t}\in\Pp(\AA)$. We introduce the notion of vague controls (throughout
$\XX=\RR^{d_{x}}$ and $\AA\subset\RR^{d_{a}}$ shall be our given
state and action spaces, however, we will also consider solutions
on extended state spaces, and hence the definitions below should be
understood to hold for any analogously defined finite dimensional
state and action spaces).
\begin{defn}
\label{def:vagconsol}(\emph{Vague controlled solution}) Let $\XX\deq\RR^{d_{x}}$,
$\WW\deq\RR^{d_{w}}$, and $\AA\subset\RR^{d_{a}}$, and let $b:\TT\times\XX\times\AA\to\XX$,
$\sigma:\TT\times\XX\times\AA\to\XX\times\WW=\RR^{d_{x}\times d_{w}}$
be given drift and diffusion functions that are continuous on $\TT\times\XX$
and measurable on $\AA$. A\emph{ vague controlled solution to the
problem $(b,\sigma,\nu)$} comprises a filtered probability space
$(\Omega,\Sigma,\Ff=(\Ff_{t})_{t\in\TT},\PP)$ and a process $(x_{t},w_{t},\pi_{t})_{t\in\TT}$,
$x_{t}\in\XX$, $w_{t}\in\WW$, $\pi_{t}\in\Pp(\AA)$ for all $t\in\TT$,
such that: (\emph{i}) the filtration $\Ff$ is complete and right-continuous,
(\emph{ii}) $(x_{t})_{t\in\TT}$ is $\Ff$-adapted with continuous
sample paths, $(w_{t})_{t\in\TT}$ is an $\Ff$-Brownian motion, and
$(\pi_{t})_{t\in\TT}$ is $\Ff$-progressively measurable, (\emph{iii})
the distribution of $x_{0}$ is $\nu$, and (\emph{iv}) the processes
satisfy, $\PP$-almost surely,
\begin{gather}
x_{t}=x_{0}+\int_{0}^{t}\int_{\AA}b(s,x_{s},a)\pi_{s}(\D a)\,\D s+\int_{0}^{t}\int_{\AA}\sigma(s,x_{s},a)\pi_{s}(\D a)\,\D w_{s}\qquad\forall t\in\TT.\label{eq:vague-sde}
\end{gather}
Moreover, we call the solution \emph{$r$-admissible}, $r\in[1,\infty]$,
if we additionally have that (\emph{v}) 
\begin{gather}
\begin{gathered}\EE\biggl[\sup_{t\in\TT}\int_{\AA}|a|^{r}\pi_{t}(\D a)\biggr]<\infty,\quad r<\infty,\\
\text{or}\,\text{\ensuremath{\supp\pi_{t}} is compact}\quad\forall t\in\TT,\quad r=\infty.
\end{gathered}
\label{eq:admissible}
\end{gather}
We shall use $\VvV(b,\sigma,\nu)$ to denote vague controlled solutions
of the problem $(b,\sigma,\nu)$.
\end{defn}

\begin{defn}
A \emph{strict controlled solution} is a vague controlled solution
$\pi\in\VvV(b,\sigma,\nu)$ such that $\pi_{t}$ is a Dirac measure
for all $t\in\TT$.
\end{defn}

For brevity, we write $\pi\in\VvV(b,\sigma,\nu)$ to refer to a vague
controlled solution, but it is important to bear in mind that the
solutions are in fact $(\Omega,\Sigma,\Ff,\PP,(x_{t})_{t\in\TT},(w_{t})_{t\in\TT},\pi=(\pi_{t})_{t\in\TT})$-tuples.
If necessary, we label the state process by the control, i.e. write
$(x_{t}^{\pi})_{t\in\TT}$. Clearly, strict controlled solutions can
be identified with controlled solutions where the control process
takes values on $\AA$ rather than $\Pp(\AA)$.

Vague controlled solutions are in the current literature frequently
referred to as relaxed controlled solutions, however, these two concepts
differ in some key aspects. In fact, up to the knowledge of the authors,
vague controlled solutions have never been called anything else but
relaxed controls, and the differences between the definitions are
not always explicitly noted. Examples of works where vague controls
are used include \cite{Ma1995,Bahlali2008,Andersson2010,Ahmed2013}.
We note, as \cite{Andersson2010}, that vague controlled stochastic
differential equations can be related to controlled stochastic processes
driven by non-orthogonal martingale measures, whereas the more canonical
relaxed controlled model can be identified with equations driven by
orthogonal martingale measures \cite{ElKaroui1990}. See e.g. \cite{Walsh1986}
for more discussion on martingale measures. It also bears pointing
out that the topology conventionally assigned to relaxed controls,
see e.g. \cite{Fleming1984}, may be too coarse for vague controlled
problems to guarantee the continuity of the mapping from controls
to the stochastic trajectories, which has implications for e.g. applying
the chattering lemma \cite[Theorem 2.2]{ElKaroui1988} to vague controls.
Indeed, as Example~\ref{exa:std-counter} below demonstrates, it
may not always be possible to find strict controls and associated
solutions of Eq.~(\ref{eq:vague-sde}) that approximate a given vague
controlled solution.

To elucidate the difference between these notions of solutions, consider
$\pi\in\VvV(b,\sigma,\nu)$, where $b,\sigma$ are for simplicity
taken to be bounded. Applying Itô's lemma to $f(x_{t}^{\pi})$, $f\in\Cc_{c}^{(2)}(\XX)$,
we have that the process $(v_{t}^{f})_{t\in\TT}$, defined 
\begin{align}
v_{t}^{f}\deq f(x_{t}^{\pi})-f(x_{0}^{\pi}) & -\int_{0}^{t}\biggl\{\nabla f(x_{s}^{\pi})\int_{\AA}b(s,x_{s}^{\pi},a)\pi_{s}(\D a)\nonumber \\
 & \qquad+\frac{1}{2}\trace\left[\nabla^{\transp}\nabla f(x_{s})\left(\int_{\AA}\sigma(s,x_{s}^{\pi},a)\pi_{s}(\D a)\right)\left(\int_{\AA}\sigma(s,x_{s}^{\pi},a)\pi_{s}(\D a)\right)^{\transp}\right]\biggr\}\,\D s\qquad\forall t\in\TT\label{eq:Vrel}
\end{align}
is a martingale for any $f\in\Cc_{c}^{(2)}(\XX)$. A \emph{relaxed
controlled solution} corresponding to the drift and diffusion functions
$b,\sigma$ is conventionally defined as a filtered probability space
together with a stochastic process $(\xi_{t}^{\eta},\eta_{t})_{t\in\TT}$,
$\xi_{t}^{\eta}\in\XX$, $\eta_{t}\in\Pp(\AA)$ for all $t\in\TT$,
satisfying items (\emph{i}-\emph{iii}) of Definition~\ref{def:vagconsol},
but characterized by the condition that the processes $(m_{t}^{f})_{t\in\TT}$,
\begin{align}
m_{t}^{f}\deq f(\xi_{t}^{\eta})-f(\xi_{0}^{\eta}) & -\int_{0}^{t}\int_{\AA}\biggl\{\nabla f(\xi_{s}^{\eta})b(s,\xi_{s}^{\eta},a)\nonumber \\
 & \qquad+\frac{1}{2}\trace\left[\nabla^{\transp}\nabla f(\xi_{s}^{\eta})\sigma(s,\xi_{s}^{\eta},a)\sigma(s,\xi_{s}^{\eta},a)^{\transp}\right]\biggr\}\eta_{s}(\D a)\,\D s\label{eq:Mrel}
\end{align}
are martingales for all $f\in\Cc_{c}^{(2)}(\XX)$. Comparing Eqs.~(\ref{eq:Mrel})
and~(\ref{eq:Vrel}) we see that the order of integration against
the control and squaring the diffusion coefficient are interchanged.
Thus, informally, a vague controlled solution corresponds to processes
where for any $t\in\TT$, the amplitude of the noise is the $\pi_{t}$-average
of $a\to\sigma(t,x_{t}^{\pi},a)$, whereas in the relaxed controlled
case, the noise is the $\eta_{t}$-root mean square of the diffusion
function.

It is well-known that a strict optimal control may fail to exist while
one can always be found within the set of relaxed controls. This is
due to the convexity of the space of probability measures, a property
that has in the past often been exploited in optimal control of stochastic
differential equations \cite{ElKaroui1987,Borkar1988,Haussmann1990}.
An additional motivation for considering generalizations of strict
controls comes from the fact that little is known about the nature
of the optimal control in the risk aware case. In discrete time, risk
aware formulations featuring generic risk functions in the objective
have been successfully described\footnote{These works used a somewhat inelegant state space augmentation scheme
that can here be avoided; see also Remark~\ref{rem:no-aug}.} using the convex analytic formulation \cite{Haskell15}, later expanded
to the continuous-time case as well \cite{Isohatala2018_subm}. Such
problems often, though certainly not exclusively, feature relaxed
controls as the optimal solution \cite{Guo07,Zhang08,Stockbridge12}
and it is therefore not unreasonable to expect that a generalization
of strict controls may be appropriate here as well.

We note that relaxed controlled solutions are more widely represented
in the literature than vague controlled solutions. This is in part
due to the fact that relaxed controlled solutions can be viewed as
the closure of strict controls, under a suitably defined topology
\cite{ElKaroui1987}. This is not the case for vague controls. The
following example demonstrates that there are vague controlled solutions
whose finite dimensional distributions cannot be approximated by those
of strict or relaxed controls. A similar example has been featured
earlier in \cite{Bahlali2018}.
\begin{example}
\label{exa:std-counter}Consider $\XX=\RR$, $\AA=\{-1,+1\}$ and
$b(t,x,a)=0$, $\sigma(t,x,a)=a$ for all $(t,x,a)\in\TT\times\XX\times\AA$,
and $\nu=\delta_{0}$. Then for all strict controls $\pi\in\VvV(b,\sigma,\nu)$
(in fact, for all relaxed controls as well), $x^{\pi}$ is an $\Ff$-Brownian
motion, but there exists a vague controlled solution $\pi^{\prime}\in\VvV(b,\sigma,\nu)$
such that $\pi_{t}^{\prime}=(\delta_{-1}+\delta_{+1})/2$ and $x_{t}^{\pi^{\prime}}=0$
for all $t\in\TT$. Consequently, considering e.g. a control problem
of $\inf_{\pi}\EE[(x_{T}^{\pi})^{2}]$, it is clear that a vague controlled
solution may attain a strictly lower optimum value that can be found
using strict controls.
\end{example}

Our main reason for considering vague controls is that the optimality
conditions obtained from a stochastic minimum principle are considerably
simpler and thus easier to use in practice. In the classical risk
neutral case, and when the control set $\AA$ is non-convex and the
diffusion coefficient depends on the control, first and second order
adjoint equations are needed to characterize the optimal control,
see e.g. the classic work by Peng \cite{Peng1990} and more recent
results for relaxed controls in \cite{Bahlali2006,Labed2017}. The
issues resulting from the need for second order expansions are exacerbated
in the risk aware setting, where the second order expansions will
also require us to compute second order functional derivatives of
the risk function $\rho$. For vague controls, first order expansions
turn out to be sufficient, which is also the case in risk neutral
problems considered in \cite{Bahlali2008,Ahmed2013}. We note that
the sufficiency of first order expansions is not entirely surprising,
since this is also the case for strict controls when the control set
$\AA$ is convex, see e.g. \cite{Bismut1973}. We demonstrate this
below in Example~\ref{exa:first-order}. Indeed, vague controls could
be viewed as \emph{measure valued strict controls}, in which case
the control set $\Pp(\AA)$ is naturally convex. Finally, as strict
controls are a subset of vague controls, optimality of a strict control
is readily shown by demonstrating that a vague control process necessarily
takes values on Dirac measures.
\begin{example}
\label{exa:first-order}Consider the problem of \cite[Example 4.1]{Yong1999}.
We set $\XX=\RR$, $\AA=\{0,1\}$ and $b(t,x,a)=0$, $\sigma(t,x,a)=a$
for all $(t,x,a)\in\TT\times\XX\times\AA$, $\nu=\delta_{0}$, and
consider minimizing $\EE[x_{T}^{2}]$ over strict controls. Clearly,
$(x_{t}=0,\pi_{t}=\delta_{0})_{t\in\TT}$ is optimal. In the context
of strict controls, one considers spike perturbations (cf. Eq.~(4.3)
of \cite{Yong1999}) to an optimal control to establish conditions
for optimality. In \cite[Example 4.1]{Yong1999}, it is shown that
if $x^{\epsilon}=(x_{t}^{\epsilon})_{t\in\TT}$ is a state space process
corresponding to such a perturbation, and $x=(x_{t})_{t\in\TT}$ is
optimal (here, zero), then $\sup_{t\in\TT}\EE[|x_{t}^{\epsilon}-x_{t}|^{2}]=\epsilon/2$.
However, if we consider a vague control formed by a convex combination
of the optimal control $\pi=(\pi_{t}=\delta_{0})_{t\in\TT}$ and an
arbitrary progressively measurable $q=(q_{t})_{t\in\TT}$ so that
$\pi_{t}^{\epsilon}=(1-\epsilon)\pi_{t}+\epsilon q_{t}$ for all $t\in\TT$,
we find that 
\begin{align*}
\sup_{t\in\TT}\EE\biggl[\bigl|x_{t}^{\epsilon}-x_{t}\bigr|^{2}\biggr] & \leq4T\epsilon^{2}.
\end{align*}
Therefore, perturbations to vague controls result in an $\Oo(\epsilon^{2})$
response in the state space paths (in the above sense), whereas for
strict controls, we only have $\Oo(\epsilon)$. This suggests that
computing the first order response may indeed be sufficient for establishing
necessary conditions for optimality of vague controlled solutions.
\end{example}

The natural downside to considering vague controls is, as Example~\ref{exa:std-counter}
demonstrates, that the optimal vague control may be something that
cannot be approximated by strict controls. This may be an issue in
practice, if a vague control cannot realistically be implemented.
This is in contrast to the case of usual relaxed controls, for which
it is typically possible to construct an $\epsilon$-optimal strict
control from an optimal relaxed control. 

We assume standard continuity and boundedness conditions that guarantee
the existence of solutions to Eq.~(\ref{eq:vague-sde}), given a
probability space, a Brownian motion, and a control process. We also
state conditions for the cost rate function, and hence, we first need
to introduce the total cost random variable. For every $\pi\in\VvV(b,\sigma,\nu)$
we define the total cost $C^{\pi}$ as 
\begin{gather}
C^{\pi}\deq\int_{0}^{T}\int_{\AA}c(t,x_{t}^{\pi},a)\pi_{t}(\D a)\,\D t+g(x_{T}^{\pi}),\label{eq:rel_cost}
\end{gather}
where $c:\TT\times\XX\times\AA\to\RR$ is the cost rate function and
$g:\XX\to\RR$ is the terminal cost. In the risk neutral case, it
generally suffices to ensure that $C^{\pi}\in\Ll^{1}(\Omega;\RR)$,
however here, we shall need to compute the risk of $C^{\pi}$ which
generally involves evaluating an $\Ll^{p}(\Omega;\RR)\to\RR$ functional
at $C^{\pi}$, where $p\in[1,\infty)$. In order to accommodate a
wider range of possible values of $p$, somewhat more elaborate conditions
(compared to the risk neutral case) on the bounds of $b,\sigma,c,g$
and their growth rates shall be needed. In addition, as the optimality
conditions given in Section~\ref{sec:probform} will be derived from
variational inequalities, we require that the relevant functions are
all also differentiable. Formally, our baseline assumptions are as
follows.
\begin{assumption}
\label{assu:sde-baseline}The initial distribution $\nu\in\Pp(\XX)$,
drift and diffusion functions $b:\TT\times\XX\times\AA\to\XX$, $\sigma:\TT\times\XX\times\AA\to\RR^{d_{x}\times d_{w}}$,
cost rate and terminal cost functions $c:\TT\times\XX\times\AA\to\RR$,
$g:\XX\to\RR$, and admissible control processes are such that there
are constants $L>0$, $\pbar_{1}\in[0,1]$, $\pbar_{2}\in[0,\infty)$,
$\pbar_{3}\in[0,\infty]$, $\pbar\in[1,\infty)$, $p_{1}\in[0,\infty)$,
$p_{2}\in[0,\infty)$, $p_{1}^{\prime}\in[0,\infty)$, $p_{2}^{\prime}\in[0,\infty)$
satisfying: (\emph{i}) if $\pbar_{3}=\infty$, then $\AA$ is compact;
(\emph{ii}) for all $(t,x,a)\in\TT\times\XX\times\AA$,
\begin{subequations}
\label{eq:baseline-growth-b-sigma}
\begin{gather}
\left|b(t,x,a)\right|\leq L\left(1+\left|x\right|^{\pbar_{1}}+\left|a\right|^{\pbar_{2}}\right),\label{eq:baseline-growth-b}\\
\left|\sigma(t,x,a)\right|\leq L\left(1+\left|x\right|^{\pbar_{1}}+\left|a\right|^{\pbar_{2}}\right);\label{eq:baseline-growth-sigma}
\end{gather}
\end{subequations}
(\emph{iii}) for all $(t,a)\in\TT\times\AA$, the functions $x\to b(t,x,a)$
and $x\to\sigma(t,x,a)$ are continuously differentiable, and the
derivatives are bounded by $L$; (\emph{iv}) for all $(t,x,a)\in\TT\times\XX\times\AA$,
\begin{subequations}
\label{eq:baseline-growth-costs}
\begin{gather}
\left|c(t,x,a)\right|\leq L\left(1+\left|x\right|^{p_{1}}+\left|a\right|^{p_{2}}\right),\label{eq:baseline-growth-c}\\
\left|g(x)\right|\leq L\left(1+\left|x\right|^{p_{1}}\right);\label{eq:baseline-growth-g}
\end{gather}
\end{subequations}
(\emph{v}) for all $(t,a)\in\TT\times\AA$, the functions $x\to c(t,x,a)$
and $x\to g(x)$ are continuously differentiable, and satisfy, for
all $(t,a)\in\TT\times\AA$,
\begin{subequations}
\label{eq:baseline-growth-nabla-costs}
\begin{gather}
\left|\nabla_{\XX}c(t,x,a)\right|\leq L\left(1+\left|x\right|^{p_{1}^{\prime}}+\left|a\right|^{p_{2}^{\prime}}\right),\label{eq:baseline-growth-nabla-c}\\
\left|\nabla_{\XX}g(x)\right|\leq L\left(1+\left|x\right|^{p_{1}^{\prime}}\right);\label{eq:baseline-growth-nabla-g}
\end{gather}
\end{subequations}
(\emph{vi}) the initial distribution $\nu\in\Pp^{\bar{p}}(\XX)$;
(\emph{vii}) all control processes are $\pbar_{3}$-admissible, i.e.
satisfy Eq.~(\ref{eq:admissible}) for $r=\pbar_{3}$.
\end{assumption}

\begin{defn}
Let $p\in[1,\infty)$. We say that a vague controlled solution $\pi\in\VvV(b,\sigma,\nu)$
is \emph{$p$-feasible} and denote $\pi\in\VvV^{p}(b,\sigma,\nu)$,
if there exists $\pbar$, $\pbar_{i}$, $i\in\{1,2,3\}$, $p_{i}$,
$p_{i}^{\prime}$, $i\in\{1,2\}$ satisfying Assumption~\ref{assu:sde-baseline}
and the following inequalities:
\begin{subequations}
\label{eq:feasible-p}
\begin{gather}
p<\pbar\leq\pbar_{3},\label{eq:p-less-pbar}\\
\pbar_{2}\leq\frac{\pbar_{3}}{\pbar},\label{eq:pbar-bounds}\\
p_{1}^{\prime}\leq p_{1},\quad p_{2}^{\prime}\leq p_{2},\label{eq:growth-deriv-relation}\\
p_{1},p_{2}<\frac{\pbar}{p}-1.\label{eq:main-deriv-bounds}
\end{gather}
\end{subequations}
\end{defn}

To give an intuition on the meanings and uses of these constants (Proposition~\ref{prop:basic-existence-uniqueness}
below gives a more formal statement), $p$ and $\pbar$ shall respectively
represent the order up to which the costs $C^{\pi}$ and the state
space variables $x_{t}^{\pi}$, $t\in\TT$, are integrable. We allow
for unbounded cost rates and terminal costs, in fact even superlinear
growth is admissible ($p_{1},p_{2}>1$ in Eq.~(\ref{eq:baseline-growth-costs})),
but in order to guarantee that costs are in $\Ll^{p}(\Omega;\RR)$,
bounds on the integrability of the state and action variables need
to be imposed. Eqs.~(\ref{eq:feasible-p}) amount to sufficient conditions
for such integrability to hold. 

We expect a typical use-case of our results is such where one is given
drift and diffusion functions $b$ and $\sigma$, a cost structure
in the form of the cost rate and terminal cost functions $c$ and
$g$, and a risk function as a map $\Ll^{p}(\Omega;\RR)\to\RR$ with
a fixed $p\in[1,\infty)$, mapping from costs to risks. The growth
rates of these functions dictate the values of $\pbar_{1}$, $\pbar_{2}$,
$p_{1}$, $p_{2}$, $p_{1}^{\prime}$, and $p_{2}^{\prime}$. The
feasibility conditions of Eqs.~(\ref{eq:feasible-p}) can then be
understood as determining admissible $\pbar$ and $\pbar_{3}$, representing
the level of randomness in the initial condition and the range of
control values that yield $\Ll^{p}(\Omega;\RR)$-finite costs. 

Given a filtered probability space with a Brownian motion and a progressively
measurable $\Pp(\AA)$-valued control process, stochastic differential
equations satisfying the Assumptions~\ref{assu:sde-baseline} and
inequalities of Eqs.~(\ref{eq:feasible-p}) for a given $p\in[1,\infty)$
have strong solutions. Together, these comprise a $p$-feasible vague
controlled solution, and moreover, if $\pi\in\VvV^{p}(b,\sigma,\nu)$,
then the costs $C^{\pi}\in\Ll^{p}(\Omega;\RR)$.
\begin{prop}
\label{prop:basic-existence-uniqueness} Let $p\in[1,\infty)$ and
suppose Assumptions~\ref{assu:sde-baseline} and Eq.~(\ref{eq:feasible-p})
hold. Let $(\Omega,\Sigma,\Ff=(\Ff_{t})_{t\in\TT},\PP)$ be a filtered
probability space, $\Ff$ a complete and right continuous filtration,
and $(w_{t})_{t\in\TT}$ a $d_{w}$-dimensional $\Ff$-Brownian motion.
(\emph{i}) If $\pi=(\pi_{t})_{t\in\TT}$ is an $\Ff$-progressively
measurable $\Pp(\AA)$-valued stochastic process that satisfies Eq.~(\ref{eq:admissible})
for $\pbar_{3}$, then there exists a pathwise unique solution $x^{\pi}=(x_{t}^{\pi})_{t\in\TT}$
to the stochastic differential equation~(\ref{eq:vague-sde}) such
that 
\begin{gather}
x^{\pi}\in\Ss_{\Ff}^{\pbar}(\Omega;\XX),\qquad C^{\pi}\in\Ll^{p}(\Omega;\RR),\label{eq:x-main-estimate}
\end{gather}
so that $(\Omega,\Sigma,\Ff,\PP,\pi,x^{\pi})$ is in $\VvV^{p}(b,\sigma,\nu)$
or in other words is a $p$-feasible vague controlled solution. (\emph{ii})
If additionally $a\to b(t,x,a)$ and $a\to\sigma(t,x,a)$ are $L$-Lipschitz
for all $(t,x)\in\TT\times\XX$, then the mapping $\Pi_{\Ff}^{\pbar_{3}}(\Omega;\AA)\ni\pi\to x^{\pi}\in\Ss_{\Ff}^{\pbar}(\Omega;\XX)$
is continuous.
\end{prop}

In order to state the risk aware control problem, we need to first
establish some basic properties of risk functions. We collate our
discussions on their properties in the next section, where we first
describe the subset of risk functions that can be used to evaluate
the risk associated with $C^{\pi}$ when $\pi\in\VvV^{p}(b,\sigma,\nu)$.

\section{\label{sec:risk-functions}Risk functions}

\paragraph*{Risk aware objective function}

Given the definition of feasible vague controlled solutions, $\pi\in\VvV^{p}(b,\sigma,\nu)$,
$p\in[1,\infty)$, and the cost functional $C^{\pi}$, the risk neutral
control problem could now be simply stated as
\begin{gather*}
\pP_{0}:\qquad\inf_{\pi\in\VvV^{1}(b,\sigma,\nu)}\EE_{\pi}[C^{\pi}],
\end{gather*}
where we have written the expectation as $\EE_{\pi}$ to highlight
the fact that the probability space, and in particular the probability
measure used to compute the expectation, is a part of the vague controlled
solution $\pi\in\VvV^{p}(b,\sigma,\nu)$. This problem statement does
not trivially generalize to the risk aware case: Here, we presume
we are given a risk function $\rho:\Ll^{p}(\Omega;\RR)\to\RR$, $p\in[1,\infty)$,
defined on some unspecified probability space $(\Omega,\Sigma,\PP)$,
mapping an $\Ll^{p}(\Omega;\RR)$ random variable to a real-valued
measure of risk that quantifies the variability associated with this
random variable. Since in general the probability space for a given
$\rho$ is fixed, we cannot use $\rho$ to evaluate the risk of $C^{\pi}$
when the probability space potentially varies with each $\pi\in\VvV^{p}(b,\sigma,\nu$). 

To remedy this issue, note first that the risk neutral problem, Problem
$\pP_{0}$, makes sense since the expectation does not depend on the
particulars of the underlying probability space, but rather only on
the distributions of the random variables. This is to say, for any
two $\Ll^{1}$-random variables $X$ and $\tilde{X}$, defined on
different probability spaces $(\Omega,\Sigma,\PP)$ and $(\tilde{\Omega},\tilde{\Sigma},\tilde{\PP})$,
we have that $\EE_{\PP}[X]=\EE_{\tilde{\PP}}[\tilde{X}]$ whenever
the laws of $X$ and $\tilde{X}$ agree. In order to generalize Problem~$\pP_{0}$
to the risk aware case, we restrict ourselves to risk functions having
this same, law invariance property:
\begin{defn}
\label{def:law-invariance}Let $(\Omega,\Sigma,\PP)$ be a probability
space. A mapping $\phi\text{ : }\Ll^{p}(\Omega;\RR)\rightarrow\RR$,
$p\in[1,\infty]$, is \emph{law invariant }if there is a $\psi:\Pp^{p}(\RR)\to\RR$
such that $\phi(U)=\psi(\lL(U))$ for all $U\in\Ll^{p}(\Omega;\RR)$.
\end{defn}

Law invariant risk functions have been extensively studied in the
literature, in particular they admit well-known and widely exploited
representation theorems \cite{Kusuoka2001,Frittelli2005,Jouini2006}.
Here however, the law invariance property allows us to state the risk
aware version of Problem~$\pP_{0}$. For any law invariant $\rho:\Ll^{p}(\Omega;\RR)\to\RR$,
$p\in[1,\infty]$, we define Problem $\pP_{1}$ as
\begin{gather*}
\pP_{1}:\qquad\inf_{\pi\in\VvV^{p}(b,\sigma,\nu)}\rho(C^{\pi}).
\end{gather*}
We adopt the view that a law invariant risk function can be equivalently
seen as a mapping from $\Ll^{p}(\Omega;\RR)$-random variables to
reals, or as a function from $\Pp^{p}(\RR)$-measures to reals. This
latter representation of risk functions has been used also in previous
works, see e.g. \cite{Haskell15,Isohatala17_smdp,Isohatala2018_subm}.
We emphasize that here, we consider the expression of risk for random
variables and measures on equal footing: While viewing $\rho$ exclusively
as a function from measures to reals is appealing in its simplicity,
in doing so we would firstly lose some convexity and coherence properties
that are better defined for $\Ll^{p}(\Omega;\RR)$-functionals (see
e.g. Definition~\ref{def:coherent} below). Indeed, it is quite often
true that if a risk function $\rho:\Ll^{p}(\Omega;\RR)\to\RR$ is
convex, then its representation as a function $\tilde{\rho}:\Pp^{p}(\RR)\to\RR$,
$\tilde{\rho}(\lL(X))=\rho(X)$ for all $X\in\Ll^{p}(\Omega;\RR)$
is \emph{concave} \cite{Acciaio2013} -- this has obvious implications
for minimization of such functionals. Secondly, we will in the following
need some notion of differentiability of risk functions. Functional
differentiation is more readily defined on the Banach spaces $\Ll^{p}(\Omega;\RR)$,
and moreover, an appropriate theory has recently been developed in
the context of mean field games and McKean-Vlasov problems \cite{Carmona2018_I,Carmona2018_II}.
In their context, differentiability with respect to distributions
was needed for treating the mean field, i.e. distribution dependent
terms in the model equations.

In the following, we will go back and forth between representations
of a risk function as a mapping over random variables or measures.
What we mean by this is formalized in the following definition:
\begin{defn}
\label{def:pp-lp-representations}Let $(\Omega,\Sigma,\PP)$ be a
probability space, $p\in[1,\infty]$, and let $\VV$ be a metric space.
(\emph{i}) Suppose $\phi\text{ : }\Ll^{p}(\Omega;\VV)\rightarrow\RR$
is a law invariant mapping. A function $\psi:\Pp^{p}(\VV)\to\RR$
is a \emph{$\Pp^{p}$-representation of $\phi$} if $\phi(U)=\psi(\lL(U))$
for all $U\in\Ll^{p}(\Omega;\VV)$. (\emph{ii}) Suppose $\psi:\Pp^{p}(\VV)\to\RR$.
If there is a probability space $(\tilde{\Omega},\tilde{\Sigma},\tilde{\PP})$
and a function $\phi\text{ : }\Ll^{p}(\tilde{\Omega};\VV)\rightarrow\RR$
such that $\psi(\lL(U))=\phi(U)$ for all $U\in\Ll^{p}(\tilde{\Omega};\VV)$,
then we say $\phi$ is an \emph{$\Ll^{p}(\tilde{\Omega})$-representation
of $\psi$}.
\end{defn}

In order to impose more structure on the set of risk functions we
consider, some of the following properties, frequently considered
in the literature \cite{Artzner99,Follmer02}, are assumed. These
properties are more naturally defined for the $\Ll^{p}(\Omega)$-representation
of the risk function.
\begin{defn}
\label{def:coherent}Let $(\Omega,\Sigma,\PP)$ be a probability space
and denote $\Ll\deq\Ll^{p}(\Omega;\RR)$, $p\in[1,\infty]$. (\emph{i})
\emph{Monotonicity}: for all $X_{1},X_{2}\in\Ll$ such that $X_{1}\leq X_{2}$
almost surely, $\rho(X_{1})\leq\rho(X_{2})$. (\emph{ii}) \emph{Convexity}:
$\rho(\alpha X_{1}+(1-\alpha)X_{2})\leq\alpha\rho(X_{1})+(1-\alpha)\rho(X_{2})$
for all $X_{1},X_{2}\in\Ll$ and $\alpha\in[0,1]$. (\emph{iii}) \emph{Positive
homogeneity}: $\rho(aX)=a\rho(X)$ for all $a\geq0$ and $X\in\Ll$.
(\emph{iv}) \emph{Translation invariance}: $\rho(X+a)=\rho(X)+a$
for all $a\in\RR$ and $X\in\Ll$. (\emph{v}) If the risk function
satisfies (\emph{i}--\emph{iv}), it is called \emph{coherent}.
\end{defn}

\paragraph*{\label{subsec:rho-diffability}Differentiability of risk functions}

We begin by recalling the following, standard definitions of functional
derivatives.
\begin{defn}
Let $\VV$ and $\UU$ be real Banach spaces. (\emph{i}) For any $f:\VV\to\RR$,
the \emph{subdifferential} of $f$ at $X\in\VV$, denoted $\SD f(X)$,
is the set $\partial f(X)=\{Y\in\VV^{\ast}:\langle Y,X^{\prime}-X\rangle\leq f(X^{\prime})-f(X),\,\forall X^{\prime}\in\VV\}$.
The function is \emph{subdifferentiable} at $X\in\VV$ if $\partial f(X)\neq\emptyset$.
(\emph{ii}) For any $F:\VV\to\UU$, the \emph{directional derivative}
of $F$ at $X\in\VV$ in direction $Y\in\VV$, denoted $\DF_{Y}F(X)$,
is defined as
\begin{gather}
\DF_{Y}F(X)\deq\frac{\D}{\D\epsilon}F(X+\epsilon Y)\biggl\vert_{\epsilon=0}=\lim_{\epsilon\to0}\frac{F(X+\epsilon Y)-F(X)}{\epsilon}\label{eq:dir-deriv}
\end{gather}
if the limit exists. Further, we will say that the function $F$ is
\emph{Gâteaux differentiable} at $X\in\VV$ if the above limit exists
for all $Y\in\VV$ and if the mapping $Y\to\DF_{Y}F(X)$ is linear.
(\emph{iii}) The function $F:\VV\to\UU$ is \emph{Fréchet differentiable}
at $X\in\VV$ if there is a continuous linear operator $\DF F(X)\in\VV^{\ast}$,
the Fréchet derivative, such that
\begin{gather*}
F(X+Y)-F(X)-\langle\DF F(X),Y\rangle\in o\left(\Vert Y\Vert_{\VV}\right),
\end{gather*}
where $\Vert\cdot\Vert_{\VV}$ denotes the norm on $\VV$.
\end{defn}

\begin{rem}
\label{rem:no-diff}A coherent risk function $\rho:\Ll^{p}(\Omega;\RR)\to\RR$
that is nonlinear cannot be everywhere Gâteaux or Fréchet differentiable.
Specifically, $\rho$ cannot be Fréchet differentiable at $X=0$,
if positive homogeneity, Definition~\ref{def:coherent}(\emph{iii}),
holds. Then $\DF_{Y}\rho(0)=\D[\rho(\epsilon Y)]/\D\epsilon|_{\epsilon=0}=\rho(Y)$,
which is not linear; this was earlier pointed out in \cite[Proposition 3.1]{Fischer2003}.
Relaxing the assumption that the Gâteaux derivative must be linear
would remove the issue, but everywhere Fréchet differentiability of
$\rho$ is nonetheless not possible. In Section~\ref{sec:exprob}
we demonstrate, by way of an example, that a risk function can be
shown to be differentiable at the cost random variable $C^{\pi}$.
We also show that e.g. the entropic risk measure, frequently encountered
in the literature, is everywhere Fréchet differentiable, and that
additionally, it may be possible to approximate a risk function with
another, everywhere Fréchet differentiable functional.
\end{rem}

We can now define a useful notion of a derivative of a law invariant
risk function. Here, we use the definition used in e.g. \cite{Carmona2015_FBSDE_McKV,Carmona2018_I,Carmona2018_II},
which we extend slightly to cover $\Ll^{p}(\Omega;\RR)$ spaces, $p\in[1,\infty]$,
not just $\Ll^{2}(\Omega;\RR)$. The definition relies on the Fréchet
differentiability of the $\Ll^{p}(\Omega)$-representation of the
function.
\begin{defn}
\label{def:diffability-wrt-measures}Let $\phi:\Pp^{p}(\RR^{n})\to\RR$,
$n\in\NN$, $p\in[1,\infty]$, and suppose there is a probability
space $(\Omega,\Sigma,\PP)$ and an $\Ll^{p}(\Omega$)-representation
of $\phi$, denoted $\psi$. (\emph{i}) We say the function $\phi$
is \emph{$\Ll$-differentiable} at $\mu_{0}\in\Pp^{p}(\Reals^{n})$
if its $\Ll^{p}(\Omega$)-representation $\psi$ is Fréchet differentiable
at any point $U_{0}\in\Ll^{p}(\Omega;\RR^{n})$ such that $\lL(U_{0})=\mu_{0}$. 

(\emph{ii}) The function $\phi$ is \emph{continuously $\Ll$-differentiable},
if the Fréchet derivative of $\psi$ as seen as a function $\Ll^{p}(\Omega;\RR^{n})\ni X\to\DF\psi(X)\in\Ll^{q}(\Omega;\RR^{n})$,
$q=p/(p-1)$, is continuous. 

(\emph{iii}) Given $\mu\in\Pp^{p}(\RR^{n})$, we say the function
$f:\RR^{n}\to\RR^{1\times n}$ is an \emph{$\Ll$-derivative }of $\phi$
at $\mu$, if the Fréchet derivative of $\psi$, $\DF\psi(X)\in\Ll^{q}(\Omega;\RR^{n})$
is such that $\DF\psi(X)(\omega)=f(X(\omega))$ for $\PP$-almost
all $\omega\in\Omega$, implying that $\langle\DF\psi(X),Y\rangle=\langle f(X),Y\rangle$
for all $X,Y\in\Ll^{p}(\Omega;\RR)$ such that $\lL(X)=\mu$. We will
denote a representative $\Ll$-derivative by $\DF\phi(\mu)$.
\end{defn}

We have the following result concerning the existence of $\Ll$-derivatives.
It demonstrates that $\Ll$-derivatives commonly exist, and are not
limited to exceptional cases of risk functions.
\begin{prop}
\label{prop:L-deriv}Suppose $\phi:\Pp^{p}(\RR^{n})\to\RR$, $n\in\NN$,
$p\in[2,\infty]$, is continuously $\Ll$-differentiable. Then an
$\Ll$-derivative exists, and is unique in the sense that if $f_{1}$
and $f_{2}$ are $\Ll$-derivatives at $\mu\in\Pp^{p}(\RR^{n})$,
then $f_{1}(x)=f_{2}(x)$ for $\mu$-almost every $x\in\RR^{n}$. 
\end{prop}

Our main use of the $\Ll$-derivative is in evaluating the first-order
response of functions of probability measures. If $\mu,\mu_{0}\in\Pp^{p}(\RR^{n})$,
then for any random variables $U,U_{0}\in\Ll^{p}(\Omega;\RR^{n})$,
$p\in[1,\infty]$, whose laws equal $\mu$ and $\mu_{0}$, respectively,
we get the following expansion directly from the definitions of the
Fréchet- and $\Ll$-derivatives:
\begin{gather}
\phi(\mu)=\phi(\mu_{0})+\EE[\DF\phi(\mu_{0})(U_{0})(U-U_{0})]+o\left(\Lnorm{U-U_{0}}p\right),\label{eq:phi-expansion}
\end{gather}
where $\EE[\DF\phi(\mu_{0})(U_{0})(U-U_{0})]=\langle\DF\phi(\mu_{0})(U_{0}),U-U_{0}\rangle$.
\begin{rem}
The need for the notion of $\Ll$-differentiability ultimately arises
from the use of vague controls and weak solutions of the stochastic
differential equations, that is, vague controlled solutions in the
sense of Definition~\ref{def:vagconsol}. If it were possible to
always consider a fixed probability space, there would not be a need
for the notion of $\Ll$-differentiability, and we could instead solely
use the Fréchet derivative on a fixed $\Ll^{p}(\Omega;\RR)$ space
to construct the first order responses of the form given in Eq.~(\ref{eq:phi-expansion}). 
\end{rem}

Differentiability can subsequently be used to construct a notion of
convexity of a real-valued function of probability measures $\mu\in\Pp^{p}(\VV)$,
where $\VV$ is a Banach space and $p\in[1,\infty]$, without needing
to impose vector space structure on $\Pp^{p}(\VV)$.
\begin{defn}
\label{def:diff-convexity}An $\Ll$-differentiable function $\phi:\Pp^{p}(\RR^{n})\to\RR_{\infty}$,
$n\in\NN$, $p\in[1,\infty]$, with the $\Ll$-derivative $\DF\phi:\Pp^{p}(\RR^{n})\times\RR^{n}\to\RR$
is \emph{$\Ll$-convex} if 
\begin{gather*}
\phi(\mu')-\phi(\mu)-\EE[\DF\phi(\mu)(U)(U'-U)]\geq0\quad\forall\mu,\mu'\in\Pp^{p}(\RR^{n})
\end{gather*}
and where $U,U'\in\Ll^{p}(\Omega;\RR^{n})$ are any random variables
over some probability space $(\Omega,\Sigma,\PP)$ such that $\lL(U)=\mu$
and $\lL(U')=\mu'$.
\end{defn}

It is straight-forward to verify that for a law invariant, $\Ll$-differentiable
risk function $\rho\in\Ll^{p}(\Omega;\RR)\to\RR$ with an $\Ll$-derivative
$\DF\rho(\cdot)(\cdot):\Pp^{p}(\RR)\times\RR\to\RR$, convexity in
the sense of Definition~\ref{def:coherent}(\emph{ii}) implies $\Ll$-convexity,
i.e. the notion of convexity in Definition~\ref{def:diff-convexity}.

For brevity of notations, in the following we will for all law-invariant
functions $\rho:\Ll^{p}(\Omega;\RR^{n})\to\RR$, $p\in[1,\infty]$,
denote its $\Pp^{p}$-representation by the same symbol $\rho$ --
which function we mean will always be clear from its arguments. With
the above definitions, we can now state the necessary assumptions
regarding the risk functions we consider.
\begin{assumption}
\label{assu:baseline-rho}The risk function $\rho$ is such that:
(\emph{i}) $\rho:\Ll^{p}(\Omega;\RR)\to\RR$, $p\in[1,\infty)$, and
$\rho$ is law invariant. (\emph{ii}) The risk function $\rho$ has
an $\Ll$-derivative on some open subset $\Pp^{\prime}\subset\Pp^{p}(\RR)$.
(\emph{iii}) The law of the cost functional is in $\Pp^{\prime}$,
i.e. $\lL(C^{\pi})\in\Pp^{\prime}$ for all $\pi\in\VvV^{p}(b,\sigma,\nu)$.
\end{assumption}

These assumptions simply assert that $\rho$ is differentiable over
a sufficiently large set of random variables. Note that $\Ll$-convexity
is not yet assumed; it will be needed when we state conditions that
are sufficient for optimality of controls.

The probabilistic formulation of the risk aware problem relies on
Assumption~\ref{assu:baseline-rho}, and in particular on the existence
of Fréchet derivatives of the risk function. In general, the question
of the Fréchet differentiability of a function defined over an infinite
dimensional Banach space is a rather complicated one, and presently,
there does not appear to be a well-developed theory of Fréchet differentiability
of risk functions. When the underlying Banach spaces are Asplund spaces,
say in particular if we consider the Hilbert space $\Ll^{2}(\Omega;\RR)$,
Fréchet differentiability is guaranteed at least over a dense $G_{\delta}$-subset
of the space, see e.g. \cite{Phelps1989}. This is somewhat unsatisfactory,
since here we would like to be able to say whether or not a risk function
is differentiable at a specific random variable we have in mind. A
broad treatment of the differentiability of risk functions is beyond
the scope of this paper, however, in Section~\ref{sec:exprob} we
present examples of non-trivial risk functions (i.e. risk functions
that are not simply the expectation) that are in fact differentiable
over a sufficiently large subset of the space. 
\begin{rem}
While the question of the existence of Fréchet derivatives is open,
more is known about the Gâteaux differentiability of risk functions.
For any risk function $\rho:\Ll^{p}(\Omega;\RR)\to\RR$, $p\in[1,\infty]$,
we denote $\dom\rho\deq\{X\in\Ll^{p}(\Omega;\RR):\rho(X)<\infty\}$,
and we say $\rho$ is proper if $\dom\rho\neq\emptyset$. We then
have that a proper, coherent risk function is continuous and subdifferentiable
in the interior of its domain \cite{Ruszczynski06}. In addition,
if $\rho$ is continuous at $X\in\Ll^{p}(\Omega;\RR)$ and the subdifferential
is a singleton, then $\rho$ is Gâteaux differentiable at $X$, and
the mapping $\Ll^{p}\ni Y\to\DF_{Y}\rho(X)$ is continuous (in fact
$\rho$ is differentiable in the somewhat stronger sense of Hadamard
\cite{Bonnans2000}). The Gâteaux differentiability of distortion
risk measures was shown in \cite{Marinacci2004}. Although excluded
from the published version, the Fréchet differentiability was also
discussed in an earlier working paper version of the work, see \cite{Marinacci2002}.
We refer the reader to the work \cite{Lindenstrauss2012} for recent
advances in Fréchet differentiability of convex Lipschitz functions,
such as coherent risk functions over $\Ll^{\infty}(\Omega;\RR)$.
\end{rem}

\section{\label{sec:probform}Risk aware minimum principle}

\paragraph*{Main results}

We begin by stating our risk aware generalization of the stochastic
Pontryagin's minimum principle for Problem~$\pP_{1}$. We denote
$\YY\deq\RR^{1\times d_{x}},$ $\YY^{\prime}\deq\RR$, and $\ZZ\deq\RR^{d_{w}\times d_{x}}$,
and define the Hamiltonian $H$ as
\begin{align}
H(t,x,y,y',z,a) & \deq yb(t,x,a)+y^{\prime}c(t,x,a)+\trace[z\sigma(t,x,a)]\nonumber \\
 & \qquad\forall(t,x,y,y^{\prime},z,a)\in\TT\times\XX\times\YY\times\YY^{\prime}\times\ZZ\times\AA.\label{eq:relaxed-hamiltonian-1}
\end{align}
Note that compared to the risk neutral case, the term involving the
cost rate function has been modified to feature an additional adjoint
variable $y^{\prime}\in\YY^{\prime}$. We shall elaborate on this
significant point later in more detail. We will give both necessary
and sufficient conditions for the $\pP_{1}$-optimality of a control
$\pi\in\VvV^{p}(b,\sigma,\nu)$. For sufficiency, we need an additional
convexity assumption.
\begin{assumption}
\label{assu:sufrel}Suppose Assumptions~\ref{assu:sde-baseline}
and~\ref{assu:baseline-rho} hold, and that additionally: (\emph{i})
the functions $x\to g(x)$ and 
\begin{gather*}
(x,\pi)\to\int_{\AA}H(t,x,y,y^{\prime},z,a)\pi(\D a)
\end{gather*}
are convex for all $(t,y,y^{\prime},z)\in\TT\times\YY\times\YY^{\prime}\times\ZZ$,
and (\emph{ii}) the risk function $\rho$ is $\Ll$-convex.
\end{assumption}

The risk aware minimum principle can then be stated as follows.
\begin{thm}[Risk aware minimum principle]
 \label{thm:ra-min-principle}(\emph{i}) Suppose $\rho:\Pp^{p}(\RR)\to\RR$,
$p\in[1,\infty)$, satisfies Assumption~\ref{assu:baseline-rho}.
If $\pi\in\VvV^{p}(b,\sigma,\nu)$ is $\pP_{1}$-optimal, then there
exists unique $\Ff$-adapted continuous processes $y^{\pi}\in\Ss_{\Ff}^{\pbar/(\pbar-1)}(\Omega;\YY)$
and $y^{\prime\,\pi}\in\Ss_{\Ff}^{p/(p-1)}(\Omega;\YY^{\prime})$,
and a unique $\Ff$-predictable $z^{\pi}\in\Hh_{\Ff}^{\pbar/(\pbar-1)}(\Omega;\ZZ)$
that satisfy the backward stochastic differential equation
\begin{gather}
\D y_{t}^{\pi}=-\nabla_{\XX}H(t,x_{t}^{\pi},y_{t}^{\pi},y_{t}^{\prime\,\pi},z_{t}^{\pi},\pi_{t})\,\D t+z_{t}^{\pi}\cdot\D w_{t},\label{eq:y-bsde}\\
y_{T}^{\pi}=y_{T}^{\prime\,\pi}\nabla_{\XX}g(x_{T}^{\pi}),\nonumber 
\end{gather}
and the representation
\begin{gather}
y_{t}^{\prime\,\pi}=\EE\bigl[\DF\rho(C^{\pi})\bigm|\Ff_{t}\bigr]\quad\forall t\in\TT.\label{eq:y-prime}
\end{gather}
Moreover, the Hamiltonian of Eq.~(\ref{eq:relaxed-hamiltonian-1})
is optimized in the sense that
\begin{align}
\int_{\AA}H(t,x_{t}^{\pi},y_{t}^{\pi},y_{t}^{\prime\,\pi},z_{t}^{\pi},a)\pi_{t}(\D a) & =\inf_{q_{t}\in\Pp^{\pbar_{3}}(\AA)}\int_{\AA}H(t,x_{t}^{\pi},y_{t}^{\pi},y_{t}^{\prime\,\pi},z_{t}^{\pi},a)q_{t}(\D a)\label{eq:H-inf-1}\\
 & \qquad\text{\ensuremath{\PP}-almost surely, for Lebesgue almost all \ensuremath{t\in\TT}.}\nonumber 
\end{align}
(\emph{ii}) Suppose the stronger assumption, Assumption~\ref{assu:sufrel}
holds. If $\pi\in\VvV^{p}(b,\sigma,\nu)$, and if there exist processes
$y^{\pi}\in\Ss_{\Ff}^{\pbar/(\pbar-1)}(\Omega;\YY)$, $y^{\prime\,\pi}\in\Ss_{\Ff}^{p/(p-1)}(\Omega;\YY^{\prime})$,
$z^{\pi}\in\Hh_{\Ff}^{\pbar/(\pbar-1)}(\Omega;\ZZ)$ satisfying Eqs.~(\ref{eq:y-bsde},
\ref{eq:y-prime}, \ref{eq:H-inf-1}), then $\pi$ is $\pP_{1}$-optimal.
\end{thm}

\begin{rem}
\label{rem:no-aug}In some previous works on risk aware optimization
utilizing generic risk functions \cite{Haskell15,Isohatala17_smdp,Isohatala2018_subm},
a state space augmentation scheme was used to derive a computationally
viable form of the risk aware problem. However here, no augmentation
is necessary. This is in contrast to the earlier result where the
state space augmentation was an inextricable part of the end results.
It should also be emphasized that these earlier papers focused on
a convex analytic formulation of the problem whereas here, we consider
a purely probabilistic approach.
\end{rem}

\begin{rem}
It has also not escaped us that the Clark-Ocone theorem, c.f. \cite[Theorem 4.1]{DiNunno2009},
may be further used to characterize the process $(z_{t}^{\prime\,\pi})_{t\in\TT}$
in terms of the Malliavin derivatives of $\DF\rho(\lL(C^{\pi}))(C^{\pi})$.
However, we leave the exploration of this connection to future work.
\end{rem}

Intuitively, the risk aware minimum principle can be seen as a modification
of the risk neutral Pontryagin's minimum principle: Going from the
risk neutral to the risk aware case, an additional process $(y_{t}^{\prime\,\pi})_{t\in\TT}$
is introduced which acts as a rescaling or adjustment factor for given
cost rate and terminal cost functions $c$ and $g$. Moreover, as
per Eq.~(\ref{eq:y-prime}), the values $y_{t}^{\prime\,\pi}$, $t\in\TT$
of the process represent the controller's time $t\in\TT$ expectation
of the derivative of the risk function evaluated at the total cost
$C^{\pi}$. Indeed, if $\rho$ is the expectation, the risk neutral
minimum principle (see e.g. \cite[Section 3.2]{Bahlali2008}) is recovered
with the process $(y_{t}^{\prime\,\pi})_{t\in\TT}$ disappearing in
a natural way.
\begin{cor}
\label{cor:rn-limit}Suppose that the assumptions of Theorem~\ref{thm:ra-min-principle}
hold, and additionally, $p=1$ and $\rho$ is the expectation. Then
the statement of the theorem holds, with the Hamiltonian $H$ replaced
by 
\begin{align}
H_{0}(t,x,y,z,a) & \deq c(t,x,a)+yb(t,x,a)+\trace[z\sigma(t,x,a)]\nonumber \\
 & \qquad\forall(t,x,y,z,a)\in\TT\times\XX\times\YY\times\ZZ\times\AA,\label{eq:relaxed-hamiltonian-risk-neutral-1}
\end{align}
and where the process $(y_{t}^{\prime\,\pi})_{t\in\TT}$ is constant,
$y_{t}^{\prime\,\pi}=1$ for all $t\in\TT$.
\end{cor}

\begin{proof}
Follows from Theorem~\ref{thm:ra-min-principle} due to the fact
that if $\rho=\EE$, then the $\Ll$-derivative $\DF\rho(\cdot)(\cdot)$
is identically one, and by Eq.~(\ref{eq:y-prime}), we have $y_{t}^{\prime\,\pi}=1$
for all $t\in\TT$. Thus, we may also set $y^{\prime}=1$ in the definition
of the Hamiltonian $H$, Eq.~(\ref{eq:relaxed-hamiltonian-1}) to
recover the risk neutral minimum principle.
\end{proof}
\begin{rem}
\label{rem:y-prime-bsde}The process $(y_{t}^{\prime\,\pi})_{t\in\TT}$
in the statement of Theorem~\ref{thm:ra-min-principle} also satisfies
a backward stochastic differential equation that is obtained in an
intermediate step when proving the minimum principle. Specifically,
there is a unique $\Ff$-predictable process $z^{\prime}\in\Hh_{\Ff}^{p/(p-1)}(\Omega;\ZZ^{\prime})$,
$\ZZ^{\prime}\deq\RR^{d_{w}\times1}$, such that 
\begin{gather}
\D y_{t}^{\prime\,\pi}=z_{t}^{\prime\,\pi}\cdot\D w_{t},\label{eq:y-prime-bsde}\\
y_{T}^{\prime\,\pi}=\DF\rho(C^{\pi}).\nonumber 
\end{gather}
Therefore together, Eqs.~(\ref{eq:vague-sde}), (\ref{eq:y-bsde}),
and~(\ref{eq:y-prime-bsde}) form a forward-backward system of stochastic
differential equations with $d_{x}$ and $d_{x}+1$ state and adjoint
state variables, respectively.
\end{rem}

\begin{rem}
Returning to Example~\ref{exa:std-counter}, and setting $c=0$ and
$g(x)=x^{2}/2$ for all $x\in\XX$, and $\rho=\EE$, we can easily
see that Assumption~\ref{assu:sufrel} is satisfied. The Hamiltonian
becomes $H(z,a)=az$, and by Eq.~(\ref{eq:H-inf-1}), $\PP$-almost
surely for almost all $t\in\TT$, $\pi_{t}=\delta_{+1}$ if $z_{t}^{\pi}<0$
and $\pi_{t}=\delta_{-1}$ if $z_{t}>0$. However, the minimization
of the Hamiltonian does not determine $\pi_{t}$ when $z_{t}=0$.
At first glance this would seem to imply that the conditions given
in Theorem~\ref{thm:ra-min-principle} are not in fact sufficient
to fix the optimal control. But since $\D x_{t}^{\pi}=\int_{\AA}a\pi_{t}(\D a)\,\D w_{t}$,
and $\D y_{t}^{\pi}=z_{t}^{\pi}\D w_{t}$, $y_{T}=x_{T}$, by the
uniqueness of the solutions we must have that $y_{t}=x_{t}$ and $z_{t}=\int_{\AA}a\pi_{t}(\D a)\,\D w_{t}$
for all $t\in\TT$, which then implies $z_{t}=-z_{t}$ and in turn
that $z_{t}=0$ and $\pi_{t}=(\delta_{+1}+\delta_{-1})/2$ for all
$t\in\TT$. Therefore in order to find the optimal control it may
be insufficient to only minimize the Hamiltonian, and instead one
needs to determine the adjoint processes as well.
\end{rem}

\paragraph*{Proofs of the main results}

The rest of this section is dedicated to proving the risk aware minimum
principle, Theorem~\ref{thm:ra-min-principle}. We present our intermediate
steps in reaching the main result, but defer the details of their
proofs to Appendix~\ref{subsec:main-proofs}.

We adopt the following short-hand: For every Borel measurable function
$f:\AA\to\VV$ and every $\pi_{1},\pi_{2}\in\Pp(\AA)$ and $a_{1},a_{2}\in\RR$,
we denote
\begin{gather*}
f(a_{1}\pi_{1}+a_{2}\pi_{2})\deq a_{1}\int_{\AA}f(a)\pi_{1}(\D a)+a_{2}\int_{\AA}f(a)\pi_{2}(\D a).
\end{gather*}
In addition to the original stochastic differential equation, Eq.~(\ref{eq:vague-sde})
describing a controlled process $(x_{t}^{\pi})_{t\in\TT}$, we introduce
the additional, coupled differential equation for a $\RR$-valued
process, the running costs, $x^{\prime}=(x_{t}^{\prime\,\pi})_{t\in\TT}$
defined as
\begin{align}
x_{t}^{\prime\,\pi} & =\int_{0}^{t}\int_{\AA}c(s,x_{s}^{\pi},a)\pi_{s}(\D a)\,\D s.\label{eq:sde-aug}
\end{align}
We can then re-write the total cost as 
\begin{gather*}
C^{\pi}=x_{T}^{\prime\,\pi}+g(x_{T}^{\pi}).
\end{gather*}
We shall use $\XX^{\prime}=\RR$ to indicate the range of the process
$x^{\prime}$.

In order to establish optimality conditions for vague controlled solutions,
we need first and foremost some means of comparing pairs of solutions.
This poses a slight technical challenge, as the state and control
space processes for any given pair $\pi,\pi^{\prime}\in\VvV^{p}(b,\sigma,\nu)$
may be defined on different probability spaces. A natural way of comparing
weak solutions of stochastic differential equations would be to compare
the finite dimensional distributions of the state space processes
(and the distributions of the cost variables $C^{\pi}$). However,
since by Assumptions~\ref{assu:sde-baseline} and Proposition~\ref{prop:basic-existence-uniqueness},
strong solutions exists for given filtered probability spaces and
control process, we can do slightly better. Specifically, we can construct
an extended probability space simultaneously supporting both vague
controlled solutions, and on which we can compare the \emph{pathwise
laws} of the solutions.
\begin{lem}
\label{lem:join-spaces}Suppose Assumption~\ref{assu:sde-baseline}
holds, and that $(\Omega,\Sigma,\Ff,\PP,x,\pi)\in\VvV^{p}(b,\sigma,\nu)$
and $(\Omega^{\prime},\Sigma^{\prime},\Ff^{\prime},\PP^{\prime},x^{\prime},\pi^{\prime})\in\VvV^{p}(b,\sigma,\nu)$,
and let $w=(w_{t})_{t\in\TT}$ and $w^{\prime}=(w_{t}^{\prime})_{t\in\TT}$
be the corresponding Brownian motions. Then there exists a filtered
probability space $(\tilde{\Omega},\tilde{\Sigma},\tilde{\Ff},\tilde{\PP})$
supporting an $\tilde{\Ff}$-Brownian motion $\tilde{w}=(\tilde{w}_{t})_{t\in\TT}$,
and processes $(\tilde{x}_{t},\tilde{\pi}_{t})_{t\in\TT}$ and $(\tilde{x}_{t}^{\prime},\tilde{\pi}_{t}^{\prime})_{t\in\TT}$
satisfying
\begin{gather}
\begin{aligned}\tilde{x}_{t} & =x_{0}+\int_{0}^{t}b(s,\tilde{x}_{s},\tilde{\pi}_{s})\,\D s+\int_{0}^{t}\sigma(s,\tilde{x}_{s},\tilde{\pi}_{s})\,\D\tilde{w}_{s},\\
\tilde{x}_{t}^{\prime} & =x_{0}^{\prime}+\int_{0}^{t}b(s,\tilde{x}_{s}^{\prime},\tilde{\pi}_{s}^{\prime})\,\D s+\int_{0}^{t}\sigma(s,\tilde{x}_{s}^{\prime},\tilde{\pi}_{s}^{\prime})\,\D\tilde{w}_{s}
\end{aligned}
\label{eq:sde-joint}
\end{gather}
for all $t\in\TT$, and such that their laws equal those of $(x_{t},\pi_{t})_{t\in\TT}$
and $(x_{t}^{\prime},\pi_{t}^{\prime})_{t\in\TT}$.
\end{lem}

For simplicity, we shall implicitly suppose that pairs of vague controlled
solutions are defined on the same probability space. In addition,
for every pair of vague controlled solutions $\pi,q\in\VvV^{p}(b,\sigma,\nu)$,
the convex combination of their control processes shall be denoted
by the short-hand $\pi(\alpha,q)$, that is, for all $\pi$, $q$
and $\alpha\in[0,1]$,
\begin{gather*}
\pi_{t}(\alpha,q)\deq\pi_{t}+\alpha(q_{t}-\pi_{t})\quad\forall t\in\TT.
\end{gather*}
We will the control $q$ as a perturbation of the original, reference
control $\pi$, and our goal is to derive optimality conditions from
variational equations representing the response of the solution to
$q$.

We begin with a few auxiliary results, variations of which have appeared
in the literature. The following lemma states that solutions corresponding
to perturbed controls are, uniformly in time, good approximations
of the unperturbed solutions.
\begin{lem}
\label{lem:var-vanish}For all $\pi,q\in\VvV^{p}(b,\sigma,\nu)$ and
$\alpha\in[0,1]$,
\begin{subequations}
\label{eq:sups}
\begin{align}
\Snorm{x^{\pi(\alpha,q)}-x^{\pi}}{\pbar} & \in\Oo(\alpha),\label{eq:sup-pbar}\\
\Snorm{x^{\prime\,\pi(\alpha,q)}-x^{\prime\,\pi}}p & \in\Oo(\alpha),\label{eq:sup-p}
\end{align}
\end{subequations}
where $\pbar$ is as in Assumption~\ref{assu:sde-baseline}. In addition,
we have for the terminal cost
\begin{align}
\Lnorm{g\bigl(x_{T}^{\pi(\alpha,q)}\bigr)-g\bigl(x_{T}^{\pi}\bigr)}p & \in\Oo(\alpha).\label{eq:g-sup-p}
\end{align}
\end{lem}

The following lemma provides the means for computing the first-order
response of solutions to perturbations of the control process.
\begin{lem}
\label{lem:epsilon}Let $\pi,q\in\VvV^{p}(b,\sigma,\nu)$ be arbitrary.
Then there exists an $\XX$-valued process $\delta^{\pi,q}=(\delta_{t}^{\pi,q})_{t\in\TT}$
that is the unique strong solution of
\begin{align}
\delta_{t}^{\pi,q} & =\int_{0}^{t}\left[\nabla_{\XX}b(s,x_{s}^{\pi},\pi_{s})\delta_{s}^{\pi,q}+b(s,x_{s}^{\pi},q_{s}-\pi_{s})\right]\,\D s\nonumber \\
 & \qquad+\int_{0}^{t}\left[\nabla_{\XX}\sigma(s,x_{s}^{\pi},\pi_{s})\delta_{s}^{\pi,q}+\sigma(s,x_{s}^{\pi},q_{s}-\pi_{s})\right]\,\D w_{s}.\label{eq:delta-solo}
\end{align}
Moreover, defining $\delta^{\prime\,\pi,q}=(\delta_{t}^{\prime\,\pi,q})_{t\in\TT}$
as
\begin{gather}
\delta_{t}^{\prime\,\pi,q}\deq\int_{0}^{t}\left[\nabla_{\XX}c(s,x_{s}^{\pi},\pi_{s})\delta_{s}^{\pi,q}+c(s,x_{s}^{\pi},q_{s}-\pi_{s})\right]\,\D s,\label{eq:delta-prime}
\end{gather}
we have
\begin{subequations}
\label{eq:delta-finite}
\begin{gather}
\EE\biggl[\sup_{t\in\TT}\bigl|\delta_{t}^{\pi,q}\bigr|^{\pbar}\biggr]<\infty,\label{eq:delta-pbar-finite}\\
\exists r>p:\quad\EE\biggl[\sup_{t\in\TT}\bigl|\delta_{t}^{\prime\,\pi,q}\bigr|^{r}\biggr]<\infty,\label{eq:delta-prime-p-finite}
\end{gather}
\end{subequations}
and, for all $\alpha\in[0,1]$,
\begin{subequations}
\label{eq:sup-eps}
\begin{gather}
\Snorm{x^{\pi(\alpha,q)}-x^{\pi}-\alpha\delta^{\pi,q}}{\pbar}\in o(\alpha),\label{eq:sup-eps-o-pbar}\\
\Snorm{x^{\prime\,\pi(\alpha,q)}-x^{\prime\,\pi}-\alpha\delta^{\prime\,\pi,q}}{\pbar}\in o(\alpha).\label{eq:sup-eps-prime-o-p}
\end{gather}
\end{subequations}
\end{lem}

The next results connect the response of the dynamics to the perturbation,
described by the process $(\delta^{\pi,q},\delta^{\prime\,\pi,q})$
and characterized by the above lemmas, to the risk aware objectives.
Let us for brevity denote for all $\pi\in\VvV^{p}(b,\sigma,\nu)$
and for any law invariant risk function $\rho:\Pp^{p}(\RR)\to\RR$,
with an $\Ll$-derivative $\DF\rho(\cdot)(\cdot):\Pp^{p}(\RR)\times\RR\to\RR$
\begin{gather}
\begin{gathered}\Theta^{\pi}\deq\lL(C^{\pi})=\lL(\theta(x_{T}^{\pi},x_{T}^{\prime\,\pi})),\\
D^{\pi}\deq\DF\rho(\Theta^{\pi})(\theta(x_{T}^{\pi},x_{T}^{\prime\,\pi})),
\end{gathered}
\label{eq:ThetaD}
\end{gather}
where
\begin{gather*}
\theta(x,x^{\prime})\deq g(x)+x^{\prime}\qquad\forall(x,x^{\prime})\in\XX\times\XX^{\prime}.
\end{gather*}

If $\pi\in\VvV^{p}(b,\sigma,\nu)$ is $\pP_{1}$-optimal, then by
definition for any $q\in\VvV^{p}(b,\sigma,\nu)$ and $\alpha\in[0,1]$
we have that
\begin{gather}
0\leq\rho(\Theta^{\pi(\alpha,q)})-\rho(\Theta^{\pi}).\label{eq:var-optimality}
\end{gather}
We will use Eq.~(\ref{eq:var-optimality}) as a starting point for
deriving our optimality conditions.
\begin{lem}
\label{lem:delta-term}Suppose Assumption~\ref{assu:baseline-rho}
holds. Let $\pi\in\VvV^{p}(b,\sigma,\nu)$ be $\pP_{1}$-optimal and
$q\in\VvV^{p}(b,\sigma,\nu)$ arbitrary, and let the process $(\delta_{t}^{\pi,q},\delta_{t}^{\prime\,\pi,q})_{t\in\TT}$
be as in the statement of Lemma~\ref{lem:epsilon}. Then
\begin{gather}
0\leq\EE\biggl[D^{\pi}\Bigl(\nabla_{\XX}g(x_{T}^{\pi})\delta_{T}^{\pi,q}+\delta_{T}^{\prime\,\pi,q}\Bigr)\biggr],\label{eq:var-deltaform}
\end{gather}
where $D^{\pi}$ is as defined in Eq.~(\ref{eq:ThetaD}).
\end{lem}

We can now construct the adjoint processes $(y_{t}^{\pi},y_{t}^{\prime\,\pi},z_{t}^{\pi},z_{t}^{\prime\,\pi})_{t\in\TT}$
appearing in Eqs.~(\ref{eq:y-bsde}, \ref{eq:y-prime-bsde}), and
use them to restate the optimality condition of Eq.~(\ref{eq:var-deltaform}).
The proof of the lemma follows roughly the same ideas as used in the
risk neutral case, see e.g. \cite{Bismut1973,Bismut1978,Bensoussan1982,Peng1990},
and relies primarily on the martingale representation theorem. In
the risk aware case, we need to additionally handle the nonlinearity
of the risk aware objective, which gives rise to the risk adjustment
process.
\begin{lem}
\label{lem:var-hamil}Suppose Assumptions~~\ref{assu:baseline-rho}
hold, and that 
\[
(\Omega,\Sigma,\Ff=(\Ff_{t})_{t\in\TT},\PP,x^{\pi}=(x_{t}^{\pi})_{t\in\TT},w=(w_{t})_{t\in\TT},\pi=(\pi_{t})_{t\in\TT})\in\VvV^{p}(b,\sigma,\nu)
\]
is $\pP_{1}$-optimal. Then there are unique $\Ff$-adapted continuous
processes $y^{\pi}\in\Ss_{\Ff}^{\pbar/(\pbar-1)}(\Omega;\YY)$ and
$y^{\prime\,\pi}\in\Ss_{\Ff}^{p/(p-1)}(\Omega;\YY^{\prime})$, and
unique $\Ff$-predictable processes $z^{\pi}\in\Hh_{\Ff}^{\pbar/(\pbar-1)}(\Omega;\ZZ)$
and $z^{\prime\,\pi}\in\Hh_{\Ff}^{p/(p-1)}(\Omega;\ZZ^{\prime})$
satisfying the backward stochastic differential equations
\begin{subequations}
\label{eq:relaxed-bsde}
\begin{gather}
\D y_{t}^{\pi}=-\nabla_{\XX}H(t,x_{t}^{\pi},y_{t}^{\pi},y_{t}^{\prime\,\pi},z_{t}^{\pi},\pi_{t})\,\D t+z_{t}^{\pi}\cdot\D w_{t},\label{eq:proof-y-bsde}\\
\D y_{t}^{\prime\,\pi}=z_{t}^{\prime\,\pi}\cdot\D w_{t},\label{eq:proof-y-prime-bsde}\\
y_{T}^{\pi}=D^{\pi}\nabla_{\XX}g(x_{T}^{\pi}),\quad y_{T}^{\prime\,\pi}=D^{\pi},\label{eq:proof-bsde-tc}
\end{gather}
\end{subequations}
where $H$ is as defined in Eq.~(\ref{eq:relaxed-hamiltonian-1}),
and Eq.~(\ref{eq:var-deltaform}) implies that for all $q\in\VvV^{p}(b,\sigma,\nu)$,
\begin{gather}
0\leq\EE\left[\int_{0}^{T}H(t,x_{t}^{\pi},y_{t}^{\pi},y_{t}^{\prime\,\pi},z_{t}^{\pi},q_{t})\,\D t-\int_{0}^{T}H(t,x_{t}^{\pi},y_{t}^{\pi},y_{t}^{\prime\,\pi},z_{t}^{\pi},\pi_{t})\,\D t\right].\label{eq:var-hamil}
\end{gather}
\end{lem}

Finally, we show that existence of the adjoints and minimization of
the Hamiltonian is indeed sufficient to establish optimality. 
\begin{lem}
\label{lem:suff}Suppose Assumption~\ref{assu:sufrel} holds, $\pi\in\VvV^{p}(b,\sigma,\nu)$,
and there exists processes $y^{\pi}\in\Ss_{\Ff}^{\pbar/(\pbar-1)}(\Omega;\YY)$,
$y^{\prime\,\pi}\in\Ss_{\Ff}^{p/(p-1)}(\Omega;\YY^{\prime})$, $z^{\pi}\in\Hh_{\Ff}^{\pbar/(\pbar-1)}(\Omega;\ZZ)$
satisfying Eqs.~(\ref{eq:y-bsde}, \ref{eq:y-prime}, \ref{eq:H-inf-1}).
Then 
\begin{gather*}
\rho(\Theta^{\pi})-\rho(\Theta^{q})\leq0
\end{gather*}
for every $q\in\VvV^{p}(b,\sigma,\nu)$.
\end{lem}

We can now collect the above together and give the proof of our main
result, Theorem~\ref{thm:ra-min-principle}.
\begin{proof}[Proof of Theorem~\ref{thm:ra-min-principle}]
The first part of the theorem now follow directly from Lemma~\ref{lem:var-hamil}
and Eq.~(\ref{eq:var-hamil}), while the second is a direct consequence
of Lemma~\ref{lem:suff}. The representation of Eq.~(\ref{eq:y-prime-bsde})
follows directly from Lemma~\ref{lem:var-hamil}.
\end{proof}

\section{\label{sec:exprob}Examples of differentiable risk functions and
a portfolio allocation problem}

The purpose of this section is to present an application of the results
of previous sections, and hence the problem we consider is selected
for simplicity while attempting to retain a reasonable degree of practical
significance.

\paragraph*{Risk functions}

As examples of law invariant risk functions, we use the mean-deviation,
the (smoothed) mean-semideviation, and entropic risk functionals.
\begin{defn}
Let $(\Omega,\Sigma,\PP)$ be a probability space. (\emph{i}) \emph{Mean-deviation
risk function} $\rhoMD:\Ll^{2}(\Omega;\RR)\to\RR$ is defined as the
mapping
\begin{gather}
\rhoMD(X)\deq\EE\left[X\right]+\beta\bigl\Vert X-\EE\left[X\right]\bigr\Vert_{2}\quad\forall X\in\Ll^{2}(\Omega;\RR),\label{eq:rho-mean-dev}
\end{gather}
where $\beta>0$. (\emph{ii}) \emph{Mean-semideviation risk function}
$\rhoMSD:\Ll^{1}(\Omega;\RR)\to\RR$ and the \emph{$\epsilon$-smoothed
mean-semideviation risk function} $\rhoEpsMSD:\Ll^{1}(\Omega;\RR)\to\RR$,
$\epsilon>0$, are defined as
\begin{gather*}
\rhoMSD(X)\deq\EE\left[X\right]+\beta\EE\left[(X-\EE\left[X\right])_{+}\right]\quad\forall X\in\Ll^{1}(\Omega;\RR),\\
\rhoEpsMSD(X)\deq\EE\left[X\right]+\beta\EE\left[(X-\EE\left[X\right])_{\epsilon+}\right]\quad\forall X\in\Ll^{1}(\Omega;\RR),
\end{gather*}
where $(\cdot)_{+}:\RR\to\RR_{\geq0}$ and $(\cdot)_{\epsilon+}:\RR\to\RR_{>0}$
are the positive part and $\epsilon$-smoothed positive part functions,
$(x)_{+}\deq x\vee0$ and $(x)_{\epsilon+}\deq x+\epsilon\ln(1+\E^{-x/\epsilon})$
for all $x\in\RR$ and $\epsilon>0$. (\emph{iii}) \emph{Entropic
risk function} is the risk measure $\rhoEntr:\Ll^{\infty}(\Omega;\RR)\to\RR$
defined as
\begin{gather}
\rhoEntr(X)\deq\frac{1}{\theta}\ln\EE\left[\E^{\theta X}\right]\quad\forall X\in\Ll^{\infty}(\Omega;\RR),\label{eq:rho-entropic}
\end{gather}
where $\theta>0$.
\end{defn}

We note that the mean-deviation risk function is convex, positively
homogeneous, and translation invariant, that is, it satisfies Definition~\ref{def:coherent}
items (\emph{ii}), (\emph{iii}), and (\emph{iv}). The $\Ll^{1}(\Omega;\RR)$
mean-semideviation risk measure $\rhoMSD$ was considered in e.g.
\cite{Ruszczynski06}, and it too is convex, positively homogeneous,
and translation invariant, but is additionally monotonic, satisfying
Definition~\ref{def:coherent}(\emph{i}). As noted in Remark~\ref{rem:no-diff},
the positive homogeneity of these functionals implies that they cannot
be everywhere Fréchet differentiable. We demonstrate in the example
problem below that this is not necessarily an issue for our purposes.
Moreover, the $\epsilon$-smoothed mean-semideviation risk function
$\rhoEpsMSD$ uniformly approximates $\rhoMSD$, that is,
\begin{gather*}
0<\rhoEpsMSD(X)-\rhoMSD(X)\leq\epsilon\beta\ln2\quad\forall X\in\Ll^{1}(\Omega;\RR),\,\forall\epsilon>0,
\end{gather*}
but its restriction to $\Ll^{2}(\Omega;\RR)$ is in fact everywhere
Fréchet differentiable (this will be established in Lemma~\ref{lem:easy-derivatives}
below). The smoothed mean-semideviation is also convex and monotonic
which, along with the above estimate, follows directly from the properties\footnote{Specifically, from the inequality $0<(x)_{\epsilon+}-(x)_{+}\leq\epsilon\ln2\,\forall x\in\RR$,
and the monotonicity and convexity of $(\cdot)_{\epsilon+}$.} of the $\epsilon$-smoothed positive part function \cite{Chen95}.
Our definition of $\rhoEpsMSD$ was inspired by the construction of
a smoothed conditional value-at-risk risk functional in \cite{Kouri16}.
The entropic risk function $\rhoEntr$ on the other hand satisfies
monotonicity, convexity, and translation invariance properties, or
items (\emph{i}), (\emph{ii}) and (\emph{iv}) of Definition~\ref{def:coherent}.
It serves as an example of a commonly used risk function that is everywhere
Fréchet differentiable.
\begin{lem}
\label{lem:easy-derivatives}(\emph{i}) The mean-deviation risk function
is Fréchet differentiable at every $X\in\Ll^{2}(\Omega;\RR)$ that
is not almost surely constant, with the derivative $\DF\rhoMD(X)\in\Ll^{2}(\Omega;\RR)$
being
\begin{gather}
\DF\rhoMD(X)=1+\beta\frac{X-\EE[X]}{\bigl\Vert X-\EE[X]\bigr\Vert_{2}}.\label{eq:D-mean-dev}
\end{gather}
Moreover, the derivative does not exist at $X\in\Ll^{2}(\Omega;\RR)$
such that $X=\EE[X]$. It additionally has the $\Ll$-derivative $\DF\rhoMD:\Pp^{2}(\RR)\times\RR\to\RR$
that reads, for all $\mu\in\Pp^{2}(\RR)$ that are not a Dirac measures,
\begin{gather}
\DF\rhoMD(\mu)(x)=1+\beta\frac{x-\int x^{\prime}\mu(\D x^{\prime})}{\left[\int\left(x^{\prime\prime}-\int x^{\prime}\mu(\D x^{\prime})\right)^{2}\mu(\D x^{\prime\prime})\right]^{1/2}}\quad\forall x\in\RR.\label{eq:LD-mean-dev}
\end{gather}

(\emph{ii}) The $\Ll^{2}(\Omega;\RR)$-restriction of the $\epsilon$-smoothed
mean-semideviation risk function $\rhoEpsMSD$, $\epsilon>0$, is
Fréchet differentiable at every $X\in\Ll^{2}(\Omega;\RR)$, and has
the Fréchet- and $\Ll$-derivatives 
\begin{align}
\DF\rhoEpsMSD(X) & =1+\beta\left\{ U_{\epsilon}(X-\EE[X])-\EE\left[U_{\epsilon}(X-\EE[X])\right]\right\} \in\Ll^{\infty}(\Omega;\RR)\nonumber \\
 & \qquad\forall X\in\Ll^{2}(\Omega;\RR),\label{eq:D-mean-semidev}\\
\DF\rhoEpsMSD(\mu)(x) & =1+\beta\biggl\{ U_{\epsilon}\left(x-\int x^{\prime}\mu(\D x^{\prime})\right)\nonumber \\
 & \qquad\qquad-\int U_{\epsilon}\left(x^{\prime\prime}-\int x^{\prime}\mu(\D x^{\prime})\right)\,\mu(\D x^{\prime\prime})\biggr\}\in(1-\beta,1+\beta)\nonumber \\
 & \qquad\forall\mu\in\Pp^{2}(\RR),\,x\in\RR,\nonumber 
\end{align}
respectively, and where $U_{\epsilon}(x)\deq\D(x)_{\epsilon+}/\D x=1/(1+\E^{-x/\epsilon})$
for all $x\in\RR$.

(\emph{iii}) The entropic risk measure is Fréchet differentiable at
every $X\in\Ll^{\infty}(\Omega;\RR)$, with the Fréchet- and $\Ll$-derivatives
$\DF\rhoEntr(X)\in\Ll^{1}(\Omega;\RR)$ and $\DF\rhoEntr(\mu)\in\RR\to\RR$,
$\mu\in\Pp^{\infty}(\RR)$,
\begin{gather*}
\DF\rhoEntr(X)=\frac{\E^{\theta X}}{\EE\left[\E^{\theta X}\right]},\quad\DF\rhoEntr(\mu)(x)=\frac{\E^{\theta x}}{\int\E^{\theta x^{\prime}}\mu(\D x^{\prime})}\quad\forall x\in\RR.
\end{gather*}
\end{lem}

\begin{rem}
If the $\Ll_{2}$-norm $\Vert\cdot\Vert_{2}$ in Eq.~(\ref{eq:rho-mean-dev})
is replaced by its square, it is easy to verify that the resulting
risk function is everywhere Fréchet differentiable. 
\end{rem}

\paragraph*{Portfolio allocation problem}

As a practical example, we consider a simplified portfolio allocation
problem. An agent manages a portfolio consisting of a risk free bond,
yielding a constant return rate $r>0$, and a risky stock whose price
$(q_{t})_{t\in\TT}$ evolves according to $\D q_{t}=\mu q_{t}\,\D t+\sigma q_{t}\,\D w_{t}$,
$q_{0}=1$, $\mu>0$, $\sigma>0$. Let $N_{t}=B_{t}+q_{t}S_{t}$ be
the net value of the agent's portfolio where $B_{t}$ and $S_{t}$
represent the agent's bond and stock holdings at any $t\in\TT$, respectively.
Let $\phi_{t}\deq q_{t}S_{t}/N_{t}$ be the proportion of the agent's
portfolio allocated to the risky asset, so that $N_{t}$ follows the
stochastic differential equation
\begin{gather}
\D N_{t}=\left[r+(\mu-r)\phi_{t}\right]N_{t}\,\D t+\sigma\phi_{t}N_{t}\,\D w_{t},\label{eq:N-sde}
\end{gather}
with a given initial condition $N_{0}$. Trading is costless and unconstrained
so that $\phi_{t}$ is a choice variable for each $t\in\TT$. We suppose
$\phi_{t}$ is constrained to the interval $\AA=[\underline{\phi},\bar{\phi}]$
where $0<\underline{\phi}<\text{\ensuremath{\bar{\phi}}}<\infty$,
the agent optimizes the allocation so that the risk of the utility
of $N_{T}$ is minimized. Here, the agent values their profits or
losses using a logarithmic utility, so that their total cost evaluates
to $-\ln N_{T}$.

Re-writing Eq.~(\ref{eq:N-sde}) for the logarithm of $N_{t}$, $x_{t}^{\pi}\deq\ln N_{t}$
for all $t\in\TT$, and generalizing to a relaxed controlled process,
we have that
\begin{gather}
\D x_{t}^{\pi}=\left[r+(\mu-r)\int_{\AA}\phi\,\pi_{t}(\D\phi)-\frac{1}{2}\sigma^{2}\int_{\AA}\phi^{2}\,\pi_{t}(\D\phi)\right]\,\D t+\sigma\int_{\AA}\phi\,\pi_{t}(\D\phi)\,\D w_{t},\label{eq:x-phi}
\end{gather}
where $x_{0}^{\pi}=x_{0}\in\RR$ is given. Let $b_{\phi}$ and $\sigma_{\phi}$
be the drift and diffusion coefficients of Eq.~(\ref{eq:x-phi}),
and let $\nu_{\phi}=\delta_{x_{0}}$. Assumption~\ref{assu:sde-baseline}
is now satisfied, with $\pbar_{1}=0$, $\pbar_{2}=0$, $\pbar_{3}=\infty$,
$p_{1}=1$, $p_{2}=0$, $p_{1}^{\prime}=0$, and $p_{2}^{\prime}=0$.
Since the initial condition is deterministic, $\pbar$ may be selected
to be arbitrarily large. It is easy to verify that Eq.~(\ref{eq:feasible-p})
holds for any $p\in[1,\infty)$, so that we may consider $p$-feasible
solutions $\pi\in\VvV^{p}(b_{\phi},\sigma_{\phi},\nu_{\phi})$.

The risk aware control problem, Problem $\pP_{\phi}$, becomes
\begin{gather*}
\pP_{\phi}:\quad\inf_{\pi\in\VvV^{p}(b_{\phi},\sigma_{\phi},\nu_{\phi})}\rho(-x_{T}^{\pi}).
\end{gather*}
We note that for instance the mean-deviation risk function of Eq.~(\ref{eq:rho-mean-dev})
is $\Ll$-differentiable at $-x_{T}^{\pi}$.
\begin{prop}
There is no $\pi\in\VvV^{p}(b_{\phi},\sigma_{\phi},\nu_{\phi})$ such
that $-x_{T}^{\pi}$ is almost surely bounded.
\end{prop}

\begin{proof}
Since for any $\pi\in\VvV^{p}(b_{\phi},\sigma_{\phi},\nu_{\phi})$
the drift and diffusion are bounded, and the latter is always non-zero,
$x_{T}^{\pi}$ can take arbitrarily large values. 
\end{proof}
Since $-x_{T}^{\pi}$ is not bounded, it cannot be constant, and therefore
$\rhoMD$ is $\Ll$-differentiable at the terminal cost. In addition,
the e.g. the mean-deviation risk function or the $\Ll^{2}(\Omega;\RR)$
restriction of the $\epsilon$-smoothed mean-semideviation risk function
together with the cost $-x_{T}^{\pi}$ satisfy Assumption~\ref{assu:baseline-rho}.

We can now use the risk aware minimum principle to characterize an
optimal allocation process. For simplicity, we assume that the $\Ll$-derivative
of the risk function is positive (this is the case for e.g. the $\epsilon$-smoothed
mean-semideviation when $\beta<1$). Non-positive values of the derivative
can also be easily accommodated, but the added complexity would detract
from the intuition of this example, which is to illustrate how risk
awareness can manifest itself in real world applications.
\begin{prop}
\label{prop:alloc}Suppose $\rho:\Ll^{p}(\Omega;\RR)\to\RR$, $p\in[1,\infty)$,
is convex, satisfies Assumption~\ref{assu:baseline-rho}, and has
a positive $\Ll$-derivative, i.e. $\DF\rho(\mu)(x)>0$ for all $\mu\in\Pp^{p}(\RR)$,
$x\in\RR$. The optimal portfolio allocation for Problem~$\pP_{\phi}$
is a strict control $\pi\in\VvV^{p}(b_{\phi},\sigma_{\phi},\nu_{\phi})$
such that $\pi_{t}=\delta_{\phi_{t}}$ for all $t\in\TT$ where 
\begin{gather*}
\phi_{t}=\underline{\phi}\vee\frac{\mu-r+\iota_{t}}{\sigma^{2}}\wedge\bar{\phi}\quad\forall t\in\TT,
\end{gather*}
and where 
\begin{gather}
\iota_{t}\deq\frac{\sigma z_{t}^{\prime}}{y_{t}^{\prime}}\quad\forall t\in\TT,\label{eq:risk-premium}
\end{gather}
is a risk premium in which
\begin{gather}
y_{t}^{\prime}=\EE\bigl[\DF\rho(\lL(-x_{T}^{\pi}))(-x_{T}^{\pi})\bigm|\Ff_{t}\bigr],\label{eq:phi-y-prime}\\
\EE\bigl[\DF\rho(\lL(-x_{T}^{\pi}))(-x_{T}^{\pi})\bigm|\Ff_{t}\bigr]=\EE\bigl[\DF\rho(\lL(-x_{T}^{\pi}))(-x_{T}^{\pi})\bigr]+\int_{0}^{t}z_{t'}^{\prime\,\pi}\D w_{t'},\label{eq:phi-y-prime-bsde}
\end{gather}
for all $t\in\TT$.
\end{prop}

We note that interestingly, the risk awareness of the objective function
has now given rise to the additional risk premium process $(\iota_{t})_{t\in\TT}$
defined in Eq.~(\ref{eq:risk-premium}). To wit, the risk premium
vanishes if $\rho$ is the expectation, since then as noted in Corollary~\ref{cor:rn-limit},
$y_{t}^{\prime\,\pi}=1$ for all $t\in\TT$ implying that $z_{t}^{\prime\,\pi}=0$
for all $t\in\TT$. Thus, the risk aware minimum principle may open
new possibilities in e.g. risk pricing theory.

\section{\label{sec:concl}Conclusions}

In Theorem~\ref{thm:ra-min-principle} we have given a risk aware
generalization of the stochastic minimum principle. A notable feature
of the result is the way risk is captured via the risk adjustment
process, essentially the marginal risk at a given time $t\in\TT$,
Eq.~(\ref{eq:y-prime}). We argue that at least some form of a risk
adjustment process is an inevitable consequence of the risk awareness,
or effectively of the nonlinearity of the risk function. In our risk
aware context, it is natural to expect that the optimal control should
account for changes in the way the risk responds to changes in the
terminal cost, given the information $\Ff_{t}$ at any time $t\in\TT$.
Indeed, the \emph{raison d'etre} of dynamic risk measures is their
property of time-consistency which prescribes the dependence of the
risk function on the filtration. In the result we obtained, this risk
accounting is represented by the $\Ff_{t}$-conditional expectation
of the $\Ll$-derivative of the risk function evaluated at the terminal
cost.

Although by not requiring that the risk functions are time-consistent
we have provided a rather general version of a risk aware minimum
principle, we have on the other hand opened ourselves to the possibility
that the \emph{optimal controls} might not be time-consistent. By
this we mean that if the optimization problem were restarted at some
time $t>0$, the optimal value and control might change, and the controller
could be better off by switching to a different control policy. However,
since our risk aware minimum principle characterizes the optimal control
in terms of the risk adjustment process, it is now possible to find
new, sufficient conditions for time consistency of the controls. Moreover,
it may now also be possible to consider constrained optimal control
problems, where the purpose of the constraint is to enforce time-consistency
of optimal controls.

Our minimum principle also gives, up to the knowledge of the authors,
the first characterization of the risk aware optimal control that
can be used to derive conditions under which an optimal control is
strict or Markov. A simple application of Jensen's inequality was
used in the example problem to show that in that instance, a strict
optimal control exists. Generalizations of this statement are not
hard to imagine. The question of existence of Markov controls may
be possible to explore using recent results on forward-backward stochastic
differential equations. For instance \cite{Fromm2013,Fromm2015} present
conditions under which the adjoint processes can be expressed as functions
of the time and state variables, which could allow writing the optimal
control as a function of time and state variables only. 

Finally, one of the key assumptions in the risk aware minimum principle
is the Fréchet differentiability of the risk function $\rho$. For
our results to hold, it is necessary that the risk function be differentiable
over the random variables representing the total cost. Establishing
more precisely what risk functions are Fréchet differentiable over
a sufficiently large subset of random variables would widen the applicability
of the results given in this paper. 

\section*{Acknowledgments}

The authors acknowledge the funding provided by the Singapore Ministry
of Education, Tier II grant MOE2015-T2-2-148, \emph{Practical algorithms
for large-scale sequential optimization}.

\appendix

\section{\label{sec:proofs}Proofs of the results}

\subsection{Proofs for Section 2}

\begin{proof}[Proof of Example~\ref{exa:first-order}]
The main inequality of the example follows from a straight-forward
application of the definition of the perturbed control and the Burkholder-Davis-Gundy
inequality, \cite[Theorem 1.76]{Pardoux2014}. Explicitly,
\begin{align*}
\sup_{t\in\TT}\EE\biggl[\bigl|x_{t}^{\epsilon}-x_{t}\bigr|^{2}\biggr] & =\sup_{t\in\TT}\EE\biggl[\biggl|\int_{0}^{t}\int_{\AA}a\Bigl((1-\epsilon)\pi_{t}+\epsilon q_{t}\Bigr)(\D a)\,\D w_{s}\biggr|^{2}\biggr]\\
 & =\epsilon^{2}\sup_{t\in\TT}\EE\biggl[\biggl|\int_{0}^{t}\int_{\AA}aq_{t}(\D a)\,\D w_{s}\biggr|^{2}\biggr]\leq4\epsilon^{2}\EE\biggl[\int_{0}^{T}\biggl(\int_{\AA}aq_{t}(\D a)\biggr)^{2}\,\D s\biggr]\\
 & \leq4T\epsilon^{2}.
\end{align*}
\end{proof}
\begin{proof}[Proof of Proposition~\ref{prop:basic-existence-uniqueness}]
The drift and diffusion functions $b$ and $\sigma$ are by Assumption~\ref{assu:sde-baseline}(\emph{iii})
$L$-Lipschitz. In addition, by using the growth conditions of Assumption~\ref{assu:sde-baseline}(\emph{ii}),
they satisfy the following boundedness conditions:
\begin{align*}
\EE\left[\left(\int_{0}^{T}\left|\int_{\AA}b(s,0,a)\pi_{s}(\D a)\right|\,\D s\right)^{\bar{p}}\right] & \leq\EE\left[\left(\int_{0}^{T}\int_{\AA}L\left(1+\left|a\right|^{\pbar_{2}}\right)\pi_{s}(\D a)\,\D s\right)^{\bar{p}}\right]\\
 & \leq2^{\bar{p}-1}L^{\bar{p}}T^{\bar{p}}\EE\left[\sup_{t\in\TT}\int_{\AA}\left(1+\left|a\right|^{\pbar\pbar_{2}}\right)\pi_{t}(\D a)\right]\\
 & <\infty,\\
\EE\left[\left(\int_{0}^{T}\left|\int_{\AA}\sigma(s,0,a)\pi_{s}(\D a)\right|^{2}\,\D s\right)^{\bar{p}/2}\right] & \leq\EE\left[\left(\int_{0}^{T}\left(\int_{\AA}L\left(1+\left|a\right|^{\pbar_{2}}\right)\pi_{s}(\D a)\right)^{2}\,\D s\right)^{\bar{p}/2}\right]\\
 & \leq2^{\pbar-1}L^{\bar{p}}T^{\bar{p}}\EE\left[\sup_{t\in\TT}\left(\int_{\AA}\left(1+\left|a\right|^{\bar{p}\pbar_{2}}\right)\pi_{t}(\D a)\right)\right]\\
 & <\infty,
\end{align*}
where we have used the $\pbar_{3}$-admissibility of the control,
that is, Assumption~\ref{assu:sde-baseline}(\emph{vii}) and Eq.~(\ref{eq:admissible}),
and the $\pbar_{2}$ upper bound of Eq.~(\ref{eq:pbar-bounds}),
$\pbar\pbar_{2}\leq\pbar_{3}$. So being, \cite[Theorem 3.17]{Pardoux2014}
states that a unique strong solution $x=(x_{t})_{t\in\TT}\in\Ss_{\Ff}^{\pbar}(\Omega;\XX)$
of Eq.~(\ref{eq:vague-sde}) exists.

We can now estimate the costs using the growth conditions of Assumption~\ref{assu:sde-baseline}(\emph{iv})
\begin{align*}
\EE\left[\left|\int_{0}^{T}\int_{\AA}c(s,x_{s}^{\pi},a)\pi_{s}(\D a)\right|^{p}\right] & \leq\EE\left[T^{p-1}\int_{0}^{T}\int_{\AA}\left|c(s,x_{s}^{\pi},a)\right|^{p}\pi_{s}(\D a)\,\D s\right]\\
 & \leq\EE\left[T^{p-1}\int_{0}^{T}\int_{\AA}L^{p}\left(1+\left|x_{s}^{\pi}\right|^{p_{1}}+\left|a\right|^{p_{2}}\right)^{p}\pi_{s}(\D a)\,\D s\right]\\
 & \leq3^{p-1}L^{p}T^{p-1}\EE\left[\int_{0}^{T}\int_{\AA}\left(1+\left|x_{s}^{\pi}\right|^{pp_{1}}+\left|a\right|^{pp_{2}}\right)\pi_{s}(\D a)\,\D s\right]\\
 & <\infty,\\
\EE\left[\left|g(x_{T}^{\pi})\right|^{p}\right] & \leq L^{p}\EE\left[\left(1+\left|x_{T}^{\pi}\right|^{p_{1}}\right)^{p}\right]\\
 & \leq2^{p-1}L^{p}\EE\left[1+\left|x_{T}^{\pi}\right|^{pp_{1}}\right]\\
 & <\infty.
\end{align*}
To reach the final inequalities, we have used Eq.~(\ref{eq:main-deriv-bounds})
giving $pp_{1}<\pbar-p<\pbar$ (clearly the weaker assumption $p_{1}<\pbar/p$
would have sufficed, but the stronger form shall be used later), and
additionally using Eq.~(\ref{eq:p-less-pbar}), $pp_{2}<\pbar_{3}$,
so that the finiteness of the terms is implied by $x\in\Ss_{\Ff}^{\pbar}(\Omega;\XX)$
and the $\pbar_{3}$-admissibility of the control. We see that $C^{\pi}\in\Ll^{p}(\Omega;\RR)$,
and the proof of the first part is complete.
\end{proof}

\subsection{Proofs for Section 3}
\begin{proof}[Proof of Proposition~\ref{prop:L-deriv}]
This result is proven in \cite[Proposition 5.25]{Carmona2018_I}
for the case of $p=2$; here we are merely pointing out that the statement
naturally holds also in the ``smaller'' spaces $\Ll^{p}(\Omega;\RR^{n})$,
$p\in(2,\infty]$. Let $\psi$ be an $\Ll^{p}(\Omega)$-representation
of $\phi$, whose Fréchet derivative is continuous. Since the embedding
of $\Ll^{p}(\Omega;\RR^{n})$ into $\Ll^{2}(\Omega;\RR^{n})$ is continuous,
the Fréchet derivative is continuous on $\Ll^{2}(\Omega;\RR^{n})$
as well. Therefore, there is an almost surely unique $\Ll$-derivative
$f$ such that $Y=(\Omega\ni\omega\to f(X(\omega)))\in\Ll^{2}(\Omega;\RR^{n})$.
As an element of $\Ll^{2}(\Omega;\RR^{n})$, $Y$ is also in $\Ll^{q}(\Omega;\RR^{n})$,
$q=p/(p-1)$.
\end{proof}

\subsection{\label{subsec:main-proofs}Proofs for Section 4}

For the detailed proofs, we need to extend our notations somewhat.
Let $n,m\in\NN$ and $k_{i}\in\NN$ for all $i\in\{1,\ldots,m\}$.
For all differentiable functions $f:\RR^{n}\to\RR^{k_{1}\times\cdots\times k_{m}}$,
we define $\nabla f:\RR^{n}\to\RR^{k_{1}\times\cdots\times k_{m}\times n}$
so that $(\nabla f(x))_{i_{1},\ldots,i_{m},j}\deq\partial f_{i_{1},\ldots,i_{m}}(x)/\partial x_{j}$
for all $x\in\RR^{n}$, $i_{\ell}\in\{1,\ldots,k_{\ell}\}$, $\ell\in\{1,\ldots,m\}$. 

Let $N,M\in\NN$, and $n_{i},m_{j}\in\NN$ for all $i\in\{1,\ldots,N\}$,
$j\in\{1,\ldots,M\}$. Let $U\in\RR^{n_{1}\times\cdots\times n_{N}}$
and $V\in\RR^{m_{1}\times\cdots\times m_{M}}$. The arrays $UV\in\RR^{n_{1}\times\cdots\times n_{N-1}\times m_{2}\times\cdots\times m_{M}}$
and $U\cddot V\in\RR^{n_{1}\times\cdots\times n_{N-2}\times m_{3}\times\cdots\times m_{M}}$
are defined so that

\begin{gather*}
(UV)_{i_{1},\ldots,i_{N-1},j_{2},\ldots,j_{M}}\deq\sum_{k=1}^{m_{1}}U_{i_{1},\ldots,i_{N-1},k}V_{k,j_{2},\ldots,j_{M}},\\
(U\cddot V)_{i_{1},\ldots,i_{N-2},j_{3},\ldots,j_{M}}\deq\sum_{k_{1}=1}^{m_{1}}\sum_{k_{2}=1}^{m_{2}}U_{i_{1},\ldots,i_{N-2},k_{2},k_{1}}V_{k_{1},k_{2}j_{3},\ldots,j_{M},}
\end{gather*}
for all $i_{\ell}\in\{1,\ldots,n_{\ell}\}$, $\ell\in\{1,\ldots,N-1\}$
and $j_{\ell}\in\{1,\ldots,m_{\ell}\}$, $\ell\in\{2,\ldots,M\}$,
where in the former definition $n_{N}=m_{1}$ and in the latter $n_{N}=m_{1}$
and $n_{N-1}=m_{2}$. In addition, for all $X\in\RR^{n_{N-1}}$, we
define $U\cdot X\in\RR^{n_{1}\times\cdots\times n_{N-2}\times n_{N}}$
as such that
\begin{gather*}
(U\cdot X)_{i_{1},\ldots,i_{N-2},i_{N},}\deq\sum_{k=1}^{m_{1}}U_{i_{1},\ldots,i_{N-2},k,i_{N}}X_{k},
\end{gather*}
for all $i_{\ell}\in\{1,\ldots,n_{\ell}\}$, $\ell\in\{1,\ldots,N-2,N\}$.

We will also repeatedly use the following identity and estimates:
\begin{subequations}
\begin{gather}
f(x)-f(y)=\int_{0}^{1}\nabla f\left((1-\lambda)x+\lambda y\right)(x-y)\,\D\lambda\qquad\forall x,y\in\RR^{n},\,f\in\Cc^{(2)}(\RR^{n}),\,n\in\NN,\label{eq:lambda-diff}\\
\begin{gathered}\int_{0}^{1}\left|(1-\lambda)x+\lambda y\right|^{\alpha}\D\lambda\leq\ell_{\gamma}\left[\bigl(\left|x\right|^{\gamma}+\left|y\right|^{\gamma}\bigr)\vee\bigl(\left|x+y\right|^{\gamma}\bigr)\right]\qquad\forall x,y\in\RR^{n},\,n\in\NN,\,\gamma\in\RR_{\geq0},\\
\ell_{\gamma}\deq2^{-\gamma}\vee\frac{2^{\gamma-1}}{\gamma+1}\quad\forall\gamma\in\RR_{\geq0},
\end{gathered}
\label{eq:cvx-ccv-lambda}\\
\left(\sum_{i=1}^{n}x_{i}\right)^{\gamma}\leq n^{\gamma-1}\sum_{i=1}^{n}x_{i}^{\gamma}\qquad\forall(x_{1},\ldots,x_{n})\in\RR_{\geq0}^{n},\,n\in\NN,\,\gamma\geq1,\label{eq:jensen-sum}
\end{gather}
\end{subequations}
In addition, we shall frequently apply the Burkholder-Davis-Gundy
(BDG) inequality (see e.g. \cite[Theorem 1.76]{Pardoux2014}), and
we will use $C_{r}$, $r\in[1,\infty)$, to denote the constant in
the upper bound given in the inequality.
\begin{proof}[Proof of Proposition~\ref{lem:join-spaces}]
Let $\tilde{\Omega}^{0}\deq\Omega\times\Omega^{\prime}$, $\tilde{\Sigma}^{0}\deq\Sigma\times\Sigma^{\prime}$,
$\tilde{\Ff}^{0}\deq\Ff\times\Ff^{\prime}$, and $\tilde{\PP}^{0}\deq\PP\times\PP^{\prime}$.
The filtered probability space $(\tilde{\Omega},\tilde{\Sigma},\tilde{\Ff},\tilde{\PP})$
is then constructed from $(\tilde{\Omega}^{0},\tilde{\Sigma}^{0},\tilde{\Ff}^{0},\tilde{\PP}^{0})$
by conditioning on the event that the paths of the Brownian motions
$w$ and $w^{\prime}$ are the same: We set $\tilde{\Omega}\deq\{(\omega,\omega^{\prime})\in\tilde{\Omega}^{0}:w(\omega)=w^{\prime}(\omega^{\prime})\}$,
$\tilde{\PP}\deq\tilde{\PP}^{0}(\cdot\mid\tilde{\Omega})$, $\tilde{\Sigma}\deq\{\Gamma\cap\tilde{\Omega}:\Gamma\in\tilde{\Sigma}^{0}\}$,
and $\tilde{\Ff}_{t}\deq\{\Gamma\cap\tilde{\Omega}:\Gamma\in\tilde{\Ff}_{t}^{0}\}$
for all $t\in\TT$. We can then define the processes $\tilde{\pi}(\omega,\omega^{\prime})\deq\pi(\omega)$
and $\tilde{\pi}^{\prime}(\omega,\omega^{\prime})\deq\pi^{\prime}(\omega^{\prime})$
for all $(\omega,\omega^{\prime})\in\tilde{\Omega}$. These are $\tilde{\Ff}$-progressive
by virtue of $\pi$ and $\pi^{\prime}$ being $\Ff$- and $\Ff^{\prime}$-progressive,
respectively. In addition, we can define the $\tilde{\Ff}$-Brownian
motion $\tilde{w}\deq(\tilde{w}_{t})_{t\in\TT}$ as $\tilde{w}_{t}(\omega,\omega^{\prime})\deq w_{t}(\omega)$
for all $(\omega,\omega^{\prime})\in\tilde{\Omega}$. The drift and
diffusion coefficients defined by Eq.~(\ref{eq:sde-joint}) satisfy
Assumptions~\ref{assu:sde-baseline} and by \ref{prop:basic-existence-uniqueness},
a solution $(\tilde{x}_{t},\tilde{x}_{t}^{\prime})_{t\in\TT}$ exists
on $(\tilde{\Omega},\tilde{\Sigma},\tilde{\Ff},\tilde{\PP})$. The
$\Omega$ and $\Omega^{\prime}$ marginals of $\tilde{\PP}$ are again
$\PP$ and $\PP^{\prime}$, and hence the laws of the state space
and control processes on the extended space agree with those on the
original spaces.
\end{proof}
\begin{proof}[Proof of Lemma~\ref{lem:var-vanish}]
We estimate the distance between the original and the perturbed process
as follows. Let $T_{0}\in\TT$ be for now arbitrary. Using the triangle
inequality, the elementary estimate of Eq.~(\ref{eq:jensen-sum}),
and Jensen's inequality,
\begin{align}
\sup_{t\in[0,T_{0}]}\Bigl|x_{t}^{\pi(\alpha,q)}-x_{t}^{\pi}\Bigr|^{\pbar} & \leq\sup_{t\in[0,T_{0}]}\biggr[\left|\int_{0}^{t}\left[b(s,x_{s}^{\pi(\alpha,q)},\pi_{s}(\alpha,q))-b(s,x_{s}^{\pi},\pi_{s})\right]\D s\right|\nonumber \\
 & \qquad\qquad+\left|\int_{0}^{t}\left[\sigma(s,x_{s}^{\pi(\alpha,q)},\pi_{s}(\alpha,q))-\sigma(s,x_{s}^{\pi},\pi_{s})\right]\D w_{s}\right|\biggr]^{\pbar}\nonumber \\
 & =\sup_{t\in[0,T_{0}]}\biggr[\biggl|\int_{0}^{t}\Bigl[b(s,x_{s}^{\pi(\alpha,q)},\pi_{s}(\alpha,q))-b(s,x_{s}^{\pi},\pi_{s}(\alpha,q))\nonumber \\
 & \qquad\qquad\qquad+b(s,x_{s}^{\pi},\pi_{s}(\alpha,q))-b(s,x_{s}^{\pi},\pi_{s})\Bigr]\D s\biggr|\nonumber \\
 & \qquad\qquad+\biggl|\int_{0}^{t}\Bigl[\sigma(s,x_{s}^{\pi(\alpha,q)},\pi_{s}(\alpha,q))-\sigma(s,x_{s}^{\pi},\pi_{s}(\alpha,q))\nonumber \\
 & \qquad\qquad\qquad+\sigma(s,x_{s}^{\pi},\pi_{s}(\alpha,q))-\sigma(s,x_{s}^{\pi},\pi_{s})\Bigr]\D s\biggr|\biggr]^{\pbar}\nonumber \\
 & \leq4^{\pbar-1}\sup_{t\in[0,T_{0}]}\biggr[\biggl|\int_{0}^{t}\Bigl[b(s,x_{s}^{\pi(\alpha,q)},\pi_{s}(\alpha,q))-b(s,x_{s}^{\pi},\pi_{s}(\alpha,q))\Bigr]\D s\biggr|^{\pbar}\nonumber \\
 & \qquad\qquad\qquad+\biggl|\int_{0}^{t}\Bigl[b(s,x_{s}^{\pi},\pi_{s}(\alpha,q))-b(s,x_{s}^{\pi},\pi_{s})\Bigr]\D s\biggr|^{\pbar}\nonumber \\
 & \qquad\qquad\qquad+\biggl|\int_{0}^{t}\Bigl[\sigma(s,x_{s}^{\pi(\alpha,q)},\pi_{s}(\alpha,q))-\sigma(s,x_{s}^{\pi},\pi_{s}(\alpha,q))\Bigr]\D w_{s}\biggr|^{\pbar}\nonumber \\
 & \qquad\qquad\qquad+\biggl|\int_{0}^{t}\Bigl[\sigma(s,x_{s}^{\pi},\pi_{s}(\alpha,q))-\sigma(s,x_{s}^{\pi},\pi_{s})\Bigr]\D w_{s}\biggr|^{\pbar}\biggr]\nonumber \\
 & \leq4^{\pbar-1}\biggr[T^{\pbar-1}\int_{0}^{T_{0}}\Bigl|b(s,x_{s}^{\pi(\alpha,q)},\pi_{s}(\alpha,q))-b(s,x_{s}^{\pi},\pi_{s}(\alpha,q))\Bigr|^{\pbar}\D s\nonumber \\
 & \qquad\qquad+T^{\pbar-1}\int_{0}^{T_{0}}\Bigl|b(s,x_{s}^{\pi},\pi_{s}(\alpha,q))-b(s,x_{s}^{\pi},\pi_{s})\Bigr|^{\pbar}\D s\nonumber \\
 & \qquad\qquad+\sup_{t\in[0,T_{0}]}\biggl|\int_{0}^{t}\Bigl[\sigma(s,x_{s}^{\pi(\alpha,q)},\pi_{s}(\alpha,q))-\sigma(s,x_{s}^{\pi},\pi_{s}(\alpha,q))\Bigr]\D w_{s}\biggr|^{\pbar}\nonumber \\
 & \qquad\qquad+\sup_{t\in[0,T_{0}]}\biggl|\int_{0}^{t}\Bigl[\sigma(s,x_{s}^{\pi},\pi_{s}(\alpha,q))-\sigma(s,x_{s}^{\pi},\pi_{s})\Bigr]\D w_{s}\biggr|^{\pbar}\biggr].\label{eq:temp-sup}
\end{align}
We estimate individually each of the terms on the right of the above
inequality. Starting with the diffusion terms, by the BDG inequality
and by using the definition of $\pi_{t}(\alpha,q)$ as $\pi_{t}+\alpha(q_{t}-\pi_{t})$,
for all $t\in\TT$ along with Assumption~\ref{assu:sde-baseline}(\emph{i}),
i.e. the growth condition on $\sigma$, and similar estimates as above,
\begin{align}
\EE\biggl[\sup_{t\in[0,T_{0}]}\biggl|\int_{0}^{t}\Bigl[\sigma(s,x_{s}^{\pi},\pi_{s}(\alpha,q)) & -\sigma(s,x_{s}^{\pi},\pi_{s})\Bigr]\,\D w_{s}\biggr|^{\pbar}\biggr]\nonumber \\
 & \leq C_{\pbar}\EE\biggl[\biggl(\int_{0}^{T_{0}}\Bigl|\sigma(s,x_{s}^{\pi},\pi_{s}(\alpha,q))-\sigma(s,x_{s}^{\pi},\pi_{s})\Bigr|^{2}\D s\biggr)^{\pbar/2}\biggr]\nonumber \\
 & =\alpha^{\pbar}C_{\pbar}\EE\biggl[\biggl(\int_{0}^{T_{0}}\bigl|\sigma(s,x_{s}^{\pi},\pi_{s}-q_{s})\bigr|^{2}\D s\biggr)^{\pbar/2}\biggr]\nonumber \\
 & \leq\alpha^{\pbar}C_{\pbar}\EE\biggl[\biggl(\int_{0}^{T_{0}}\Bigl[\bigl|\sigma(s,x_{s}^{\pi},\pi_{s})\bigr|+\bigl|\sigma(s,x_{s}^{\pi},q_{s})\bigr|\Bigr]^{2}\D s\biggr)^{\pbar/2}\biggr]\nonumber \\
 & \leq\alpha^{\pbar}C_{\pbar}\EE\biggl[\biggl(\int_{0}^{T_{0}}4L^{2}\biggl[1+|x_{s}^{\pi}|^{\pbar_{1}}+\int_{\AA}|a|^{\pbar_{2}}\left(\frac{\pi_{s}+q_{s}}{2}\right)(\D a)\biggr]^{2}\D s\biggr)^{\pbar/2}\biggr]\nonumber \\
 & \leq(2\alpha L)^{\pbar}T^{\pbar/2}C_{\pbar}\EE\biggl[\sup_{s\in[0,T_{0}]}\biggl(1+|x_{s}^{\pi}|^{\pbar_{1}}+\int_{\AA}|a|^{\pbar_{2}}\left(\frac{\pi_{s}+q_{s}}{2}\right)(\D a)\biggr)^{\pbar}\biggr]\nonumber \\
 & \leq3^{\pbar-1}(2\alpha L)^{\pbar}T^{\pbar/2}C_{\pbar}\EE\biggl[\sup_{s\in[0,T_{0}]}\biggl(1+|x_{s}^{\pi}|^{\pbar\pbar_{1}}+\int_{\AA}|a|^{\pbar\pbar_{2}}\left(\frac{\pi_{s}+q_{s}}{2}\right)(\D a)\biggr)\biggr].\label{eq:temp-sup-a}
\end{align}
We set this inequality aside to be used a moment later, and consider
next the penultimate term in Eq.~(\ref{eq:temp-sup}). Using again
the BDG inequality and Assumption~\ref{assu:sde-baseline}(\emph{iii}),
and the inequality
\begin{align*}
\left(\int_{0}^{t}|f_{s}|^{2}\D s\right)^{\gamma} & \leq\biggl(\sup_{s\in[0,t]}|f_{s}|\biggr)^{\gamma}\biggl(\int_{0}^{t}|f_{s}|\,\D s\biggr)^{\gamma}\\
 & \leq\frac{1}{2}\left[K\sup_{s\in[0,t]}|f_{s}|^{2\gamma}+\frac{1}{K}t^{2\gamma-1}\int_{0}^{t}\sup_{r\in[0,s]}|f_{r}|^{2\gamma}\D s\right]
\end{align*}
for all measurable $f$, $\gamma\geq1$, $t\in\TT$, and any $K>0$
(pre-emptively, we select $K=4^{-\pbar+1}L^{-\pbar}C_{\pbar}^{-1}$),
we obtain
\begin{align}
\EE\biggl[\sup_{t\in[0,T_{0}]} & \biggl|\int_{0}^{t}\Bigl[\sigma(s,x_{s}^{\pi(\alpha,q)},\pi_{s}(\alpha,q))-\sigma(s,x_{s}^{\pi},\pi_{s}(\alpha,q))\Bigr]\D w_{s}\biggr|^{\pbar}\biggr]\nonumber \\
 & \leq C_{\pbar}\EE\biggl[\biggl(\int_{0}^{T_{0}}\Bigl|\sigma(s,x_{s}^{\pi(\alpha,q)},\pi_{s}(\alpha,q))-\sigma(s,x_{s}^{\pi},\pi_{s}(\alpha,q))\Bigr|^{2}\D s\biggr)^{\pbar/2}\biggr]\nonumber \\
 & \leq L^{\pbar}C_{\pbar}\EE\biggl[\biggl(\int_{0}^{T_{0}}\bigl|x_{s}^{\pi(\alpha,q)}-x_{s}^{\pi}\bigr|^{2}\D s\biggr)^{\pbar/2}\biggr]\nonumber \\
 & \leq\frac{1}{2}\EE\biggl[4^{-\pbar+1}\sup_{t\in[0,T_{0}]}\bigl|x_{t}^{\pi(\alpha,q)}-x_{t}^{\pi}\bigr|^{\pbar}+4^{\pbar-1}L^{2\pbar}C_{\pbar}^{2}T^{\pbar-1}\int_{0}^{T_{0}}\sup_{s\in[0,t]}\bigl|x_{s}^{\pi(\alpha,q)}-x_{s}^{\pi}\bigr|^{\pbar}\D t\biggr].\label{eq:temp-sup-b}
\end{align}
The drift terms in Eq.~(\ref{eq:temp-sup}) are estimated similarly,
\begin{align}
\EE\biggl[\int_{0}^{T_{0}}\Bigl|b(s,x_{s}^{\pi},\pi_{s}(\alpha,q)) & -b(s,x_{s}^{\pi},\pi_{s})\Bigr|^{\pbar}\D s\biggr]\nonumber \\
 & =\alpha^{\pbar}\EE\biggl[\int_{0}^{T_{0}}\bigl|b(s,x_{s}^{\pi},\pi_{s}-q_{s})\bigr|^{\pbar}\D s\biggr]\nonumber \\
 & \leq\alpha^{\pbar}\EE\biggl[\int_{0}^{T_{0}}\Bigl[\bigl|b(s,x_{s}^{\pi},\pi_{s})\bigr|+\bigl|b(s,x_{s}^{\pi},q_{s})\bigr|\Bigr]^{\pbar}\D s\biggr]\nonumber \\
 & \leq(2L\alpha)^{\pbar}\EE\biggl[\int_{0}^{T_{0}}\biggl[1+|x_{s}^{\pi}|^{\pbar_{1}}+\int_{\AA}|a|^{\pbar_{2}}\left(\frac{\pi_{s}+q_{s}}{2}\right)(\D a)\biggr]^{\pbar}\D s\biggr]\nonumber \\
 & \leq3^{\pbar-1}(2L\alpha)^{\pbar}\EE\biggl[\int_{0}^{T}\biggl[1+|x_{s}^{\pi}|^{\pbar\pbar_{1}}+\int_{\AA}|a|^{\pbar\pbar_{2}}\left(\frac{\pi_{s}+q_{s}}{2}\right)(\D a)\biggr]\D s\biggr]\nonumber \\
 & \leq3^{\pbar-1}(2L\alpha)^{\pbar}T\EE\biggl[\sup_{t\in[0,T]}\biggl(1+|x_{t}^{\pi}|^{\pbar\pbar_{1}}+\int_{\AA}|a|^{\pbar\pbar_{2}}\left(\frac{\pi_{t}+q_{t}}{2}\right)(\D a)\biggr)\biggr],\label{eq:temp-sup-c}
\end{align}
and
\begin{align}
\EE\biggl[\int_{0}^{T_{0}}\Bigl|b(s,x_{s}^{\pi(\alpha,q)},\pi_{s}(\alpha,q)) & -b(s,x_{s}^{\pi},\pi_{s}(\alpha,q))\Bigr|^{\pbar}\D s\biggr]\nonumber \\
 & \leq L^{\pbar}\EE\biggl[\int_{0}^{T_{0}}\bigl|x_{s}^{\pi(\alpha,q)}-x_{s}^{\pi}\bigr|^{\pbar}\D s\biggr]\nonumber \\
 & \leq L^{\pbar}\EE\biggl[\int_{0}^{T_{0}}\sup_{s\in[0,t]}\bigl|x_{s}^{\pi(\alpha,q)}-x_{s}^{\pi}\bigr|^{\pbar}\D t\biggr].\label{eq:temp-sup-d}
\end{align}
Collecting estimates of Eqs.~(\ref{eq:temp-sup-a},~\ref{eq:temp-sup-b},~\ref{eq:temp-sup-c},~\ref{eq:temp-sup-d})
and applying them to Eq.~(\ref{eq:temp-sup}), and after rearranging,
we get,
\begin{align}
\EE\biggl[\sup_{t\in[0,T_{0}]} & |x_{t}^{\pi(\alpha,q)}-x_{t}^{\pi}|^{\pbar}\biggr]\nonumber \\
 & \leq\underbrace{4^{\pbar-1}T^{\pbar-1}L^{\pbar}\left(2+4^{\pbar-1}L^{\pbar}C_{\pbar}^{2}\right)}_{\eqd S_{1}}\EE\biggl[\int_{0}^{T_{0}}\sup_{s\in[0,t]}\Bigl|x_{s}^{\pi(\alpha,q)}-x_{s}^{\pi}\Bigr|^{\pbar}\D t\biggr]\nonumber \\
 & \qquad+\underbrace{2\cdot12^{\pbar-1}(2L\alpha)^{\pbar}\left(T^{\pbar}+T^{\pbar/2}C_{\pbar}\right)\EE\biggl[\sup_{t\in[0,T]}\biggl(1+|x_{t}^{\pi}|^{\pbar\pbar_{1}}+\int_{\AA}|a|^{\pbar\pbar_{2}}\left(\frac{\pi_{t}+q_{t}}{2}\right)(\D a)\biggr)\biggr]}_{\eqd\alpha^{\pbar}S_{2}}.\label{eq:temp-pre-gronwall}
\end{align}
The expectation in $S_{2}$ is finite, since by Proposition~\ref{prop:basic-existence-uniqueness}
$x^{\pi}\in\Ss_{\Ff}^{\pbar}(\Omega;\XX)$ and $\pbar\pbar_{1}\leq\pbar$
(Assumptions~\ref{assu:sde-baseline} state $\pbar_{1}\in[0,1]$),
and $\pbar\pbar_{2}\leq\pbar_{3}$ by Eq.~(\ref{eq:pbar-bounds})
and the control is $\pbar_{3}$-admissible. Using Grönwall's inequality
(see e.g. \cite[Corollary 6.60]{Pardoux2014}), we find that
\begin{gather}
\EE\left[\sup_{t\in[0,T_{0}]}|x_{t}^{\pi(\alpha,q)}-x_{t}^{\pi}|^{\pbar}\right]\leq\alpha^{\pbar}S_{2}\E^{S_{1}T_{0}}\quad\forall T_{0}\in\TT,\label{eq:temp-gronwall}
\end{gather}
from where Eq.~(\ref{eq:sup-pbar}) follows.

To prove Eq.~(\ref{eq:sup-p}), consider the running cost processes,
$(x_{t}^{\prime\,\pi(\alpha,q)})_{t\in\TT}$ and $(x_{t}^{\prime\,\pi})_{t\in\TT}$.
We have that
\begin{align}
\EE\biggl[\sup_{t\in\TT}|x_{t}^{\prime\,\pi(\alpha,q)}-x_{t}^{\prime\,\pi}|^{p}\biggr] & =\EE\biggl[\sup_{t\in\TT}\biggl|\int_{0}^{t}\Bigl[c(s,x_{s}^{\pi(\alpha,q)},\pi_{s}(\alpha,q))-c(s,x_{s}^{\pi},\pi_{s})\Bigr]\,\D s\biggr|^{p}\biggr]\nonumber \\
 & =\EE\biggl[\sup_{t\in\TT}\biggl|\int_{0}^{t}\Bigl[c(s,x_{s}^{\pi(\alpha,q)},\pi_{s}(\alpha,q))-c(s,x_{s}^{\pi},\pi_{s}(\alpha,q))\nonumber \\
 & \qquad\qquad\qquad+c(s,x_{s}^{\pi},\pi_{s}(\alpha,q))-c(s,x_{s}^{\pi},\pi_{s})\Bigr]\,\D s\biggr|^{p}\biggr]\nonumber \\
 & \leq2^{p-1}\EE\biggl[\biggl(\int_{0}^{T}\Bigl|c(s,x_{s}^{\pi(\alpha,q)},\pi_{s}(\alpha,q))-c(s,x_{s}^{\pi},\pi_{s}(\alpha,q))\Bigr|\,\D s\biggr)^{p}\biggr]\nonumber \\
 & \qquad+2^{p-1}\EE\biggl[\biggl(\int_{0}^{T}\Bigr|c(s,x_{s}^{\pi},\pi_{s}(\alpha,q))-c(s,x_{s}^{\pi},\pi_{s})\Bigl|\,\D s\biggr)^{p}\biggr].\label{eq:temp-p-sup}
\end{align}
The latter term is $\Oo(\alpha^{p})$, which can be verified using
similar estimates and assumptions as before,
\begin{align*}
\EE\biggl[\biggl(\int_{0}^{T}\Bigr|c(s,x_{s}^{\pi},\pi_{s}(\alpha,q)) & -c(s,x_{s}^{\pi},\pi_{s})\Bigl|\biggr]\,\D s\biggr)^{p}\biggr]\\
 & \leq\alpha^{p}\EE\biggl[\biggl(\int_{0}^{T}\Bigr|c(s,x_{s}^{\pi},\pi_{s}-q_{s})\Bigl|\,\D s\biggr)^{p}\biggr]\\
 & \leq\alpha^{p}\EE\biggl[\biggl(\int_{0}^{T}\biggl[\Bigr|c(s,x_{s}^{\pi},\pi_{s})\Bigl|+\Bigr|c(s,x_{s}^{\pi},q_{s})\Bigl|\biggr]\,\D s\biggr)^{p}\biggr]\\
 & \leq3^{p-1}2^{p}\alpha^{p}L^{p}\EE\biggl[\int_{0}^{T}\biggl[1+\left|x_{s}^{\pi}\right|^{pp_{1}}+\int_{\AA}\left|a\right|^{pp_{2}}\left(\frac{\pi_{s}+q_{s}}{2}\right)(\D a)\biggr]\,\D s\biggr)\biggr]\\
 & \in\Oo(\alpha^{p}).
\end{align*}
The finiteness of the expectation on the second to last line now follows
from the fact $x^{\pi}\in\Ss_{\Ff}^{\pbar}$ and $\pbar_{3}$-admissibility
of the control, when one notes that by Eqs.~(\ref{eq:p-less-pbar},~\ref{eq:main-deriv-bounds}),
$pp_{i}<\pbar-p<\pbar\leq\pbar_{3}$, $i\in\{1,2\}$.

To complete the proof of Eq.~(\ref{eq:sup-p}), it then remains to
show that the first term on the right hand side of Eq.~(\ref{eq:temp-p-sup})
is also $\Oo(\alpha^{p})$. The boundedness of $\nabla c$, i.e. Assumption~\ref{assu:sde-baseline}(\emph{v})
implies that 
\begin{align*}
\left|c(t,x_{2},a)-c(t,x_{1},a)\right| & =\left|\int_{0}^{1}\nabla_{\XX}c\bigl(t,x_{1}+\lambda(x_{2}-x_{1}),a\bigr)(x_{2}-x_{1})\,\D\lambda\right|\\
 & \leq\int_{0}^{1}\left|\nabla_{\XX}c\bigl(t,x_{1}+\lambda(x_{2}-x_{1}),a\bigr)\right|\left|x_{2}-x_{1}\right|\,\D\lambda\\
 & \leq\int_{0}^{1}L\left(1+\left|(1-\lambda)x_{1}+\lambda x_{2}\right|^{p_{1}^{\prime}}+|a|^{p_{2}^{\prime}}\right)\left|x_{2}-x_{1}\right|\,\D\lambda\\
 & =L\left(1+\ell_{p_{1}^{\prime}}\left|x_{1}+x_{2}\right|^{p_{1}^{\prime}}\vee\left(\left|x_{1}\right|^{p_{1}^{\prime}}+\left|x_{2}\right|^{p_{1}^{\prime}}\right)+|a|^{p_{2}^{\prime}}\right)\left|x_{2}-x_{1}\right|,\\
 & \qquad\forall(t,x_{1},x_{2},a)\in\TT\times\XX\times\XX\times\AA,
\end{align*}
and where $\ell_{p_{1}^{\prime}}$ is defined in Eq.~(\ref{eq:cvx-ccv-lambda}).
Thus,
\begin{align}
\EE\biggl[\biggl(\int_{0}^{T}\Bigl|c(s,x_{s}^{\pi(\alpha,q)} & ,\pi_{s}(\alpha,q))-c(s,x_{s}^{\pi},\pi_{s}(\alpha,q))\Bigr|\,\D s\biggr)^{p}\biggr]\nonumber \\
 & =\EE\biggl[\biggl(\int_{0}^{T}\Bigl|\int_{\AA}\bigl[c(s,x_{s}^{\pi(\alpha,q)},a)-c(s,x_{s}^{\pi},a)\bigr]\pi_{s}(\alpha,q)(\D a)\Bigr|\,\D s\biggr)^{p}\biggr]\nonumber \\
 & \leq\EE\biggl[\biggl(\int_{0}^{T}L\biggl\{1+\ell_{p_{1}^{\prime}}\left|x_{s}^{\pi(\alpha,q)}+x_{s}^{\pi}\right|^{p_{1}^{\prime}}\vee\left(\left|x_{s}^{\pi(\alpha,q)}\right|^{p_{1}^{\prime}}+\left|x_{s}^{\pi}\right|^{p_{1}^{\prime}}\right)\nonumber \\
 & \qquad\qquad\qquad+\int_{\AA}|a|^{p_{2}^{\prime}}\pi_{s}(\alpha,q)(\D a)\biggr\}\left|x_{s}^{\pi(\alpha,q)}-x_{s}^{\pi}\right|\,\D s\biggr)^{p}\biggr]\nonumber \\
 & \leq L^{p}T^{p}\EE\biggl[\biggl(\sup_{t\in\TT}\biggl\{1+\ell_{p_{1}^{\prime}}\left|x_{t}^{\pi(\alpha,q)}+x_{t}^{\pi}\right|^{p_{1}^{\prime}}\vee\left(\left|x_{t}^{\pi(\alpha,q)}\right|^{p_{1}^{\prime}}+\left|x_{t}^{\pi}\right|^{p_{1}^{\prime}}\right)\nonumber \\
 & \qquad\qquad\qquad+\int_{\AA}|a|^{p_{2}^{\prime}}\pi_{t}(\alpha,q)(\D a)\biggr\}\left|x_{t}^{\pi(\alpha,q)}-x_{t}^{\pi}\right|\biggr)^{p}\biggr]\nonumber \\
 & \leq L^{p}T^{p}\Biggl\{\EE\biggl[\sup_{t\in\TT}\left|x_{t}^{\pi(\alpha,q)}-x_{t}^{\pi}\right|^{p}\biggr]\nonumber \\
 & \qquad\qquad+\ell_{p_{1}^{\prime}}^{p}\EE\biggl[\underbrace{\sup_{t\in\TT}\biggl(\left|x_{t}^{\pi(\alpha,q)}+x_{t}^{\pi}\right|^{p_{1}^{\prime}}\vee\left(\left|x_{t}^{\pi(\alpha,q)}\right|^{p_{1}^{\prime}}+\left|x_{t}^{\pi}\right|^{p_{1}^{\prime}}\right)\biggr)^{p}}_{\eqd Z_{1}}\sup_{t\in\TT}\left|x_{t}^{\pi(\alpha,q)}-x_{t}^{\pi}\right|^{p}\biggr]\nonumber \\
 & \qquad\qquad+\EE\biggl[\underbrace{\sup_{t\in\TT}\biggl(\int_{\AA}|a|^{p_{1}^{\prime}}\pi_{t}(\alpha,q)(\D a)\biggr)^{p}}_{\eqd Z_{2}}\sup_{t\in\TT}\left|x_{t}^{\pi(\alpha,q)}-x_{t}^{\pi}\right|^{p}\biggr]\Biggr\}.\label{eq:temp-zs}
\end{align}
The first term on the right is $o(\alpha^{p})$ by virtue of Eq.~(\ref{eq:sup-pbar})
which we established earlier, and the fact that $p<\pbar$ from Eq.~(\ref{eq:p-less-pbar}).
The remaining two are treated as follows. We define
\begin{gather*}
\zeta_{i}\deq\frac{\pbar}{p_{i}^{\prime}p}\quad\forall i\in\{1,2\}.
\end{gather*}
Using Eqs.~(\ref{eq:p-less-pbar},~\ref{eq:growth-deriv-relation},~\ref{eq:main-deriv-bounds}),
we have that
\begin{gather*}
\zeta_{i}\geq\frac{\pbar}{p_{i}p}>\frac{\pbar}{(\pbar/p-1)p}>\frac{\pbar}{\pbar-p}>1\quad\forall i\in\{1,2\},
\end{gather*}
and consequently,
\begin{gather*}
p<p\frac{\zeta_{i}}{\zeta_{i}-1}<p\frac{\pbar/(\pbar-p)}{\pbar/(\pbar-p)-1}=\pbar\quad\forall i\in\{1,2\}.
\end{gather*}
We can then use Hölder's inequality
\begin{gather}
\EE\biggl[Z_{i}\sup_{t\in\TT}\left|x_{s}^{\pi(\alpha,q)}-x_{s}^{\pi}\right|^{p}\biggr]\leq\EE\Bigl[Z_{i}^{\zeta_{i}}\Bigr]^{\frac{1}{\zeta_{i}}}\EE\biggl[\sup_{t\in\TT}\left|x_{t}^{\pi(\alpha,q)}-x_{t}^{\pi}\right|^{p\frac{\zeta_{i}}{\zeta_{i}-1}}\biggr]^{\frac{\zeta_{i}-1}{\zeta_{i}}}.\label{eq:z-holder-1}
\end{gather}
The latter factor is $\Oo(\alpha^{p})$, again by Eq.~(\ref{eq:sup-pbar}),
and since $p\zeta_{i}/(\zeta_{i}-1)<\pbar$, $i\in\{1,2\}$. We only
need to show $\EE[Z_{i}^{\zeta_{i}}]<\infty$, $i\in\{1,2\}$, which
is straight-forward (note that Eqs.~(\ref{eq:feasible-p}) also imply
$\pbar/p_{1}^{\prime}>1$ and $\pbar/p_{2}^{\prime}>1$):
\begin{align}
\EE\left[Z_{1}^{\zeta_{1}}\right] & =\EE\biggl[\sup_{t\in\TT}\biggl(\left|x_{s}^{\pi(\alpha,q)}+x_{s}^{\pi}\right|^{p_{1}^{\prime}}\vee\left(\left|x_{s}^{\pi(\alpha,q)}\right|^{p_{1}^{\prime}}+\left|x_{s}^{\pi}\right|^{p_{1}^{\prime}}\right)\biggr)^{\pbar/p_{1}^{\prime}}\biggr]\nonumber \\
 & \leq\EE\biggl[\sup_{t\in\TT}\biggl(\left|x_{s}^{\pi(\alpha,q)}+x_{s}^{\pi}\right|^{\pbar}\vee\left[2^{\pbar/p_{1}^{\prime}-1}\left(\left|x_{s}^{\pi(\alpha,q)}\right|^{\pbar}+\left|x_{s}^{\pi}\right|^{\pbar}\right)\right]\biggr)\biggr]\nonumber \\
 & <\infty,\label{eq:z-holder-2}\\
\EE\left[Z_{2}^{\zeta_{2}}\right] & \leq\EE\biggl[\sup_{t\in\TT}\biggl(\int_{\AA}|a|^{p_{2}^{\prime}}\pi_{s}(\alpha,q)(\D a)\biggr)^{\pbar/p_{2}^{\prime}}\biggr]\nonumber \\
 & \leq\EE\biggl[\sup_{t\in\TT}\int_{\AA}|a|^{\pbar}\pi_{s}(\alpha,q)(\D a)\biggr]\nonumber \\
 & <\infty,\label{eq:z-holder-3}
\end{align}
where we have as usual used the fact $x^{\pi},x^{\pi(\alpha,q)}\in\Ss_{\Ff}^{\pbar}$
and the $\pbar_{3}$-admissibility of the control. Therefore Eq.~(\ref{eq:sup-p})
holds.

Finally, we turn to Eq.~(\ref{eq:g-sup-p}). We have that
\begin{align}
\EE\biggl[\Bigl|g\bigl(x_{T}^{\pi(\alpha,q)}\bigr)-g\bigl(x_{T}^{\pi}\bigr)\Bigr|^{p}\biggr] & =\EE\biggl[\biggl|\int_{0}^{1}\nabla_{\XX}g(\lambda x_{T}^{\pi(\alpha,q)}+(1-\lambda)x_{T}^{\pi})(x_{T}^{\pi(\alpha,q)}-x_{T}^{\pi})\,\D\lambda\biggr|^{p}\biggr]\nonumber \\
 & \leq\EE\biggl[\biggl|\int_{0}^{1}L\Bigl(1+\Bigl|\lambda x_{T}^{\pi(\alpha,q)}+(1-\lambda)x_{T}^{\pi}\Bigr|^{p_{1}^{\prime}}\Bigr)\bigl|x_{T}^{\pi(\alpha,q)}-x_{T}^{\pi}\bigr|\,\D\lambda\biggr|^{p}\biggr]\nonumber \\
 & \leq2^{p-1}L^{p}\EE\biggl[\int_{0}^{1}\Bigl(1+\Bigl|\lambda x_{T}^{\pi(\alpha,q)}+(1-\lambda)x_{T}^{\pi}\Bigr|^{pp_{1}^{\prime}}\Bigr)\bigl|x_{T}^{\pi(\alpha,q)}-x_{T}^{\pi}\bigr|^{p}\,\D\lambda\biggr]\nonumber \\
 & \leq2^{p-1}L^{p}\EE\biggl[\ell_{pp_{1}^{\prime}}\biggl[\Bigl(\bigl|x_{T}^{\pi(\alpha,q)}\bigr|^{pp_{1}^{\prime}}+\bigl|x_{T}^{\pi}\bigr|^{pp_{1}^{\prime}}\Bigr)\vee\Bigl(\bigl|x_{T}^{\pi(\alpha,q)}+x_{T}^{\pi}\bigr|^{pp_{1}^{\prime}}\Bigr)\biggr]\bigl|x_{T}^{\pi(\alpha,q)}-x_{T}^{\pi}\bigr|^{p}\biggr]\nonumber \\
 & \leq2^{p-1}L^{p}\ell_{pp_{1}^{\prime}}\EE\biggl[\biggl[\Bigl(\bigl|x_{T}^{\pi(\alpha,q)}\bigr|^{pp_{1}^{\prime}}+\bigl|x_{T}^{\pi}\bigr|^{pp_{1}^{\prime}}\Bigr)\vee\Bigl(\bigl|x_{T}^{\pi(\alpha,q)}+x_{T}^{\pi}\bigr|^{pp_{1}^{\prime}}\Bigr)\biggr]^{\frac{\pbar}{\pbar-p}}\biggl]^{1-\frac{p}{\pbar}}\label{eq:holder-trick-x}\\
 & \qquad\qquad\times\EE\biggl[\bigl|x_{T}^{\pi(\alpha,q)}-x_{T}^{\pi}\bigr|^{\pbar}\biggr]^{\frac{p}{\pbar}}.\nonumber 
\end{align}
The latter expectation is in $\Oo(\alpha^{p})$ by Eq.~(\ref{eq:sup-pbar}),
and the former is finite since by Eqs.~(\ref{eq:growth-deriv-relation},~\ref{eq:main-deriv-bounds}),
\begin{gather*}
pp_{1}^{\prime}\frac{\pbar}{\pbar-p}<(\pbar-p)\frac{\pbar}{\pbar-p}=\pbar,
\end{gather*}
and $x^{\pi},x^{\pi(\alpha,q)}\in\Ss_{\Ff}^{\pbar}$. Eq.~(\ref{eq:g-sup-p})
holds, and the proof is complete.
\end{proof}
\begin{proof}[Proof of Lemma~\ref{lem:epsilon}]
The drift and diffusion coefficients appearing in Eq.~(\ref{eq:delta-solo})
are Lipschitz, since by Assumption~\ref{assu:sde-baseline}(\emph{iii})
the gradients of $b$ and $\sigma$ are bounded. In addition, the
terms in Eq.~(\ref{eq:delta-solo}) that do not depend on $\delta^{\pi}$
satisfy
\begin{align*}
\EE\biggl[\sup_{t\in\TT}\biggl|\int_{0}^{t}b(s,x_{s}^{\pi},q_{s}-\pi_{s})\,\D s & +\int_{0}^{t}\sigma(s,x_{s}^{\pi},q_{s}-\pi_{s})\D w_{s}\biggr|^{\pbar}\biggr]\\
 & \leq\EE\biggl[\sup_{t\in\TT}\biggl|-x_{t}^{\pi}+x_{0}^{\pi}+\int_{0}^{t}b(s,x_{s}^{\pi},q_{s})\,\D s+\int_{0}^{t}\sigma(s,x_{s}^{\pi},q_{s})\,\D w_{s}\biggr|^{\pbar}\biggr]\\
 & \leq3^{\pbar-1}\biggl\{\EE\biggl[\sup_{t\in\TT}\bigl|x_{t}^{\pi}-x_{0}^{\pi}\bigr|^{\pbar}\biggr]+\EE\biggl[\sup_{t\in\TT}\biggl|\int_{0}^{t}b(s,x_{s}^{\pi},q_{s})\,\D s\biggr|^{\pbar}\biggr]\\
 & \qquad\qquad+\EE\biggl[\sup_{t\in\TT}\biggl|\int_{0}^{t}\sigma(s,x_{s}^{\pi},q_{s})\,\D w_{s}\biggr|^{\pbar}\biggr]\biggr\}\\
 & \leq3^{\pbar-1}\biggl\{\EE\biggl[\sup_{t\in\TT}\bigl|x_{t}^{\pi}-x_{0}^{\pi}\bigr|^{\pbar}\biggr]\\
 & \qquad+3^{\pbar-1}L^{\pbar}\bigl(T^{\pbar}+C_{\pbar}T^{\pbar/2}\bigr)\EE\biggl[\sup_{t\in\TT}\Bigl(1+|x_{t}^{\pi}|^{\pbar\pbar_{1}}+\int_{\AA}|a|^{\pbar\pbar_{2}}q_{t}(\D a)\Bigr)\biggr]\biggr\}\\
 & <\infty,
\end{align*}
where we have used the growth conditions of Assumption~\ref{assu:sde-baseline}(\emph{ii}),
Eq.~(\ref{eq:pbar-bounds}), $x^{\pi}\in\Ss_{\Ff}^{\pbar}$, and
the $\pbar_{3}$-admissibility of the control. A unique strong solution
of Eq.~(\ref{eq:delta-solo}) now exists by e.g. \cite[Theorem 3.17]{Pardoux2014},
which also satisfies Eq.~(\ref{eq:delta-pbar-finite}).

To prove Eq.~(\ref{eq:delta-prime-p-finite}), let $r\in(1,\pbar)$
be to be determined. We use the growth conditions on $c$ and $\nabla_{\XX}c$,
Assumptions~\ref{assu:sde-baseline}(\emph{iv},~\emph{v}), to estimate
$\delta^{\prime\,\pi,q}$ as follows.
\begin{align*}
\EE\biggl[\sup_{t\in\TT}\bigl|\delta_{t}^{\prime\,\pi,q}\bigr|^{r}\biggr] & \leq\EE\biggl[\sup_{t\in\TT}\biggl|\int_{0}^{t}\left[\nabla_{\XX}c(s,x_{s}^{\pi},\pi_{s})\delta_{s}^{\pi,q}+c(s,x_{s}^{\pi},q_{s}-\pi_{s})\right]\,\D s\biggr|^{r}\biggr]\\
 & \leq3^{r-1}T\EE\biggl[\sup_{t\in\TT}\bigl|\nabla_{\XX}c(t,x_{t}^{\pi},\pi_{t})\bigr|^{r}\sup_{t\in\TT}\bigl|\delta_{t}^{\pi,q}\bigr|^{r}\\
 & \qquad\qquad\qquad+\sup_{t\in\TT}\bigl|c(t,x_{t}^{\pi},\pi_{t})\bigr|^{r}+\sup_{t\in\TT}\bigl|c(t,x_{t}^{\pi},q_{t})\bigr|^{r}\biggr]\\
 & \leq9^{r-1}T\EE\biggl[\sup_{t\in\TT}L^{r}\Bigl(1+\bigl|x_{t}^{\pi}\bigr|^{rp_{1}^{\prime}}+\int_{\AA}\bigl|a\bigr|^{rp_{2}^{\prime}}\pi_{t}(\D a)\Bigr)\sup_{t\in\TT}\bigl|\delta_{t}^{\pi,q}\bigr|^{r}\\
 & \qquad\qquad\qquad+2\sup_{t\in\TT}L^{r}\Bigl(1+\bigl|x_{t}^{\pi}\bigr|^{rp_{1}}+\frac{1}{2}\int_{\AA}\bigl|a\bigr|^{rp_{2}}\left(\pi_{t}+q_{t}\right)(\D a)\Bigr)\biggr]\\
 & \leq9^{r-1}TL^{r}\biggl\{\EE\biggl[\sup_{t\in\TT}\Bigl(1+\bigl|x_{t}^{\pi}\bigr|^{rp_{1}^{\prime}}+\int_{\AA}\bigl|a\bigr|^{rp_{2}^{\prime}}\pi_{t}(\D a)\Bigr)^{\zeta}\biggr]^{1/\zeta}\EE\biggl[\sup_{t\in\TT}\bigl|\delta_{t}^{\pi,q}\bigr|^{r\zeta/(\zeta-1)}\biggr]^{1-1/\zeta}\\
 & \qquad\qquad\qquad\quad+2\EE\biggl[\sup_{t\in\TT}\Bigl(1+\bigl|x_{t}^{\pi}\bigr|^{rp_{1}}+\frac{1}{2}\int_{\AA}\bigl|a\bigr|^{rp_{2}}\left(\pi_{t}+q_{t}\right)(\D a)\Bigr)\biggr]\biggr\}\quad\forall\zeta\in(1,\infty).
\end{align*}
We select $r\zeta/(\zeta-1)=\pbar$, that is, $\zeta=1/(1-r/\pbar)>1$
so that the expectation involving $\delta^{\pi,q}$ is guaranteed
to be finite (this follows from Eq.~(\ref{eq:delta-pbar-finite})).
In order to ensure that the two other expectations are also finite,
we need to fix $r$ so that
\begin{gather*}
\zeta rp_{1}^{\prime}\leq\pbar,\quad\zeta rp_{2}^{\prime}\leq\pbar_{3},\quad rp_{1}\leq\pbar,\quad rp_{2}\leq\pbar_{3},
\end{gather*}
where the upper bounds are set by $x^{\pi}\in\Ss_{\Ff}^{\pbar}$,
and the $\pbar_{3}$-admissibility of the control. Largest choice
therefore satisfies
\begin{gather*}
r=\frac{\pbar}{\zeta p_{1}^{\prime}}\wedge\frac{\pbar_{3}}{\zeta p_{2}^{\prime}}\wedge\frac{\pbar}{p_{1}}\wedge\frac{\pbar_{3}}{p_{2}}.
\end{gather*}
Note that $\zeta$ here depends on $r$, but since as a function of
$r$, the right hand side is continuous, decreasing, and maps $(0,\pbar)$
to $(0,a)$ for some $a>0$, a solution exists. 

We want to show that $r>p$. Using Eqs.~(\ref{eq:growth-deriv-relation})
and~(\ref{eq:main-deriv-bounds}),
\begin{align*}
r & \geq\frac{\pbar}{(\zeta p_{1}^{\prime})\vee(\zeta p_{2}^{\prime})\vee p_{1}\vee p_{2}}\geq\frac{\pbar}{\zeta}\frac{1}{p_{1}\vee p_{2}}\\
 & >\frac{\pbar}{\zeta}\frac{1}{\pbar/p-1}=(\pbar-r)\frac{1}{\pbar/p-1}.
\end{align*}
This implies that $(\pbar/p-1)r>\pbar-r$, from which we obtain $r>p$.
This establishes Eq.~(\ref{eq:delta-prime-p-finite}).

Let us define
\begin{align*}
b_{t}^{\pi,\pi(\alpha,q)} & \deq b(t,x_{t}^{\pi(\alpha,q)},\pi_{t}(\alpha,q))-b(t,x_{t}^{\pi},\pi_{t})-\alpha\Bigl[\nabla_{\XX}b(t,x_{t}^{\pi},\pi_{t})\delta_{t}^{\pi,q}+b(t,x_{t}^{\pi},q_{t}-\pi_{t})\Bigr],\\
\sigma_{t}^{\pi,\pi(\alpha,q)} & \deq\sigma(t,x_{t}^{\pi(\alpha,q)},\pi_{t}(\alpha,q))-\sigma(t,x_{t}^{\pi},\pi_{t})-\alpha\Bigl[\nabla_{\XX}\sigma(t,x_{t}^{\pi},\pi_{t})\delta_{t}^{\pi,q}+\sigma(t,x_{t}^{\pi},q_{t}-\pi_{t})\Bigr]\\
 & \qquad\qquad\forall t\in\TT.
\end{align*}
Note now that we can write
\begin{align*}
b_{t}^{\pi,\pi(\alpha,q)} & =b(t,x_{t}^{\pi(\alpha,q)},\pi_{t}(\alpha,q))-b(t,x_{t}^{\pi},\pi_{t}(\alpha,q))\\
 & \qquad+b(t,x_{t}^{\pi},\pi_{t}(\alpha,q))-b(t,x_{t}^{\pi},\pi_{t})\\
 & \qquad-\alpha\Bigl[\nabla_{\XX}b(t,x_{t}^{\pi},\pi_{t})\delta_{t}^{\pi,q}+b(t,x_{t}^{\pi},q_{t}-\pi_{t})\Bigr]\\
 & =\int_{0}^{1}\nabla_{\XX}b(t,x_{t}^{\pi}+\lambda(x_{t}^{\pi(\alpha,q)}-x_{t}),\pi_{t}(\alpha,q))(x_{t}^{\pi(\alpha,q)}-x_{t})\,\D\lambda\\
 & \qquad+\alpha b(t,x_{t}^{\pi},q_{t}-\pi_{t})\\
 & \qquad-\alpha\Bigl[\nabla_{\XX}b(t,x_{t}^{\pi},\pi_{t})\delta_{t}^{\pi,q}+b(t,x_{t}^{\pi},q_{t}-\pi_{t})\Bigr]\\
 & =\int_{0}^{1}\nabla_{\XX}b(t,x_{t}^{\pi}+\lambda(x_{t}^{\pi(\alpha,q)}-x_{t}),\pi_{t}(\alpha,q))(x_{t}^{\pi(\alpha,q)}-x_{t})\,\D\lambda\\
 & \qquad-\alpha\nabla_{\XX}b(t,x_{t}^{\pi},\pi_{t})\delta_{t}^{\pi,q}\\
 & =\int_{0}^{1}\nabla_{\XX}b(t,x_{t}^{\pi}+\lambda(x_{t}^{\pi(\alpha,q)}-x_{t}),\pi_{t})(x_{t}^{\pi(\alpha,q)}-x_{t}-\alpha\delta_{t}^{\pi,q})\,\D\lambda\\
 & \qquad+\alpha\int_{0}^{1}\nabla_{\XX}b(t,x_{t}^{\pi}+\lambda(x_{t}^{\pi(\alpha,q)}-x_{t}),q_{t}-\pi_{t})(x_{t}^{\pi(\alpha,q)}-x_{t})\,\D\lambda\\
 & \qquad+\alpha\Bigl[\int_{0}^{1}\nabla_{\XX}b(t,x_{t}^{\pi}+\lambda(x_{t}^{\pi(\alpha,q)}-x_{t}),\pi_{t})\,\D\lambda-\nabla_{\XX}b(t,x_{t}^{\pi},\pi_{t})\Bigr]\delta_{t}^{\pi,q}\\
 & =b_{t}^{\pi,\pi(\alpha,q),1}+\alpha b_{t}^{\pi,\pi(\alpha,q),2}\qquad\forall t\in\TT,
\end{align*}
where we have defined
\begin{align*}
b_{t}^{\pi,\pi(\alpha,q),1} & \deq\int_{0}^{1}\nabla_{\XX}b(t,x_{t}^{\pi}+\lambda(x_{t}^{\pi(\alpha,q)}-x_{t}),\pi_{t})(x_{t}^{\pi(\alpha,q)}-x_{t}-\alpha\delta_{t}^{\pi,q})\,\D\lambda,\\
b_{t}^{\pi,\pi(\alpha,q),2} & \deq\int_{0}^{1}\nabla_{\XX}b(t,x_{t}^{\pi}+\lambda(x_{t}^{\pi(\alpha,q)}-x_{t}),q_{t}-\pi_{t})(x_{t}^{\pi(\alpha,q)}-x_{t})\,\D\lambda\\
 & \qquad+\biggl[\int_{0}^{1}\nabla_{\XX}b(t,x_{t}^{\pi}+\lambda(x_{t}^{\pi(\alpha,q)}-x_{t}),\pi_{t})\,\D\lambda-\nabla_{\XX}b(t,x_{t}^{\pi},\pi_{t})\biggr]\delta_{t}^{\pi,q}\qquad\forall t\in\TT.
\end{align*}
Analogously, we find that
\begin{align*}
\sigma_{t}^{\pi,\pi(\alpha,q)} & =\sigma_{t}^{\pi,\pi(\alpha,q),1}+\alpha\sigma_{t}^{\pi,\pi(\alpha,q),2}\qquad\forall t\in\TT,
\end{align*}
where
\begin{align*}
\sigma_{t}^{\pi,\pi(\alpha,q),1} & \deq\int_{0}^{1}\nabla_{\XX}\sigma(t,x_{t}^{\pi}+\lambda(x_{t}^{\pi(\alpha,q)}-x_{t}),\pi_{t})(x_{t}^{\pi(\alpha,q)}-x_{t}-\alpha\delta_{t}^{\pi,q})\,\D\lambda,\\
\sigma_{t}^{\pi,\pi(\alpha,q),2} & \deq\int_{0}^{1}\nabla_{\XX}\sigma(t,x_{t}^{\pi}+\lambda(x_{t}^{\pi(\alpha,q)}-x_{t}),q_{t}-\pi_{t})(x_{t}^{\pi(\alpha,q)}-x_{t})\,\D\lambda\\
 & \qquad+\biggl[\int_{0}^{1}\nabla_{\XX}\sigma(t,x_{t}^{\pi}+\lambda(x_{t}^{\pi(\alpha,q)}-x_{t}),\pi_{t})\,\D\lambda-\nabla_{\XX}\sigma(t,x_{t}^{\pi},\pi_{t})\biggr]\delta_{t}^{\pi,q}\qquad\forall t\in\TT.
\end{align*}
We can then write, for arbitrary $T_{0}\in\TT$,
\begin{align}
\EE\biggl[\sup_{t\in[0,T_{0}]}\Bigl|x_{t}^{\pi(\alpha,q)} & -x_{t}^{\pi}-\alpha\delta_{t}^{\pi,q}\Bigr|^{\pbar}\biggr]\nonumber \\
 & =\EE\biggl[\sup_{t\in[0,T_{0}]}\biggl|\int_{0}^{t}\Bigl(b_{s}^{\pi,\pi(\alpha,q),1}+\alpha b_{s}^{\pi,\pi(\alpha,q),2}\Bigr)\,\D s\nonumber \\
 & \qquad\qquad+\int_{0}^{t}\Bigl(\sigma_{s}^{\pi,\pi(\alpha,q),1}+\alpha\sigma_{s}^{\pi,\pi(\alpha,q),2}\Bigr)\,\D w_{s}\biggr|^{\pbar}\biggr]\nonumber \\
 & \leq\EE\biggl[\biggl(\int_{0}^{T_{0}}\bigl|b_{s}^{\pi,\pi(\alpha,q),1}\bigr|\,\D s+\alpha\int_{0}^{T_{0}}\bigl|b_{s}^{\pi,\pi(\alpha,q),2}\bigr|\,\D s\nonumber \\
 & \qquad\qquad+\sup_{t\in[0,T_{0}]}\biggl|\int_{0}^{t}\sigma_{s}^{\pi,\pi(\alpha,q),1}\,\D w_{s}\biggr|+\sup_{t\in[0,T_{0}]}\biggl|\int_{0}^{t}\alpha\sigma_{s}^{\pi,\pi(\alpha,q),2}\,\D w_{s}\biggr|\biggr)^{\pbar}\biggr]\nonumber \\
 & \leq4^{\pbar-1}\EE\biggl[T^{\pbar-1}\int_{0}^{T_{0}}\bigl|b_{s}^{\pi,\pi(\alpha,q),1}\bigr|^{\pbar}\,\D s+\alpha^{\pbar}T^{\pbar-1}\int_{0}^{T_{0}}\bigl|b_{s}^{\pi,\pi(\alpha,q),2}\bigr|^{\pbar}\,\D s\nonumber \\
 & \qquad\qquad+C_{\pbar}\biggl(\int_{0}^{T_{0}}\bigl|\sigma_{s}^{\pi,\pi(\alpha,q),1}\bigr|^{2}\,\D s\biggr)^{\pbar/2}+\alpha^{\pbar}C_{\pbar}\biggl(\int_{0}^{T_{0}}\bigl|\sigma_{s}^{\pi,\pi(\alpha,q),2}\bigr|^{2}\,\D s\biggr)^{\pbar/2}\biggr].\label{eq:temp-eps-pbar}
\end{align}
Using arguments we used in the proof of the previous lemma, we can
estimate
\begin{align}
\EE\biggl[\int_{0}^{T_{0}}\Bigl|b_{s}^{\pi,\pi(\alpha,q),1}\Bigr|^{\pbar}\,\D s\biggr] & \leq L^{\pbar}\EE\biggl[\int_{0}^{T_{0}}\Bigl|x_{s}^{\pi(\alpha,q)}-x_{s}-\alpha\delta_{s}^{\pi,q}\Bigr|^{\pbar}\,\D s\biggr]\nonumber \\
 & \leq L^{\pbar}\EE\biggl[\int_{0}^{T_{0}}\sup_{s\in[0,t]}\Bigl|x_{s}^{\pi(\alpha,q)}-x_{s}-\alpha\delta_{s}^{\pi,q}\Bigr|^{\pbar}\,\D t\biggr],\label{eq:temp-bpi}\\
\EE\biggl[\biggl(\int_{0}^{T_{0}}\bigl|\sigma_{s}^{\pi,\pi(\alpha,q),1}\bigr|^{2}\,\D s\biggr)^{\pbar/2}\biggr] & \leq\EE\biggl[\biggl(\sup_{t\in[0,T_{0}]}\bigl|\sigma_{t}^{\pi,\pi(\alpha,q),1}\bigr|\biggr)^{\pbar/2}\biggl(\int_{0}^{T_{0}}\bigl|\sigma_{s}^{\pi,\pi(\alpha,q),1}\bigr|\,\D s\biggr)^{\pbar/2}\biggr]\nonumber \\
 & \leq\frac{1}{2}\EE\biggl[4^{-\pbar+1}C_{\pbar}^{-1}\sup_{t\in[0,T_{0}]}\bigl|x_{s}^{\pi(\alpha,q)}-x_{s}-\alpha\delta_{s}^{\pi,q}\bigr|^{\pbar}\nonumber \\
 & \qquad+4^{\pbar-1}C_{\pbar}L^{2\pbar}T^{\pbar-1}\int_{0}^{T_{0}}\sup_{s\in[0,t]}\bigl|x_{s}^{\pi(\alpha,q)}-x_{s}-\alpha\delta_{s}^{\pi,q}\bigr|^{\pbar}\,\D t\biggr].\label{eq:temp-spi}
\end{align}
Returning to Eq.~(\ref{eq:temp-eps-pbar}), applying the estimates
of Eqs.~(\ref{eq:temp-bpi}) and~(\ref{eq:temp-spi}), and rearranging
and using Grönwall's inequality (see steps leading to Eqs.~(\ref{eq:temp-pre-gronwall},~\ref{eq:temp-gronwall})),
we obtain
\begin{align*}
\EE\biggl[\sup_{t\in[0,T]}\Bigl|x_{t}^{\pi(\alpha,q)}-x_{t}^{\pi}-\alpha\delta_{t}^{\pi,q}\Bigr|^{\pbar}\biggr] & \leq2\cdot4^{\pbar-1}\alpha^{\pbar}\biggl\{ T^{\pbar-1}\int_{0}^{T}\bigl|b_{s}^{\pi,\pi(\alpha,q),2}\bigr|^{\pbar}\,\D s+C_{\pbar}\EE\biggl[\biggl(\int_{0}^{T}\bigl|\sigma_{s}^{\pi,\pi(\alpha,q),2}\bigr|^{2}\,\D s\biggr)^{\pbar/2}\biggr]\biggr\}\\
 & \qquad\times\E^{4^{\pbar-1}T^{\pbar-1}L^{\pbar}\Bigl(2+4^{\pbar-1}L^{\pbar}C_{\pbar}^{2}\Bigr)T}.
\end{align*}
To show that Eq.~(\ref{eq:sup-eps-o-pbar}) holds, it then remains
to show that
\begin{gather*}
\lim_{\alpha\to0}\EE\biggl[\int_{0}^{T}\bigl|b_{s}^{\pi,\pi(\alpha,q),2}\bigr|^{\pbar}\,\D s\biggr]=0,\\
\lim_{\alpha\to0}\EE\biggl[\biggl(\int_{0}^{T}\bigl|\sigma_{s}^{\pi,\pi(\alpha,q),2}\bigr|^{2}\,\D s\biggr)^{\pbar/2}\biggr]=0.
\end{gather*}
These limits are obtained from the definitions of $(b_{t}^{\pi,\pi(\alpha,q),2})_{t\in\TT}$
and $(\sigma_{t}^{\pi,\pi(\alpha,q),2})_{t\in\TT}$ by using Eqs.~(\ref{eq:sup-pbar})
and~(\ref{eq:delta-pbar-finite}), and the continuity and boundedness
of the gradients of $b$ and $\sigma$. Explicitly,
\begin{align*}
\EE\biggl[\int_{0}^{T}\bigl|b_{s}^{\pi,\pi(\alpha,q),2}\bigr|^{\pbar}\,\D s\biggr] & \leq T\EE\biggl[\sup_{t\in\TT}\bigl|b_{t}^{\pi,\pi(\alpha,q),2}\bigr|^{\pbar}\biggr]\\
 & \leq2^{\pbar-1}L^{\pbar}T\EE\biggl[\sup_{t\in\TT}\bigl|x_{t}^{\pi(\alpha,q)}-x_{t}^{\pi}\bigr|^{\pbar}\biggr]\\
 & \qquad+2^{\pbar-1}T\EE\biggl[\biggl|\int_{0}^{1}\nabla_{\XX}b(t,x_{t}^{\pi}+\lambda(x_{t}^{\pi(\alpha,q)}-x_{t}),\pi_{t})\,\D\lambda-\nabla_{\XX}b(t,x_{t}^{\pi},\pi_{t})\biggr|^{\pbar}\bigl|\delta_{t}^{\pi,q}\bigr|^{\pbar}\biggr].
\end{align*}
The first term is $\Oo(\alpha^{\pbar})$ by Eq.~(\ref{eq:sup-pbar}),
and the second is finite by Eq.~(\ref{eq:delta-pbar-finite}) and
boundedness of $\nabla_{\XX}b$, and tends to zero as $\alpha\to0$,
since $\nabla_{\XX}b$ is continuous. This completes the estimates
of the right hand side of Eq.~(\ref{eq:temp-eps-pbar}), and so Eq.~(\ref{eq:sup-eps-o-pbar})
is proven.

Moving on to proving Eq.~(\ref{eq:sup-eps-prime-o-p}), we proceed
as above, and use estimates analogous to those used in the lead up
to Eq.~(\ref{eq:temp-zs}), so that
\begin{align*}
\EE\biggl[\sup_{t\in\TT} & \Bigl|x_{t}^{\prime\,\pi(\alpha,q)}-x_{t}^{\prime\,\pi}-\alpha\delta_{t}^{\prime\,\pi,q}\Bigr|^{p}\biggr]\\
 & =\EE\biggl[\sup_{t\in\TT}\biggl|\int_{0}^{t}\Bigl(c(s,x_{s}^{\pi(\alpha,q)},\pi_{s}(\alpha,q))-c(s,x_{s}^{\pi},\pi_{s}(\alpha,q))\Bigr)\,\D s\\
 & \qquad\qquad+\int_{0}^{t}\Bigl(c(s,x_{s}^{\pi},\pi_{s}(\alpha,q))-c(s,x_{s}^{\pi},\pi_{s})\Bigr)\,\D s\\
 & \qquad\qquad-\alpha\int_{0}^{t}\left[\nabla_{\XX}c(s,x_{s}^{\pi},\pi_{s})\delta_{s}^{\pi,q}+c(t,x_{s}^{\pi},q_{s}-\pi_{s})\right]\,\D s\biggr|^{p}\biggr]\\
 & =\EE\biggl[\sup_{t\in\TT}\biggl|\int_{0}^{t}\int_{0}^{1}c(s,x_{s}^{\pi}+\lambda(x_{s}^{\pi(\alpha,q)}-x_{s}^{\pi}),\pi_{s}(\alpha,q))(x_{s}^{\pi(\alpha,q)}-x_{s}^{\pi}-\alpha\delta_{s}^{\pi,q})\,\D\lambda\,\D s\\
 & \qquad\qquad+\alpha\int_{0}^{t}\biggl[\int_{0}^{1}\nabla_{\XX}c(s,x_{s}^{\pi}+\lambda(x_{s}^{\pi(\alpha,q)}-x_{s}^{\pi}),\pi_{s}(\alpha,q))\,\D\lambda-\nabla_{\XX}c(s,x_{s}^{\pi},\pi_{s})\biggr]\delta_{s}^{\pi,q}\,\D s\biggr|^{p}\biggr]\\
 & =2^{p-1}\EE\biggl[L^{p}T^{p}\biggl\{\sup_{t\in\TT}\bigl|x_{t}^{\pi(\alpha,q)}-x_{t}^{\pi}-\alpha\delta_{t}^{\pi,q}\bigr|^{p}\\
 & \qquad\qquad\qquad\qquad+\ell_{pp_{1}^{\prime}}\sup_{t\in\TT}\Bigl(\Bigr(\bigl|x_{t}^{\pi(\alpha,q)}\bigr|^{pp_{1}^{\prime}}+\bigl|x_{t}^{\pi}\bigr|^{pp_{1}^{\prime}}\Bigl)\vee\bigl|x_{t}^{\pi(\alpha,q)}+x_{t}^{\pi}\bigr|^{pp_{1}^{\prime}}\Bigr)\sup_{t\in\TT}\bigl|x_{t}^{\pi(\alpha,q)}-x_{t}^{\pi}-\alpha\delta_{t}^{\pi,q}\bigr|^{p}\\
 & \qquad\qquad\qquad\qquad+\sup_{t\in\TT}\biggl(\int_{\AA}|a|^{pp_{2}^{\prime}}\pi_{t}(\alpha,q)(\D a)\biggr)\sup_{t\in\TT}\bigl|x_{t}^{\pi(\alpha,q)}-x_{t}^{\pi}-\alpha\delta_{t}^{\pi,q}\bigr|^{p}\biggr\}\\
 & \qquad\qquad+\alpha^{p}\sup_{t\in\TT}\biggl|\int_{0}^{t}\biggl[\int_{0}^{1}\nabla_{\XX}c(s,x_{s}^{\pi}+\lambda(x_{s}^{\pi(\alpha,q)}-x_{s}^{\pi}),\pi_{s}(\alpha,q))\,\D\lambda-\nabla_{\XX}c(s,x_{s}^{\pi},\pi_{s})\biggr]\delta_{s}^{\pi,q}\,\D s\biggr|^{p}\biggr].
\end{align*}
The terms on the three first lines on the right of the last inequality
can now be shown to be $o(\alpha^{p})$ using Eq.~(\ref{eq:sup-eps-o-pbar})
and an application of Hölder's inequality mimicking the steps in Eqs.~(\ref{eq:z-holder-1},~\ref{eq:z-holder-2},~\ref{eq:z-holder-3}).
Similarly applying Hölder's inequality to the remaining term on the
last line yields
\begin{align*}
\EE\biggl[\sup_{t\in\TT}\Bigl|\int_{0}^{1}\nabla_{\XX}c(t,x_{t}^{\pi}+\lambda(x_{t}^{\pi(\alpha,q)}-x_{t}^{\pi}),\pi_{t}(\alpha,q))\,\D\lambda & -\nabla_{\XX}c(t,x_{t}^{\pi},\pi_{t})\Bigr|^{p}\sup_{t\in\TT}\bigl|\delta_{t}^{\pi,q}\bigr|^{p}\biggr]<\infty,
\end{align*}
and, by continuity of $\nabla_{\XX}c$, this term is in $o(\alpha^{0})$,
and Eq.~(\ref{eq:sup-eps-prime-o-p}) follows.
\end{proof}
\begin{proof}[Proof of Lemma~\ref{lem:delta-term}]
Using the $\Ll$-differentiability of $\rho$ and Eq.~(\ref{eq:phi-expansion})
and the short-hand of Eq.~(\ref{eq:ThetaD}), we have that
\begin{gather*}
\rho(\Theta^{\pi(\alpha,q)})=\rho(\Theta^{\pi})+\EE\Bigr[D^{\pi}\left(\theta(x_{T}^{\pi(\alpha,q)},x_{T}^{\prime\,\pi(\alpha,q)})-\theta(x_{T}^{\pi},x_{T}^{\prime\,\pi})\right)\Bigl]+R^{\pi},
\end{gather*}
where 
\[
R^{\pi}\in o\left(\Lnorm{\theta(x_{T}^{\pi(\alpha,q)},x_{T}^{\prime\,\pi(\alpha,q)})-\theta(x_{T}^{\pi},x_{T}^{\prime\,\pi})}p\right)
\]
By Lemma~\ref{lem:var-vanish},
\begin{align*}
\Lnorm{\theta(x_{T}^{\pi(\alpha,q)},x_{T}^{\prime\,\pi(\alpha,q)})-\theta(x_{T}^{\pi},x_{T}^{\prime\,\pi})}p & =\Lnorm{g(x_{T}^{\pi(\alpha,q)})-g(x_{T}^{\pi})+x_{T}^{\prime\,\pi(\alpha,q)}-x_{T}^{\prime\,\pi}}p\in\Oo(\alpha),
\end{align*}
and therefore $R^{\pi}\in o(\alpha)$. From the optimality condition
of Eq.~(\ref{eq:var-optimality}),
\begin{align*}
0 & \leq\EE\left[D^{\pi}\left(\theta(x_{T}^{\pi(\alpha,q)},x_{T}^{\prime\,\pi(\alpha,q)})-\theta(x_{T}^{\pi},x_{T}^{\prime\,\pi})\right)\right]+R^{\pi}\\
 & =\EE\left[D^{\pi}\left(g(x_{T}^{\pi(\alpha,q)})-g(x_{T}^{\pi})+x_{T}^{\prime\,\pi(\alpha,q)}-x_{T}^{\prime\,\pi}\right)\right]+R^{\pi}\\
 & =\EE\left[D^{\pi}\left(\int_{0}^{1}\nabla_{\XX}g(\lambda x_{T}^{\pi(\alpha,q)}+(1-\lambda)x_{T}^{\pi})(x_{T}^{\pi(\alpha,q)}-x_{T}^{\pi})\,\D\lambda+x_{T}^{\prime\,\pi(\alpha,q)}-x_{T}^{\prime\,\pi}\right)\right]+R^{\pi}\\
 & =\EE\biggl[D^{\pi}\biggl(\int_{0}^{1}\nabla_{\XX}g(\lambda x_{T}^{\pi(\alpha,q)}+(1-\lambda)x_{T}^{\pi})(x_{T}^{\pi(\alpha,q)}-x_{T}^{\pi}-\alpha\delta_{t}^{\pi,q}+\alpha\delta_{t}^{\pi,q})\,\D\lambda\\
 & \qquad\qquad\quad+x_{T}^{\prime\,\pi(\alpha,q)}-x_{T}^{\prime\,\pi}-\alpha\delta_{t}^{\prime\,\pi,q}+\alpha\delta_{t}^{\prime\,\pi,q}\biggr)\biggr]+R^{\pi}\\
 & =\alpha\EE\biggl[D^{\pi}\biggl(\int_{0}^{1}\nabla_{\XX}g(\lambda x_{T}^{\pi(\alpha,q)}+(1-\lambda)x_{T}^{\pi})\delta_{t}^{\pi,q}\,\D\lambda+\delta_{t}^{\prime\,\pi,q}\biggr)\biggr]\\
 & \qquad+\EE\biggl[D^{\pi}\biggl(\int_{0}^{1}\nabla_{\XX}g(\lambda x_{T}^{\pi(\alpha,q)}+(1-\lambda)x_{T}^{\pi})(x_{T}^{\pi(\alpha,q)}-x_{T}^{\pi}-\alpha\delta_{t}^{\pi,q})\,\D\lambda\\
 & \qquad\qquad\quad+x_{T}^{\prime\,\pi(\alpha,q)}-x_{T}^{\prime\,\pi}-\alpha\delta_{t}^{\prime\,\pi,q}\biggr)\biggr]+R^{\pi}\\
 & \leq\alpha\EE\biggl[D^{\pi}\biggl(\int_{0}^{1}\nabla_{\XX}g(\lambda x_{T}^{\pi(\alpha,q)}+(1-\lambda)x_{T}^{\pi})\delta_{t}^{\pi,q}\,\D\lambda+\delta_{t}^{\prime\,\pi,q}\biggr)\biggr]\\
 & \qquad+\EE\Bigl[\bigl|D^{\pi}\bigr|^{p/(p-1)}\Bigr]^{1-1/p}\EE\biggl[\biggl|\int_{0}^{1}\nabla_{\XX}g(\lambda x_{T}^{\pi(\alpha,q)}+(1-\lambda)x_{T}^{\pi})(x_{T}^{\pi(\alpha,q)}-x_{T}^{\pi}-\alpha\delta_{t}^{\pi,q})\,\D\lambda\\
 & \qquad\qquad\quad+x_{T}^{\prime\,\pi(\alpha,q)}-x_{T}^{\prime\,\pi}-\alpha\delta_{t}^{\prime\,\pi,q}\biggr|^{p}\biggr]^{1/p}+R^{\pi}\\
 & \leq\alpha\EE\biggl[D^{\pi}\biggl(\int_{0}^{1}\nabla_{\XX}g(\lambda x_{T}^{\pi(\alpha,q)}+(1-\lambda)x_{T}^{\pi})\delta_{t}^{\pi,q}\,\D\lambda+\delta_{t}^{\prime\,\pi,q}\biggr)\biggr]\\
 & \qquad+\EE\Bigl[\bigl|D^{\pi}\bigr|^{p/(p-1)}\Bigr]^{1-1/p}\EE\biggl[2^{p-1}\biggl(\int_{0}^{1}L\Bigl(1+\bigl|\lambda x_{T}^{\pi(\alpha,q)}+(1-\lambda)x_{T}^{\pi}\bigr|^{p_{1}^{\prime}}\Bigr)^{p}\Bigl|x_{T}^{\pi(\alpha,q)}-x_{T}^{\pi}-\alpha\delta_{t}^{\pi,q}\Bigr|^{p}\,\D\lambda\\
 & \qquad\qquad\quad+\Bigl|x_{T}^{\prime\,\pi(\alpha,q)}-x_{T}^{\prime\,\pi}-\alpha\delta_{t}^{\prime\,\pi,q}\Bigr|^{p}\biggr)\biggr]^{1/p}+R^{\pi}.
\end{align*}
Dividing this by $\alpha$, taking the limit $\alpha\to0$ while using
a similar estimate as in Eq.~(\ref{eq:holder-trick-x}) and applying
Lemma~\ref{lem:epsilon}, we obtain the equation claimed in the statement
of this lemma.
\end{proof}
\begin{proof}[Proof of Lemma~\ref{lem:var-hamil}]
We give the proof for $p>1$, but comment at relevant places on changes
needed to accommodate the $p=1$ case. The statement of the lemma
amounts to expressing Eq.~(\ref{eq:var-deltaform}) by using processes
that are constructed to satisfy Eq.~(\ref{eq:relaxed-bsde}). For
brevity, we set $B_{t}^{\pi}\deq\nabla_{\XX}b(t,x_{t}^{\pi},\pi_{t})\in\RR^{d_{x}\times d_{x}}$,
$F_{t}^{\pi}\deq\nabla_{\XX}c(t,x_{t}^{\pi},\pi_{t})\in\RR^{1\times d_{x}}$,
$S_{t}^{\pi}\deq\nabla_{\XX}\sigma(t,x_{t}^{\pi},\pi_{t})\in\RR^{d_{x}\times d_{w}\times d_{x}}$
for all $t\in\TT$.

Let $(U_{t}^{\pi})_{t\in\TT}$ be the fundamental solution of Eq.~(\ref{eq:delta-solo}),
i.e. $U_{t}^{\pi}\in\RR^{d_{x}\times d_{x}}$, $t\in\TT$, $U_{0}^{\pi}=I$,
where $I$ is the identity matrix, and 
\begin{gather}
\D U_{t}^{\pi}=B_{t}^{\pi}U_{t}^{\pi}\,\D t+(S_{t}^{\pi}U_{t}^{\pi})\cdot\D w_{t}.\label{eq:U-sde-1}
\end{gather}
The drift and diffusion functions, $(t,U)\to(B_{t}^{\pi}U)_{t\in\TT}$
and $(t,U)\to(S_{t}^{\pi}U)_{t\in\TT}$, $(t,U)\in\TT\times\RR^{d_{x}\times d_{s}}$,
are $L$-Lipschitz in $U$ for all $t\in\TT$ since by Assumption~\ref{assu:sde-baseline}(\emph{iii})
the gradients of $b$ and $\sigma$ are bounded. By an application
of e.g. \cite[Theorem 3.17]{Pardoux2014}, we find that Eq.~(\ref{eq:U-sde-1})
has a unique strong solution satisfying
\begin{gather}
\EE\biggl[\sup_{t\in\TT}\bigl|U_{t}^{\pi}\bigr|^{p_{0}}\biggr]<\infty\qquad\forall p_{0}\in[1,\infty).\label{eq:U-S-norm}
\end{gather}
In addition, we define the process $(V_{t}^{\pi})_{t\in\TT}$ as the
unique strong solution of the stochastic differential equation
\begin{gather*}
\D V_{t}^{\pi}=V_{t}^{\pi}\left(-B_{t}^{\pi}+S_{t}^{\pi}\cddot S_{t}^{\pi}\right)\,\D t-(V_{t}^{\pi}S_{t}^{\pi})\cdot\D w_{t},\\
V_{0}^{\pi}=I.
\end{gather*}
Such a solution exists by the same argument as used for Eq.~(\ref{eq:U-sde-1}),
and moreover,
\begin{gather}
\EE\biggl[\sup_{t\in\TT}\bigl|V_{t}^{\pi}\bigr|^{p_{0}}\biggr]<\infty\qquad\forall p_{0}\in[1,\infty)\label{eq:V-S-norm}
\end{gather}
holds. It is easily verified that $V_{t}^{\pi}=(U_{t}^{\pi})^{-1}$
for all $t\in\TT$ by applying Itô's lemma to $t\to V_{t}^{\pi}U_{t}^{\pi}$
and $t\to U_{t}^{\pi}V_{t}^{\pi}$. Finally, one additional, $\RR^{1\times d_{x}}$-valued
process will be necessary. We define $(Q_{t}^{\pi})_{t\in\TT}$ as
such that
\begin{gather*}
Q_{t}^{\pi}=\int_{0}^{t}F_{s}^{\pi}U_{s}^{\pi}\,\D s\quad\forall t\in\TT.
\end{gather*}

Let us then define the random variable $\Xi^{\pi}$ as such that
\begin{gather}
\Xi^{\pi}\deq D^{\pi}\nabla_{\XX}g(x_{T}^{\pi})U_{T}^{\pi}+D^{\pi}Q_{T}^{\pi},\label{eq:temp-xi-1}
\end{gather}
We begin by showing that $\Xi\in\Ll^{r}(\Omega;\RR^{1\times d_{x}})$
for a $r\in[1,\infty)$ to be determined later. For this task, we
additionally define 
\begin{gather}
p^{\ast}\deq\frac{\pbar}{p_{1}^{\prime}}\wedge\frac{\pbar_{3}}{p_{2}^{\prime}},\label{eq:p-ast}\\
\phat\deq\frac{1}{1/p-1/p^{\ast}}.\label{eq:p-hat}
\end{gather}
Using Eq.~(\ref{eq:feasible-p}), we find that
\begin{gather*}
p^{\ast}>\frac{\pbar}{p_{1}^{\prime}\vee p_{2}^{\prime}}>\frac{\pbar}{\pbar/p-1}=\frac{1}{1/p-1/\pbar}>p.
\end{gather*}
A straight-forward calculation yields additionally the following
\begin{align*}
p^{\ast}>\frac{1}{1/p-1/\pbar} & \quad\implies\quad\frac{1}{p^{\ast}}-\frac{1}{p}<-\frac{1}{\pbar}\\
 & \quad\implies\quad\frac{1}{1/p-1/p^{\ast}}=\phat<\pbar,
\end{align*}
that is, $\phat\in(p,\pbar)$. Let $\ptilde\in(\phat,\pbar)$ be arbitrary,
and set 
\begin{gather*}
\qtilde\deq\frac{1}{1/p-1/\ptilde}.
\end{gather*}
We claim that 
\begin{gather}
\Xi^{\pi}\in\Ll^{\ptilde/(\ptilde-1)}(\Omega;\RR^{1\times d_{x}}).\label{eq:Xi-in-Lr}
\end{gather}
Since $\ptilde<\pbar$, we have $\ptilde/(\ptilde-1)>\pbar/(\pbar-1)$
and the above implies that $\Xi^{\pi}\in\Ll^{\pbar/(\pbar-1)}(\Omega;\RR^{1\times d_{x}})$.
To prove Eq.~(\ref{eq:Xi-in-Lr}), we first show that 
\begin{gather}
\begin{gathered}\nabla_{\XX}g(x_{T}^{\pi})U_{T}^{\pi}\in\Ll^{\qtilde}(\Omega;\RR^{1\times d_{x}}),\\
Q_{T}^{\pi}\in\Ll^{\qtilde}(\Omega;\RR^{1\times d_{x}}).
\end{gathered}
\label{eq:gcU-in-Lr}
\end{gather}
We explicitly prove only the latter inclusion, the former is established
in very much the same way. Basic estimates show that
\begin{align}
\bigl|Q_{T}^{\pi}\bigr|^{\qtilde} & =\biggl|\int_{0}^{T}F_{s}^{\pi}U_{s}^{\pi}\,\D s\biggr|^{\qtilde}\nonumber \\
 & \leq\biggl(\int_{0}^{T}L\Bigl(1+\bigl|x_{s}^{\pi}\bigr|^{p_{1}^{\prime}}+\int_{\AA}|a|^{p_{2}^{\prime}}\pi_{s}(\D a)\Bigr)\bigl|U_{s}^{\pi}\bigr|\,\D s\biggr)^{\qtilde}\nonumber \\
 & \leq L^{\qtilde}T^{\qtilde-1}\int_{0}^{T}\Bigl(1+\bigl|x_{s}^{\pi}\bigr|^{p_{1}^{\prime}}+\int_{\AA}|a|^{p_{2}^{\prime}}\pi_{s}(\D a)\Bigr)^{\qtilde}\bigl|U_{s}^{\pi}\bigr|^{\qtilde}\,\D s\nonumber \\
 & \leq3^{\qtilde-1}L^{\qtilde}T^{\qtilde-1}\int_{0}^{T}\Bigl(1+\bigl|x_{s}^{\pi}\bigr|^{\qtilde p_{1}^{\prime}}+\int_{\AA}|a|^{\qtilde p_{2}^{\prime}}\pi_{s}(\D a)\Bigr)\bigl|U_{s}^{\pi}\bigr|^{\qtilde}\,\D s\nonumber \\
 & \leq3^{\qtilde-1}L^{\qtilde}T^{\qtilde}\sup_{t\in\TT}\Bigl(1+\bigl|x_{t}^{\pi}\bigr|^{\qtilde p_{1}^{\prime}}+\int_{\AA}|a|^{\qtilde p_{2}^{\prime}}\pi_{t}(\D a)\Bigr)\sup_{t\in\TT}\bigl|U_{t}^{\pi}\bigr|^{\qtilde}.\label{eq:nabla-c-U-estimate}
\end{align}
Using Hölder's inequality,
\begin{align*}
\EE\biggl[\bigl|Q_{T}^{\pi}\bigr|^{\qtilde}\biggr] & \leq3^{\qtilde-1}L^{\qtilde}T^{\qtilde}\EE\biggl[\sup_{t\in\TT}\bigl|U_{t}^{\pi}\bigr|^{\qtilde\zeta/(\zeta-1)}\biggr]^{1-1/\zeta}\\
 & \qquad\times\EE\biggl[\sup_{t\in\TT}\Bigl(1+\bigl|x_{t}^{\pi}\bigr|^{\qtilde p_{1}^{\prime}}+\int_{\AA}|a|^{\qtilde p_{2}^{\prime}}\pi_{t}(\D a)\Bigr)^{\zeta}\biggr]^{1/\zeta}\quad\forall\zeta\in(1,\infty).
\end{align*}
The first factor is finite by Eq.~(\ref{eq:U-S-norm}). To show the
same for the second one, it suffices to show that we can select a
$\zeta\in(1,\infty)$ so that
\begin{gather}
\EE\biggl[\sup_{t\in\TT}\Bigl(1+\bigl|x_{t}^{\pi}\bigr|^{\zeta\qtilde p_{1}^{\prime}}+\int_{\AA}|a|^{\zeta\qtilde p_{2}^{\prime}}\pi_{t}(\D a)\Bigr)\biggr]<\infty.\label{eq:nabla-c-aux}
\end{gather}
By the standard estimate of Eq.~(\ref{eq:x-main-estimate}), and
the admissibility condition of Eq.~(\ref{eq:admissible}) and constraints
of Eq.~(\ref{eq:feasible-p}), the largest $\zeta$ we can take is
$p^{\ast}/\qtilde$. Th choice $\zeta=p^{\ast}/\qtilde$ is valid
if $\zeta\in(1,\infty)$, which is indeed the case:
\begin{align*}
\zeta & =\frac{p^{\ast}}{\qtilde}=\frac{1}{1/p-1/\phat}\frac{1}{\qtilde}=\frac{1/p-1/\ptilde}{1/p-1/\phat}>1.
\end{align*}
This is sufficient to establish Eq.~(\ref{eq:gcU-in-Lr}).

Turning to proving Eq.~(\ref{eq:Xi-in-Lr}), based on the above it
is enough to show that $D^{\pi}G\in\Ll^{\ptilde/(\ptilde-1)}(\Omega;\RR^{1\times d_{x}})$
for all $G\in\Ll^{\qtilde}(\Omega;\RR^{1\times d_{x}})$. Let $G$
be any such random variable. Then, again using Hölder's inequality,
\begin{gather}
\EE\Bigl[\bigr|D^{\pi}G\bigl|^{\ptilde/(\ptilde-1)}\Bigl]\leq\EE\Bigl[\bigr|D^{\pi}\bigl|^{\xi\ptilde/(\ptilde-1)}\Bigl]^{1/\xi}\EE\Bigl[\bigr|G\bigl|^{[\xi/(\xi-1)][\ptilde/(\ptilde-1)]}\Bigl]^{1-1/\xi}\quad\forall\xi\in(1,\infty).\label{eq:DG-finite}
\end{gather}
Selecting $\xi=[p/(p-1)]/[\ptilde/(\ptilde-1)]$, the first factor
in the above inequality is finite, since $D^{\pi}\in\Ll^{p/(p-1)}(\Omega;\RR)$.
Note that
\begin{gather*}
\xi=\frac{\frac{p}{p-1}}{\frac{\ptilde}{\ptilde-1}}=\frac{p}{\ptilde}\frac{\ptilde-1}{p-1}=\frac{p-(p/\ptilde)}{p-1}>1,
\end{gather*}
since $\ptilde>p$. Moreover, with this choice, the power on the second
factor becomes $\qtilde$,
\begin{align*}
\frac{\xi}{\xi-1}\frac{\ptilde}{\ptilde-1} & =\frac{[p/(p-1)]/[\ptilde/(\ptilde-1)]}{[p/(p-1)]/[\ptilde/(\ptilde-1)]-1}\frac{\ptilde}{\ptilde-1}\\
 & =\frac{1}{(\ptilde-1)/\ptilde-(p-1)/p}\\
 & =\frac{1}{\frac{1}{p}-\frac{1}{\ptilde}}=\qtilde,
\end{align*}
and since $G\in\Ll^{\qtilde}(\Omega;\RR^{1\times d_{x}})$, $D^{\pi}G\in\Ll^{\ptilde/(\ptilde-1)}(\Omega;\RR)$.
Eq.~(\ref{eq:Xi-in-Lr}) is now proven.

We can now apply the martingale representation theorem for $\Ll^{r}$-random
variables ($r>1$), given e.g. in \cite[Theorem 2.42]{Pardoux2014},
to $\Xi^{\pi}$ and $D^{\pi}$. This provides us with unique $\Ff$-predictable
processes $\Sigma^{\pi}=(\Sigma_{t}^{\pi})_{t\in\TT}$ and $z^{\prime\,\pi}=(z_{t}^{\prime\,\pi})_{t\in\TT}$,
taking respectively values in $\RR^{d_{w}\times d_{x}}$ and $\RR^{d_{w}}$,
such that 
\begin{gather}
\Xi^{\pi}=\EE\bigl[\Xi^{\pi}\bigr]+\int_{0}^{T}\Sigma_{s}^{\pi}\cdot\D w_{s},\label{eq:temp-mrt-1}\\
D^{\pi}=\EE\bigl[D^{\pi}\bigr]+\int_{0}^{T}z_{s}^{\prime\,\pi}\cdot\D w_{s}.
\end{gather}
These representations are unique, and 
\begin{gather}
\EE\biggl[\biggl(\int_{0}^{T}\bigl|\Sigma_{s}^{\pi}\bigr|^{2}\D s\biggr)^{\frac{1}{2}\ptilde/(\ptilde-1)}\biggr]<\infty,\quad\EE\biggr[\biggl(\int_{0}^{T}\bigl|z_{t}^{\prime\,\pi}\bigr|\,\D t\biggr)^{\frac{1}{2}p/(p-1)}\biggr]<\infty.\label{eq:Sigma-L-norm}
\end{gather}
Moreover, we define the processes $\Lambda^{\pi}=(\Lambda_{t}^{\pi})_{t\in\TT}$
and $y^{\prime\,\pi}=(y_{t}^{\prime\,\pi})_{t\in\TT}$ as the $\Ff_{t}$-conditional
expectations of $\Xi^{\pi}$ and $D^{\pi}$, which now by \cite[Corollary 2.44]{Pardoux2014}
satisfy
\begin{gather}
\Lambda_{t}^{\pi}\deq\EE\bigl[\Xi^{\pi}\bigm|\Ff_{t}\bigr]=\EE\bigl[\Xi^{\pi}\bigr]+\int_{0}^{t}\Sigma_{s}^{\pi}\cdot\D w_{s}\qquad\forall t\in\TT,\nonumber \\
y_{t}^{\prime\,\pi}\deq\EE\bigl[D^{\pi}\bigm|\Ff_{t}\bigr]=\EE\bigl[D^{\pi}\bigr]+\int_{0}^{t}z_{s}^{\prime\,\pi}\cdot\D w_{s}\qquad\forall t\in\TT,\nonumber \\
\EE\biggl[\sup_{t\in\TT}\bigl|\Lambda_{t}^{\pi}\bigr|^{\ptilde/(\ptilde-1)}\biggr]<\infty,\qquad\EE\biggl[\sup_{t\in\TT}\bigl|y_{t}^{\prime\,\pi}\bigr|^{p/(p-1)}\biggr]<\infty.\label{eq:Lambda-S-norm}
\end{gather}
We note that if $p=1$, in the application of the martingale representation
theorem we may instead pick an arbitrary $q\in(1,\infty)$ instead
of $p/(p-1)$.

We next define the processes $y^{\pi}=(y_{t}^{\pi})_{t\in\TT}$, and
$z^{\pi}=(z_{t}^{\pi})_{t\in\TT}$ as such that 
\begin{align}
y_{t}^{\pi} & \deq\Bigl(\Lambda_{t}^{\pi}-y_{t}^{\prime\,\pi}Q_{t}^{\pi}\Bigr)V_{t}^{\pi},\quad\forall t\in\TT,\label{eq:temp-y-def-1}\\
z_{t}^{\pi} & \deq\Bigl(\Sigma_{t}^{\pi}-z_{t}^{\prime\,\pi}Q_{t}^{\pi}\Bigr)V_{t}^{\pi}-y_{t}^{\pi}S_{t}^{\pi},\quad\forall t\in\TT.\label{eq:temp-z-def-1}
\end{align}
The process $(y_{t}^{\pi},y_{t}^{\prime\,\pi},z_{t}^{\pi},z_{t}^{\prime\,\pi})_{t\in\TT}$
solves Eq.~(\ref{eq:relaxed-bsde}). This is already shown above
for $(y_{t}^{\prime\,\pi},z_{t}^{\prime\,\pi})_{t\in\TT}$. To show
that $(y_{t}^{\pi},z_{t}^{\pi})_{t\in\TT}$ satisfies its respective
backward stochastic differential equation, we apply Itô's lemma to
$y^{\pi}$ as given in Eq.~(\ref{eq:temp-y-def-1}) to obtain 
\begin{align*}
\D y_{t}^{\pi} & =\Bigl(\D\Lambda_{t}^{\pi}-\D y_{t}^{\prime\,\pi}Q_{t}^{\pi}-y_{t}^{\prime\,\pi}\D Q_{t}^{\pi}\Bigr)V_{t}^{\pi}+\Bigl(\Lambda_{t}^{\pi}-y_{t}^{\prime\,\pi}Q_{t}^{\pi}\Bigr)\,\D V_{t}^{\pi}\\
 & \qquad+\Bigl(\D\Lambda_{t}^{\pi}-\D y_{t}^{\prime\,\pi}Q_{t}^{\pi}-y_{t}^{\prime\,\pi}\D Q_{t}^{\pi}\Bigr)\,\D V_{t}^{\pi}\\
 & =\Bigl(\Sigma_{t}^{\pi}\cdot\D w_{t}-z_{t}^{\prime\,\pi}\cdot\D w_{t}Q_{t}^{\pi}-y_{t}^{\prime\,\pi}F_{t}^{\pi}U_{t}^{\pi}\,\D t\Bigr)V_{t}^{\pi}\\
 & \qquad+\Bigl(\Lambda_{t}^{\pi}-y_{t}^{\prime\,\pi}Q_{t}^{\pi}\Bigr)\biggl[V_{t}^{\pi}\left(-B_{t}^{\pi}+S_{t}^{\pi}\cddot S_{t}^{\pi}\right)\D t-(V_{t}^{\pi}S_{t}^{\pi})\cdot\D w_{t}\biggr]\\
 & \qquad-\Bigl(\Sigma_{t}^{\pi}\cdot\D w_{t}-z_{t}^{\prime\,\pi}\cdot\D w_{t}Q_{t}^{\pi}\Bigr)(V_{t}^{\pi}S_{t}^{\pi})\cdot\D w_{t}\\
 & =\biggl[\Bigl(\Sigma_{t}^{\pi}-z_{t}^{\prime\,\pi}Q_{t}^{\pi}\Bigr)\cdot\D w_{t}V_{t}^{\pi}-y_{t}^{\prime\,\pi}F_{t}^{\pi}\,\D t\biggr]\\
 & \qquad+\Bigl(\Lambda_{t}^{\pi}-y_{t}^{\prime\,\pi}Q_{t}^{\pi}\Bigr)V_{t}^{\pi}\biggl[\left(-B_{t}^{\pi}+S_{t}^{\pi}\cddot S_{t}^{\pi}\right)\D t-S_{t}^{\pi}\cdot\D w_{t}\biggr]\\
 & \qquad-\Bigl(\Sigma_{t}^{\pi}-z_{t}^{\prime\,\pi}Q_{t}^{\pi}\Bigr)\cdot\D w_{t}(V_{t}^{\pi}S_{t}^{\pi})\cdot\D w_{t}\\
 & =\Bigl[\bigl(z_{t}^{\pi}+y_{t}^{\pi}S_{t}^{\pi}\bigr)\cdot\D w_{t}-y_{t}^{\prime\,\pi}F_{t}^{\pi}\,\D t\Bigr]\\
 & \qquad+y_{t}^{\pi}\biggl[\left(-B_{t}^{\pi}+S_{t}^{\pi}\cddot S_{t}^{\pi}\right)\D t-S_{t}^{\pi}\cdot\D w_{t}\biggr]\\
 & \qquad-\Bigl(\bigl(z_{t}^{\pi}+y_{t}^{\pi}S_{t}^{\pi}\bigr)\cdot\D w_{t}\Bigr)\Bigl(S_{t}^{\pi}\cdot\D w_{t}\Bigr)\\
 & =\Bigl(z_{t}^{\pi}\cdot\D w_{t}-y_{t}^{\prime\,\pi}F_{t}^{\pi}\,\D t\Bigr)+y_{t}^{\pi}\left(-B_{t}^{\pi}+S_{t}^{\pi}\cddot S_{t}^{\pi}\right)\D t\\
 & \qquad-\bigl(z_{t}^{\pi}+y_{t}^{\pi}S_{t}^{\pi}\bigr)\cddot S_{t}^{\pi}\,\D t\\
 & =-\Bigl(y_{t}^{\pi}B_{t}^{\pi}+y_{t}^{\prime\,\pi}F_{t}^{\pi}+z_{t}^{\pi}\cddot S_{t}^{\pi}\Bigr)\,\D t+z_{t}^{\pi}\cdot\D w_{t}\\
 & =-\nabla_{\XX}H(t,x_{t}^{\pi},y_{t}^{\pi},y_{t}^{\prime\,\pi},z_{t}^{\pi})\,\D t+z_{t}^{\pi}\cdot\D w_{t}.
\end{align*}
To verify the terminal condition $y_{T}^{\pi}=D^{\pi}\nabla_{\XX}g(x_{T}^{\pi})$
in Eq.~(\ref{eq:relaxed-bsde}), note that from the definitions of
$y_{t}^{\pi}$ and $\Lambda_{t}^{\pi}$, 
\begin{align*}
y_{T}^{\pi} & =\Bigl(\Lambda_{T}^{\pi}-y_{T}^{\prime\,\pi}Q_{T}^{\pi}\Bigr)V_{T}^{\pi}=\Bigl(\EE\bigl[\Xi^{\pi}\bigm|\Ff_{T}\bigr]-D^{\pi}Q_{T}^{\pi}\Bigr)V_{T}^{\pi}\\
 & =\Bigl(D^{\pi}\nabla_{\XX}g(x_{T}^{\pi})U_{T}^{\pi}+D^{\pi}Q_{T}^{\pi}-D^{\pi}Q_{T}^{\pi}\Bigr)V_{T}^{\pi}\\
 & =D^{\pi}\nabla_{\XX}g(x_{T}^{\pi}).
\end{align*}

We can now establish that $y^{\pi}\in\Ss_{\Ff}^{\pbar/(\pbar-1)}(\Omega;\YY)$
and $z^{\pi}\in\Hh_{\Ff}^{\pbar/(\pbar-1)}(\Omega;\ZZ)$, or in fact,
a slightly strengthened version thereof. To proceed, we unfortunately
need to add to the notational clutter: Let $\pzap\in(\ptilde,\pbar)$,
and set $\qzap\deq1/(1/p-1/\pzap)$. Consider first the process $y^{\pi}$,
and let us estimate the terms in its definition, Eq.~(\ref{eq:temp-y-def-1})
individually. For the first term we obtain
\begin{align*}
\EE\biggl[\sup_{t\in\TT}\bigl|\Lambda_{t}^{\pi}V_{t}^{\pi}\bigr|^{\pzap/(\pzap-1)}\biggr] & \leq\EE\biggl[\sup_{t\in\TT}\bigl|\Lambda_{t}^{\pi}\bigr|^{\zeta\pzap/(\pzap-1)}\biggr]^{1/\zeta}\EE\biggl[\sup_{t\in\TT}\bigl|V_{t}^{\pi}\bigr|^{[\zeta/(\zeta-1)][\pzap/(\pzap-1)]}\biggr]^{1-1/\zeta}.
\end{align*}
Selecting $\zeta=[\ptilde/(\ptilde-1)]/[\pzap/(\pzap-1)]$, and using
Eqs.~(\ref{eq:V-S-norm}) and~(\ref{eq:Lambda-S-norm}), we find
that the above is finite. Note that $\zeta>1$ since $\ptilde<\pzap$.
The second term is treated as follows. 
\begin{align*}
\EE\biggl[\sup_{t\in\TT}\Bigl|y_{t}^{\prime\,\pi}Q_{t}^{\pi}V_{t}^{\pi}\Bigr|^{\pzap/(\pzap-1)}\biggr] & =\EE\biggl[\sup_{t\in\TT}\Bigl|y_{t}^{\prime\,\pi}\int_{0}^{t}F_{s}^{\pi}U_{s}^{\pi}\,\D sV_{t}^{\pi}\Bigr|^{\pzap/(\pzap-1)}\\
 & \leq\EE\biggl[\sup_{t\in\TT}\bigl|y_{t}^{\prime\,\pi}\bigr|^{\pzap/(\pzap-1)}\sup_{t\in\TT}\Bigl|\int_{0}^{t}F_{s}^{\pi}U_{s}^{\pi}\,\D sV_{t}^{\pi}\Bigr|^{\pzap/(\pzap-1)}\biggr]\\
 & \leq\EE\biggl[\sup_{t\in\TT}\Bigl|y_{t}^{\prime\,\pi}\Bigr|^{\pzap/(\pzap-1)}TL^{\pzap/(\pzap-1)}\sup_{t\in\TT}\Bigl(1+\bigl|x_{t}^{\pi}\bigr|^{p_{1}^{\prime}}+\int_{\AA}|a|^{p_{2}^{\prime}}\pi_{t}(\D a)\Bigr)^{\pzap/(\pzap-1)}\\
 & \qquad\times\biggl(\sup_{t\in\TT}\bigl|U_{t}^{\pi}\bigr|\sup_{t\in\TT}\bigl|V_{t}^{\pi}\bigr|\biggr)^{\pzap/(\pzap-1)}\biggr]\\
 & \leq TL^{\pzap/(\pzap-1)}\EE\biggl[\sup_{t\in\TT}\bigl|y_{t}^{\prime\,\pi}\bigr|^{r_{1}\pzap/(\pzap-1)}\biggr]^{\frac{1}{r_{1}}}\\
 & \qquad\times\EE\biggl[\sup_{t\in\TT}\Bigl(1+\bigl|x_{t}^{\pi}\bigr|^{p_{1}^{\prime}}+\int_{\AA}|a|^{p_{2}^{\prime}}\pi_{t}(\D a)\Bigr)^{r_{2}\pzap/(\pzap-1)}\biggr]^{\frac{1}{r_{2}}}\\
 & \qquad\times\EE\biggl[\biggl(\sup_{t\in\TT}\bigl|U_{t}^{\pi}\bigr|\sup_{t\in\TT}\bigl|V_{t}^{\pi}\bigr|\biggr)^{r_{3}\pzap/(\pzap-1)}\biggr]^{\frac{1}{r_{3}}},
\end{align*}
where $r_{1},r_{2},r_{3}\in(1,\infty)$ and $r_{1}^{-1}+r_{2}^{-1}+r_{3}^{-1}=1$.
We select $r_{1}=[p/(p-1)]/[\pzap/(\pzap-1)]$, as this is by Eq.~(\ref{eq:Lambda-S-norm})
the largest choice still ensuring the finiteness of the first factor.
Next, we set $r_{2}=p^{\ast}/[\pzap/(\pzap-1)]$ as this is sufficient
to guarantee the finiteness of the second factor, cf. Eq.~(\ref{eq:nabla-c-aux}).
By Eqs.~(\ref{eq:U-S-norm}) and~(\ref{eq:V-S-norm}), $r_{3}$
may be arbitrarily large, and we then only need to verify that $r_{1}^{-1}+r_{2}^{-1}<1$:
\begin{align*}
\frac{1}{r_{1}}+\frac{1}{r_{2}} & =\frac{\pzap}{\pzap-1}\biggl(\frac{1}{p/(p-1)}+\frac{1}{p^{\ast}}\biggr)\\
 & =\frac{\pzap}{\pzap-1}\biggl(1-\frac{1}{p}+\frac{1}{p}-\frac{1}{\phat}\biggr)\\
 & =\frac{1-1/\phat}{1-1/\pzap}\\
 & <1,
\end{align*}
where the final inequality is a simple consequence of $\phat<\pzap$,
or in this case, $-1/\phat<-1/\pzap$. We now have that
\begin{gather*}
\EE\biggl[\sup_{t\in\TT}\bigl|y_{t}^{\pi}\bigr|^{\pzap/(\pzap-1)}\biggr]<\infty.
\end{gather*}

Regarding the process $z^{\pi}$, we estimate each of terms in its
definition, Eq.~(\ref{eq:temp-z-def-1}) individually. For the first
and second terms, an almost identical calculation as above for the
first and second terms of $y^{\pi}$, but using Eq.~(\ref{eq:Sigma-L-norm})
instead of Eq.~(\ref{eq:Lambda-S-norm}), shows
\begin{align*}
\EE\biggl[\biggl(\int_{0}^{T}\bigl|\Sigma_{t}^{\pi}V_{t}^{\pi}\bigr|^{2}\D t\biggr)^{\frac{1}{2}\pzap/(\pzap-1)}\biggr] & \leq\EE\biggl[\biggl(\int_{0}^{T}\bigl|\Sigma_{t}^{\pi}\bigr|^{2}\D t\biggr)^{\frac{1}{2}\pzap/(\pzap-1)}\biggl(\sup_{t\in\TT}\bigl|V_{t}^{\pi}\bigr|^{\pzap/(\pzap-1)}\biggr)\biggr]<\infty,\\
\EE\biggl[\biggl(\int_{0}^{T}\bigl|z_{t}^{\prime\,\pi}Q_{t}^{\pi}V_{t}^{\pi}\bigr|^{2}\D t\biggr)^{\frac{1}{2}\pzap/(\pzap-1)}\biggr] & \leq\EE\biggl[\biggl(\int_{0}^{T}\bigl|z_{t}^{\prime\,\pi}\bigr|^{2}\D t\biggr)^{\frac{1}{2}\pzap/(\pzap-1)}\sup_{t\in\TT}\Bigl|\int_{0}^{t}F_{s}^{\pi}U_{s}^{\pi}\,\D sV_{t}^{\pi}\Bigr|^{\pzap/(\pzap-1)}\biggr]\\
 & <\infty.
\end{align*}
The last term in $z^{\pi}$, $y_{t}^{\pi}S_{t}^{\pi}$, can be treated
in the same way, noting the boundedness of $S_{t}^{\pi}$ for all
$t\in\TT$. Putting the above together, we have that
\begin{gather*}
\EE\biggl[\biggl(\int_{0}^{T}\bigl|z_{t}^{\pi}\bigr|^{2}\D t\biggr)^{\frac{1}{2}\pzap/(\pzap-1)}\biggr]<\infty,
\end{gather*}
and since $\pbar>\pzap$, or $\pbar/(\pbar-1)<\pzap/(\pzap-1)$. Therefore,
$y^{\pi}\in\Ss_{\Ff}^{\pbar/(\pbar-1)}(\Omega;\YY)$ and $z^{\pi}\in\Hh_{\Ff}^{\pbar/(\pbar-1)}(\Omega;\ZZ)$.

Finally, we prove Eq.~(\ref{eq:var-hamil}). The solution of Eq.~(\ref{eq:delta-solo})
can be written using the processes $U^{\pi}$ and $V^{\pi}$ as
\begin{align}
\delta_{t}^{\pi,q} & =U_{t}^{\pi}\int_{0}^{t}V_{s}^{\pi}\left[b(s,x_{s}^{\pi},q_{s}-\pi_{s})-S_{s}^{\pi}\cddot\sigma(s,x_{s}^{\pi},q_{s}-\pi_{s})\right]\,\D s\nonumber \\
 & \qquad+U_{t}^{\pi}\int_{0}^{t}V_{s}^{\pi}\sigma(s,x_{s}^{\pi},q_{s}-\pi_{s})\,\D w_{s}.\label{eq:temp-delta-hat-1}
\end{align}
Consider next the processes $\gamma^{\pi,q}=(\gamma_{t}^{\pi,q})_{t\in\TT}$
and $\gamma^{\prime\,\pi,q}=(\gamma_{t}^{\prime\,\pi,q})_{t\in\TT}$
taking respectively values in $\RR^{d_{x}}$ and $\RR$, and defined
as 
\begin{align*}
\gamma_{t}^{\pi,q} & \deq V_{t}^{\pi}\delta_{t}^{\pi,q}\quad\forall t\in\TT,\\
\gamma_{t}^{\prime\,\pi,q} & \deq-Q_{t}^{\pi}V_{t}^{\pi}\delta_{t}^{\pi,q}+\delta_{t}^{\prime\,\pi,q}\quad\forall t\in\TT,
\end{align*}
where $\delta^{\prime\,\pi,q}$ is as defined in Eq.~(\ref{eq:delta-prime}).
From Eq.~(\ref{eq:temp-delta-hat-1}) we immediately find that $\gamma^{\pi,q}$
and $\gamma^{\prime\,\pi,q}$ satisfy
\begin{align*}
\D\gamma_{t}^{\pi,q} & =V_{t}^{\pi}\left[b(t,x_{t}^{\pi},\pi_{t}-q_{t})-S_{t}^{\pi}\cddot\sigma(t,x_{t}^{\pi},\pi_{t}-q_{t})\right]\,\D t+V_{t}^{\pi}\sigma(t,x_{t}^{\pi},\pi_{t}-q_{t})\,\D w_{t},\\
\D\gamma_{t}^{\prime\,\pi,q} & =-Q_{t}^{\pi}\,\D\gamma_{t}^{\pi,q}+c(t,x_{t}^{\pi},q_{s}-\pi_{s})\,\D t.
\end{align*}
Note now that the expectation of $\Lambda_{T}^{\pi}\gamma_{T}^{\pi,q}+y_{T}^{\prime\,\pi}\gamma_{T}^{\prime\,\pi,q}$
equals the right-hand side of the inequality of Eq.~(\ref{eq:var-deltaform}):
\begin{align}
\EE[\Lambda_{T}^{\pi}\gamma_{T}^{\pi,q}+y_{T}^{\prime\,\pi}\gamma_{T}^{\prime\,\pi,q}] & =\EE\Bigl[\EE\bigl[\Xi^{\pi}\bigm|\Ff_{T}\bigr]V_{T}^{\pi}\delta_{T}^{\pi,q}+\EE\bigl[D^{\pi}\bigm|\Ff_{T}\bigr]\bigl(-Q_{t}^{\pi}V_{t}^{\pi}\delta_{t}^{\pi,q}+\delta_{t}^{\prime\,\pi,q}\bigr)\Bigr]\nonumber \\
 & =\EE\biggl[\Bigl(D^{\pi}\nabla_{\XX}g(x_{T}^{\pi})U_{T}^{\pi}+D^{\pi}Q_{T}^{\pi}\Bigr)V_{T}^{\pi}\delta_{T}^{\pi,q}\nonumber \\
 & \qquad+D^{\pi}\bigl(-Q_{t}^{\pi}V_{t}^{\pi}\delta_{t}^{\pi,q}+\delta_{t}^{\prime\,\pi,q}\bigr)\biggl]\nonumber \\
 & =\EE\biggl[D^{\pi}\nabla_{\XX}g(x_{T}^{\pi})\delta_{T}^{\pi,q}+D^{\pi}\delta_{t}^{\prime\,\pi,q}\biggl].\label{eq:temp-eee-1}
\end{align}
To compute the above, we differentiate $\Lambda_{t}^{\pi}\gamma_{t}^{\pi}+y_{t}^{\prime\,\pi}\gamma_{t}^{\prime\,\pi,q}$
to obtain\footnote{The following identities were used in simplifying the expressions
in this chain of equations: (\emph{i}) $(\Sigma_{t}^{\pi}\cdot\D w_{t})V_{t}^{\pi}=(\Sigma_{t}^{\pi}V_{t}^{\pi})\cdot\D w_{t}$;
(\emph{ii}) $(\Sigma_{t}^{\pi}\cdot\D w_{t})V_{t}^{\pi}\sigma(t,x_{t}^{\pi},q_{t}-\pi_{t})\,\D w_{t}=\trace[\Sigma_{t}^{\pi}V_{t}^{\pi}\sigma(t,x_{t}^{\pi},q_{t}-\pi_{t})]\,\D t$;
(\emph{iii}) $(z_{t}^{\prime\,\pi}\cdot\D w_{t})Q_{t}^{\pi}V_{t}^{\pi}\sigma(t,x_{t}^{\pi},q_{t}-\pi_{t})\,\D w_{t}=\trace[z_{t}^{\prime\,\pi}Q_{t}^{\pi}V_{t}^{\pi}\sigma(t,x_{t}^{\pi},q_{t}-\pi_{t})]\,\D t$;
(\emph{iv}) $y_{t}^{\pi}S_{t}^{\pi}\cddot\sigma(t,x_{t}^{\pi},q_{t}-\pi_{t})=\trace[y_{t}^{\pi}S_{t}^{\pi}\sigma(t,x_{t}^{\pi},q_{t}-\pi_{t})]$;
(\emph{v}) $(z_{t}^{\prime\,\pi}\cdot\D w_{t})Q_{t}^{\pi}V_{t}^{\pi}=(z_{t}^{\prime\,\pi}Q_{t}^{\pi}V_{t}^{\pi})\cdot\D w_{t}$.},
\begin{align*}
\D\bigl(\Lambda_{t}^{\pi}\gamma_{t}^{\pi}+y_{t}^{\prime\,\pi}\gamma_{t}^{\prime\,\pi,q}\bigr) & =\D\Lambda_{t}^{\pi}\gamma_{t}^{\pi}+\Lambda_{t}^{\pi}\D\gamma_{t}^{\pi}+\D\Lambda_{t}^{\pi}\D\gamma_{t}^{\pi}\\
 & \qquad+\D y_{t}^{\prime\,\pi}\gamma_{t}^{\prime\,\pi,q}+y_{t}^{\prime\,\pi}\D\gamma_{t}^{\prime\,\pi,q}+\D y_{t}^{\prime\,\pi}\D\gamma_{t}^{\prime\,\pi,q}\\
 & =\D\Lambda_{t}^{\pi}\gamma_{t}^{\pi}+\Bigl(\Lambda_{t}^{\pi}-y_{t}^{\prime\,\pi}Q_{t}^{\pi}\Bigr)\,\D\gamma_{t}^{\pi}+\D\Lambda_{t}^{\pi}\D\gamma_{t}^{\pi}\\
 & \qquad+\D y_{t}^{\prime\,\pi}\gamma_{t}^{\prime\,\pi,q}+y_{t}^{\prime\,\pi}c(t,x_{t}^{\pi},q_{t}-\pi_{t})\,\D t+\D y_{t}^{\prime\,\pi}\D\gamma_{t}^{\prime\,\pi,q}\\
 & =\bigl(\Sigma_{t}^{\pi}\cdot\D w_{t}\bigr)V_{t}^{\pi}\delta_{t}^{\pi,q}\\
 & \qquad+\Bigl(\Lambda_{t}^{\pi}-y_{t}^{\prime\,\pi}Q_{t}^{\pi}\Bigr)\biggl\{ V_{t}^{\pi}\Bigl[b(t,x_{t}^{\pi},q_{t}-\pi_{t})-S_{t}^{\pi}\cddot\sigma(t,x_{t}^{\pi},q_{t}-\pi_{t})\Bigr]\,\D t\\
 & \qquad\qquad\qquad\qquad\quad+V_{t}^{\pi}\sigma(t,x_{t}^{\pi},q_{t}-\pi_{t})\,\D w_{t}\biggr\}\\
 & \qquad+\bigl(\Sigma_{t}^{\pi}\cdot\D w_{t}\bigr)\biggl[V_{t}^{\pi}\sigma(t,x_{t}^{\pi},q_{t}-\pi_{t})\,\D w_{t}\biggr]\\
 & \qquad+\bigl(z_{t}^{\prime\,\pi}\cdot\D w_{t}\bigr)\Bigl(-Q_{t}^{\pi}V_{t}^{\pi}\delta_{t}^{\pi,q}+\delta_{t}^{\prime\,\pi,q}\Bigr)\\
 & \qquad+y_{t}^{\prime\,\pi}c(t,x_{t}^{\pi},q_{t}-\pi_{t})\,\D t-\bigl(z_{t}^{\prime\,\pi}\cdot\D w_{t}\bigr)Q_{t}^{\pi}V_{t}^{\pi}\sigma(t,x_{t}^{\pi},q_{t}-\pi_{t})\,\D w_{t}\\
 & =\bigl(\Sigma_{t}^{\pi}V_{t}^{\pi}\bigr)\cdot\D w_{t}\delta_{t}^{\pi,q}\\
 & \qquad+\biggl\{ y_{t}^{\pi}b(t,x_{t}^{\pi},q_{t}-\pi_{t})-\trace\Bigl[y_{t}^{\pi}S_{t}^{\pi}\sigma(t,x_{t}^{\pi},q_{t}-\pi_{t})\Bigr]\biggr\}\,\D t\\
 & \qquad+y_{t}^{\pi}\sigma(t,x_{t}^{\pi},q_{t}-\pi_{t})\,\D w_{t}\\
 & \qquad+\trace\Bigl[\Sigma_{t}^{\pi}V_{t}^{\pi}\sigma(t,x_{t}^{\pi},q_{t}-\pi_{t})\Bigr]\,\D t\\
 & \qquad-\bigl(z_{t}^{\prime\,\pi}Q_{t}^{\pi}V_{t}^{\pi}\bigr)\cdot\D w_{t}\delta_{t}^{\pi,q}+\bigl(z_{t}^{\prime\,\pi}\cdot\D w_{t}\bigr)\delta_{t}^{\prime\,\pi,q}\\
 & \qquad+y_{t}^{\prime\,\pi}c(t,x_{t}^{\pi},q_{t}-\pi_{t})\,\D t+\trace\Bigl[z_{t}^{\prime\,\pi}Q_{t}^{\pi}V_{t}^{\pi}\sigma(t,x_{t}^{\pi},q_{t}-\pi_{t})\Bigr]\,\D t\\
 & =\biggl\{ y_{t}^{\pi}b(t,x_{t}^{\pi},q_{t}-\pi_{t})+y_{t}^{\prime\,\pi}c(t,x_{t}^{\pi},q_{t}-\pi_{t})\\
 & \qquad\qquad+\trace\Bigl[\Bigl(\Sigma_{t}^{\pi}V_{t}^{\pi}+z_{t}^{\prime\,\pi}Q_{t}^{\pi}V_{t}^{\pi}-y_{t}^{\pi}S_{t}^{\pi}\Bigr)\sigma(t,x_{t}^{\pi},q_{t}-\pi_{t})\Bigr]\biggr\}\,\D t\\
 & \qquad+y_{t}^{\pi}\sigma(t,x_{t}^{\pi},q_{t}-\pi_{t})\,\D w_{t}\\
 & \qquad+\bigl(\Sigma_{t}^{\pi}V_{t}^{\pi}-z_{t}^{\prime\,\pi}Q_{t}^{\pi}V_{t}^{\pi}\bigr)\cdot\D w_{t}\delta_{t}^{\pi,q}+\bigl(z_{t}^{\prime\,\pi}\cdot\D w_{t}\bigr)\delta_{t}^{\prime\,\pi,q}\\
 & =\biggl\{ y_{t}^{\pi}b(t,x_{t}^{\pi},q_{t}-\pi_{t})+y_{t}^{\prime\,\pi}c(t,x_{t}^{\pi},q_{t}-\pi_{t})+\trace\Bigl[z_{t}^{\pi}\sigma(t,x_{t}^{\pi},q_{t}-\pi_{t})\Bigr]\biggr\}\,\D t\\
 & \qquad+\Bigl[\bigl(z_{t}^{\pi}+y_{t}^{\pi}S_{t}^{\pi}\bigr)\delta_{t}^{\pi,q}+y_{t}^{\pi}\sigma(t,x_{t}^{\pi},q_{t}-\pi_{t})\Bigr]\,\D w_{t}+\bigl(z_{t}^{\prime\,\pi}\delta_{t}^{\prime\,\pi,q}\bigr)\cdot\D w_{t}.
\end{align*}
Evaluating this at $t=T$, taking the expectation, and using Eq.~(\ref{eq:temp-eee-1}),
we get
\begin{align}
\EE\biggl[D^{\pi}\nabla_{\XX}g(x_{T}^{\pi})\delta_{T}^{\pi,q} & +D^{\pi}\delta_{t}^{\prime\,\pi,q}\biggl]\nonumber \\
 & =\EE\biggl[\int_{0}^{T}\biggl\{ y_{s}^{\pi}b(s,x_{s}^{\pi},q_{s}-\pi_{s})+y_{s}^{\prime\,\pi}c(s,x_{s}^{\pi},q_{s}-\pi_{s})\nonumber \\
 & \qquad\qquad+\trace\Bigl[z_{s}^{\pi}\sigma(s,x_{s}^{\pi},q_{s}-\pi_{s})\Bigr]\biggr\}\,\D s\biggl]+M^{\pi,q},\label{eq:temp-var-hamil}
\end{align}
where
\begin{gather*}
M^{\pi,q}\deq\EE\biggl[\int_{0}^{T}\Bigl[\bigl(z_{t}^{\pi}+y_{t}^{\pi}S_{t}^{\pi}\bigr)\delta_{t}^{\pi,q}+y_{t}^{\pi}\sigma(t,x_{t}^{\pi},q_{t}-\pi_{t})\Bigr]\,\D w_{t}+\int_{0}^{T}\bigl(z_{t}^{\prime\,\pi}\delta_{t}^{\prime\,\pi,q}\bigr)\cdot\D w_{t}\biggr].
\end{gather*}
Eq.~(\ref{eq:var-hamil}) follows now immediately from Eqs.~(\ref{eq:var-deltaform})
and~(\ref{eq:temp-var-hamil}), if $M^{\pi,q}=0$. This is the case
if the integrands in the expression are in $\Hh_{\Ff}^{1}(\Omega;\RR^{1\times d_{w}})$,
see e.g. \cite[Theorem 2.6]{Pardoux2014}. Using the estimates of
Eqs.~(\ref{eq:delta-finite}) for $\delta^{\pi},\delta^{\prime\,\pi}$,
and the bounds on $S_{t}^{\pi}$ and $\sigma$, the square integrability
can be verified with a straight-forward application of Hölder's inequality.
For example, for the first term,
\begin{align*}
\EE\biggl[\biggl(\int_{0}^{T}\bigl|z_{t}^{\pi}\delta_{t}^{\pi,q}\bigr|^{2}\D t\biggr)^{1/2}\biggr] & \leq\EE\biggl[\biggl(\int_{0}^{T}\bigl|z_{t}^{\pi}\bigr|^{2}\D t\biggr)^{1/2}\sup_{t\in\TT}\bigl|\delta_{t}^{\pi,q}\bigr|\biggr]\\
 & \leq\EE\biggl[\biggl(\int_{0}^{T}\bigl|z_{t}^{\pi}\bigr|^{2}\D t\biggr)^{\frac{1}{2}\frac{\pbar}{\pbar-1}}\biggr]^{1-\frac{1}{\pbar}}\EE\biggl[\sup_{t\in\TT}\bigl|\delta_{t}^{\pi,q}\bigr|^{\pbar}\biggr]^{\frac{1}{\pbar}}\\
 & <\infty,
\end{align*}
and the rest follow analogously. If $p=1$, the last term is somewhat
special, since we cannot apply Hölder's inequality in the same way
(we have no $\Ll^{\infty}$-norm equivalent bound on $z^{\prime\,\pi}$).
However, the estimate of Eq.~(\ref{eq:delta-prime-p-finite}) is
slightly stronger than that of Eq.~(\ref{eq:delta-pbar-finite})
precisely to accommodate this edge case. Therefore, $M^{\pi,q}=0$,
and the proof is complete.
\end{proof}
\begin{proof}[Proof of Lemma~\ref{lem:suff}]
let $q\in\VvV^{p}(b,\sigma,\nu)$ be arbitrary. We have from the
$\Ll$-convexity of $\rho$ and convexity of $g$ that 
\begin{align*}
\rho(\Theta^{\pi})-\rho(\Theta^{q}) & =\rho(g(x_{T}^{\pi})+x_{T}^{\prime\,\pi})-\rho(g(x_{T}^{q})+x_{T}^{\prime\,q})\\
 & \leq\EE\left[\DF\rho(\Theta^{\pi})(g(x_{T}^{\pi})+x_{T}^{\prime\,\pi})(g(x_{T}^{\pi})+x_{T}^{\prime\,\pi}-g(x_{T}^{q})-x_{T}^{\prime\,q})\right]\\
 & \leq\EE\left[D^{\pi}(\nabla_{\XX}g(x_{T}^{\pi})(x_{T}^{\pi}-x_{T}^{q})+x_{T}^{\prime\,\pi}-x_{T}^{\prime\,q})\right].
\end{align*}
By using the construction of the process $(y_{t}^{\pi},y_{t}^{\prime\,\pi})_{t\in\TT}$,
the above becomes
\begin{gather*}
\rho(\Theta^{\pi})-\rho(\Theta^{q})\leq\EE\left[y_{T}^{\pi}(x_{T}^{\pi}-x_{T}^{q})+y_{T}^{\prime\,\pi}(x_{T}^{\prime\,\pi}-x_{T}^{\prime\,q})\right].
\end{gather*}
We can evaluate the above expectation by first differentiating $y_{t}^{\pi}(x_{t}^{\pi}-x_{t}^{q})+y_{t}^{\prime\,\pi}(x_{t}^{\prime\,\pi}-x_{t}^{\prime\,q})$,
using the $y$ and $x$ differential equations~(\ref{eq:vague-sde},~\ref{eq:sde-aug},~\ref{eq:relaxed-bsde}),
evaluating the integrals at $T$, and then taking expectations:
\begin{align*}
\D\Bigl[y_{t}^{\pi}(x_{t}^{\pi}-x_{t}^{q}) & +y_{t}^{\prime\,\pi}(x_{t}^{\prime\,\pi}-x_{t}^{\prime\,q})\Bigr]\\
 & =\Bigl[-\nabla_{\XX}H(t,x_{t}^{\pi},y_{t}^{\pi},y_{t}^{\prime\,\pi},z_{t}^{\pi},\pi_{t})(x_{t}^{\pi}-x_{t}^{q})+y_{t}^{\pi}(b(t,x_{t}^{\pi},\pi_{t})-b(t,x_{t}^{q},q_{t}))\\
 & \qquad+y_{t}^{\prime\,\pi}(c(t,x_{t}^{\pi},\pi_{t})-c(t,x_{t}^{q},q_{t}))+\trace\left(z_{t}^{\pi}(\sigma(t,x_{t}^{\pi},\pi_{t})-\sigma(t,x_{t}^{q},q_{t}))\right)\Bigr]\,\D t\\
 & \qquad+\left[z_{t}^{\pi}(x_{t}^{\pi}-x_{t}^{q})+z_{t}^{\prime\,\pi}(x_{t}^{\prime\,\pi}-x_{t}^{\prime\,q})+y_{t}^{\pi}(\sigma(t,x_{t}^{\pi},\pi_{t})-\sigma(t,x_{t}^{q},q_{t}))\right]\,\D w_{t},
\end{align*}
so that
\begin{align}
\EE\Bigl[y_{T}^{\pi}(x_{T}^{\pi}-x_{T}^{q}) & +y^{\prime\,\pi}(x_{T}^{\prime\,\pi}-x_{T}^{\prime\,q})\Bigr]\nonumber \\
 & =\EE\biggl[\int_{0}^{T}\Bigl(-\nabla_{\XX}H(t,x_{t}^{\pi},y_{t}^{\pi},y_{t}^{\prime\,\pi},z_{t}^{\pi},\pi_{t})(x_{t}^{\pi}-x_{t}^{q})\nonumber \\
 & \qquad\qquad+H(t,x_{t}^{\pi},y_{t}^{\pi},y_{t}^{\prime\,\pi},z_{t}^{\pi},\pi_{t})-H(t,x_{t}^{q},y_{t}^{\pi},y_{t}^{\prime\,\pi},z_{t}^{\pi},q_{t})\Bigr)\,\D t\biggr].\label{eq:suff-yx-optim}
\end{align}
The expectation of the integrals against the Brownian motion is zero
\cite[Theorem 2.6]{Pardoux2014}, since the integrands are in $\Hh_{\Ff}^{1}(\Omega;\RR)$.

Let us denote
\begin{align*}
h_{t}(x,q) & \deq H(t,x,y_{t}^{\pi},y_{t}^{\prime\,\pi},z_{t}^{\pi},q)\\
\nabla h_{t}(x,q) & \deq\nabla_{\XX}H(t,x,y_{t}^{\pi},y_{t}^{\prime\,\pi},z_{t}^{\pi},q)\\
 & \qquad\qquad\forall(t,x,q)\in\TT\times\XX\times\Pp^{\pbar_{3}}(\AA).
\end{align*}
Since we have assumed that $H$ is jointly convex in the state and
control variables, we have that for any $\alpha\in(0,1]$,
\begin{multline*}
\frac{h_{t}\left((1-\alpha)x_{0}+\alpha x_{1},(1-\alpha)\pi_{0}+\alpha\pi_{1}\right)-h_{t}\left(x_{0},\pi_{0}\right)}{\alpha}\leq h_{t}\left(x_{1},\pi_{1}\right)-h_{t}\left(x_{0},\pi_{0}\right)\\
\forall(t,x_{0},x_{1},\pi_{0},\pi_{1})\in\TT\times\XX\times\XX\times\Pp^{\pbar_{3}}(\AA)\times\Pp^{\pbar_{3}}(\AA).
\end{multline*}
On the other hand, using the differentiability of $H$ on $\XX$,
\begin{align*}
\lim_{\alpha\to0}\frac{1}{\alpha}\Bigl[h_{t} & \left((1-\alpha)x_{0}+\alpha x_{1},(1-\alpha)\pi_{0}+\alpha\pi_{1}\right)-h_{t}\left(x_{0},\pi_{0}\right)\Bigr]\\
 & =\lim_{\alpha\to0}\frac{1}{\alpha}\Bigl[(1-\alpha)h_{t}\left((1-\alpha)x_{0}+\alpha x_{1},\pi_{0}\right)\\
 & \qquad\qquad+\alpha h_{t}\left((1-\alpha)x_{0}+\alpha x_{1},\pi_{1}\right)-h_{t}\left(x_{0},\pi_{0}\right)\Bigr]\\
 & =\lim_{\alpha\to0}\frac{1}{\alpha}\Bigl\{ h_{t}\left((1-\alpha)x_{0}+\alpha x_{1},\pi_{0}\right)-h_{t}\left(x_{0},\pi_{0}\right)\\
 & \qquad\qquad+\alpha\Bigl[h_{t}\left((1-\alpha)x_{0}+\alpha x_{1},\pi_{1}\right)-h_{t}\left((1-\alpha)x_{0}+\alpha x_{1},\pi_{0}\right)\Bigr]\Bigr\}\\
 & =\nabla h_{t}(x_{0},\pi_{0})(x_{1}-x_{0})+h_{t}(x_{0},\pi_{1})-h_{t}(x_{0},\pi_{0}),
\end{align*}
so that 
\begin{multline*}
\nabla h_{t}(x_{0},\pi_{0})(x_{1}-x_{0})+h_{t}(x_{0},\pi_{1})-h_{t}(x_{0},\pi_{0})\leq h_{t}\left(x_{1},\pi_{1}\right)-h_{t}\left(x_{0},\pi_{0}\right)\\
\forall(t,x_{0},x_{1},\pi_{0},\pi_{1})\in\TT\times\XX\times\XX\times\Pp^{\pbar_{3}}(\AA)\times\Pp^{\pbar_{3}}(\AA).
\end{multline*}
Using the above, along with the assumption that $\pi_{t}$ minimizes
$\eta\to h_{t}(x_{t}^{\pi},\eta)$ for $\PP\times\D t$-almost every
$(\omega,t)\in\Omega\times\TT$, we have that
\begin{align*}
h_{t}\left(x_{t}^{q},q_{t}\right)-h_{t}\left(x_{t}^{\pi},\pi_{t}\right) & \geq\nabla h_{t}(x_{t}^{\pi},\pi_{t})(x_{t}^{q}-x_{t}^{\pi})+h_{t}(x_{t}^{\pi},q_{t})-h_{t}(x_{t}^{\pi},\pi_{t})\\
 & \geq\nabla h_{t}(x_{t}^{\pi},\pi_{t})(x_{t}^{q}-x_{t}^{\pi})
\end{align*}
$\PP\times\D t$-almost always. Applying this estimate in Eq.~(\ref{eq:suff-yx-optim}),
we get
\begin{align*}
\EE\Bigl[y_{T}^{\pi}(x_{T}^{\pi}-x_{T}^{q}) & +y_{T}^{\prime\,\pi}(x_{T}^{\prime\,\pi}-x_{T}^{\prime\,q})\Bigr]\leq0,
\end{align*}
implying that 
\begin{gather*}
\rho(\Theta^{\pi})-\rho(\Theta^{q})\leq0,
\end{gather*}
and the proof is complete.
\end{proof}

\subsection{Proofs for Section 5}
\begin{proof}[Proof of Lemma~\ref{lem:easy-derivatives}]
(\emph{i}) Note that $X=\EE[X]$ if and only if $X$ is almost surely
constant. The Fréchet derivative of the first term in Eq.~(\ref{eq:rho-mean-dev}),
the expectation, is clearly $\DF\EE[X]=1$. Focusing then on the derivative
of the second, norm term, we first note that the derivative of $\Ll^{2}(\Omega;\RR)\ni X\to\Vert X\Vert_{2}\in\RR_{\geq0}$
is $\Vert X\Vert_{2}^{-1}X\in\Ll^{2}(\Omega;\RR)$, which suggests
that
\begin{gather*}
\langle\DF\Vert X-\EE[X]\Vert_{2},Y\rangle=\left\langle \frac{X-\EE[X]}{\bigl\Vert X-\EE[X]\bigr\Vert_{2}},Y-\EE[Y]\right\rangle \quad\forall Y\in\Ll^{2}(\Omega;\RR).
\end{gather*}
This can be verified through a direct calculation: By straight-forward
algebraic manipulation, one obtains
\begin{multline*}
\frac{\bigl\Vert X+Y-\EE\left[X+Y\right]\bigr\Vert_{2}-\Vert X-\EE[X]\Vert_{2}-\left\langle \frac{X-\EE[X]}{\bigl\Vert X-\EE\left[X\right]\bigr\Vert_{2}},Y-\EE[Y]\right\rangle }{\Vert Y\Vert_{2}}\\
=\frac{\Vert Y-\EE[Y]\Vert_{2}^{2}\Vert X-\EE[X]\Vert_{2}^{2}-\left\langle X-\EE[X],Y-\EE[Y]\right\rangle ^{2}}{\Vert X-\EE[X]\Vert_{2}\Vert Y\Vert_{2}\left[\bigl\Vert X+Y-\EE\left[X+Y\right]\bigr\Vert_{2}\Vert X-\EE[X]\Vert_{2}+\left\langle X-\EE[X],X-\EE[X]-Y+\EE[Y]\right\rangle \right]}.
\end{multline*}
Taking the limit $\Vert Y\Vert_{2}\to0$, the following is quickly
recovered
\begin{multline*}
\lim_{\Vert Y\Vert_{2}\to0}\frac{\bigl\Vert X+Y-\EE\left[X+Y\right]\bigr\Vert_{2}-\Vert X-\EE[X]\Vert_{2}-\left\langle \frac{X-\EE[X]}{\bigl\Vert X-\EE\left[X\right]\bigr\Vert_{2}},Y-\EE[Y]\right\rangle }{\Vert Y\Vert_{2}}\\
=\frac{1}{2\bigl\Vert X-\EE[X]\bigr\Vert_{2}^{3}}\lim_{\Vert Y\Vert_{2}\to0}\frac{\left\langle X-\EE[X],X-\EE[X]\right\rangle \left\langle Y-\EE[Y],Y-\EE[Y]\right\rangle -\left\langle X-\EE[X],Y-\EE[Y]\right\rangle ^{2}}{\Vert Y\Vert_{2}}.
\end{multline*}
The right-hand side is clearly zero. Noting that $\langle X-\EE[X],1\rangle=0$,
Eq.~(\ref{eq:D-mean-dev}) follows. The non-differentiability at
almost surely constant random variables follows from the positive
homogeneity and translation invariance of $\rho$: At $X=x$, $x\in\RR$,
we have $\rhoMD(X)=\rhoMD(0)+x$, but at $X=0$, we have that $\lim_{\epsilon\to0}\epsilon^{-1}[\rhoMD(0+\epsilon Y)-\rhoMD(0)]=\lim_{\epsilon\to0}\epsilon^{-1}\rhoMD(\epsilon Y)=\rhoMD(Y)$
which is not linear. Eq.~(\ref{eq:LD-mean-dev}) is easily found
from the form of the Fréchet derivative of Eq.~(\ref{eq:D-mean-dev}).

(\emph{ii}) We first note that
\begin{gather}
(x+h)_{\epsilon+}-(x)_{\epsilon+}=h\int_{0}^{1}U_{\epsilon}(x+h\xi)\,\D\xi\quad\forall x,h\in\RR,\label{eq:temp-plus}\\
U_{\epsilon}(x+h)-U_{\epsilon}(x)=h\int_{0}^{1}U_{\epsilon}^{\prime}(x+h\xi)\,\D\xi\quad\forall x,h\in\RR.\nonumber 
\end{gather}
Since $U_{\epsilon}^{\prime}(x)=\epsilon^{-1}\E^{-x/\epsilon}/(1+\E^{-x/\epsilon})^{2}\in(0,1/(4\epsilon)]\,\forall x\in\RR$,
from the second equality it follows that
\begin{gather}
|U_{\epsilon}(x+h)-U_{\epsilon}(x)|\leq\frac{|h|}{4\epsilon}\quad\forall x,h\in\RR.\label{eq:temp-U-estim}
\end{gather}
We verify Eq.~(\ref{eq:D-mean-semidev}) by a direct calculation.
Consider now
\begin{align*}
\Bigl|\rhoEpsMSD(X+H)-\rhoEpsMSD(X) & -\EE[\DF\rhoEpsMSD(X)H]\Bigr|\\
 & =\Bigl|\EE[H]+\beta\EE\left[(X-\EE\left[X\right]+H-\EE[H])_{\epsilon+}-(X-\EE[X])_{\epsilon+}\right]\\
 & \qquad-\EE\left[H\right]-\beta\EE\left[U_{\epsilon}(X-\EE[X])H\right]+\beta\EE\left[U_{\epsilon}(X-\EE[X])\right]\EE\left[H\right]\Bigr|\\
 & =\beta\Bigl|\EE\left[(X-\EE\left[X\right]+H-\EE[H])_{\epsilon+}-(X-\EE[X])_{\epsilon+}\right]\\
 & \qquad-\EE\left[U_{\epsilon}(X-\EE[X])(H-\EE\left[H]\right])\right]\Bigr|.
\end{align*}
Next, using the identity of Eq.~(\ref{eq:temp-plus}), followed by
Hölder's inequality, and the estimate of Eq.~(\ref{eq:temp-U-estim}),
we get that
\begin{align*}
\Bigl|\rhoEpsMSD(X+H)-\rhoEpsMSD(X) & -\EE[\DF\rhoEpsMSD(X)H]\Bigr|\\
 & =\beta\Bigl|\EE\left[(H-\EE[H])\left(\int_{0}^{1}U_{\epsilon}(X-\EE[X]+(H-\EE[H])\xi)\,\D\xi-U_{\epsilon}(X-\EE[X])\right)\right]\Bigl|\\
 & \leq\beta\Vert H-\EE[H]\Vert_{2}\left\Vert \int_{0}^{1}\left[U_{\epsilon}\left(X-\EE[X]+(H-\EE[H])\xi\right)-U_{\epsilon}(X-\EE[X])\right]\,\D\xi\right\Vert _{2}\\
 & \leq\beta\Vert H-\EE[H]\Vert_{2}\left\Vert \int_{0}^{1}\frac{|H-\EE[H]|\xi}{4\epsilon}\,\D\xi\right\Vert _{2}\\
 & =\frac{\beta}{8\epsilon}\Vert H-\EE[H]\Vert_{2}^{2}.
\end{align*}
This is sufficient to show that $\rhoEpsMSD(X+H)-\rhoEpsMSD(X)-\EE[\DF\rhoEpsMSD(X)H]\in o(\Vert H\Vert_{2})$
for all $X,H\in\Ll^{2}(\Omega;\RR)$, and so $\rhoEpsMSD$ is Fréchet
differentiable on $\Ll^{2}(\Omega;\RR)$ with the given derivative
$\DF\rhoEpsMSD(X)$ for all $X\in\Ll^{2}(\Omega;\RR)$. The form of
the $\Ll$-derivative is easily verified from $\DF\rhoEpsMSD(X)$.

(\emph{iii}) It suffices to show that the limit in Eq.~(\ref{eq:dir-deriv})
is attained uniformly over $Y\in\Ll^{\infty}(\Omega;\RR)$ such that
$\Vert Y\Vert_{\infty}=1$. We now have that $\EE[\E^{\theta(X+\epsilon Y)}]-\EE[\E^{\theta X}]=\epsilon\EE[\E^{\theta X}\theta Y]+o(\epsilon)$
by Taylor series expanding the exponential and using the fact $Y(\omega)\leq\Vert Y\Vert_{\infty}$
almost everywhere. By the chain rule of differentiation, $\DF\rhoEntr(X)=\E^{\theta X}/\EE[\E^{\theta X}]$
follows. The $\Ll$-derivative is similarly easily found.
\end{proof}
\begin{proof}[Proof of Proposition~\ref{prop:alloc}]
We first note that Assumptions~\ref{assu:sde-baseline} are easily
verified for the stochastic differential equations of Problem~$\pP_{\phi}$.
Suppose $\pi\in\VvV^{p}(b_{\phi},\sigma_{\phi},\nu_{\phi})$ is $\pP_{\phi}$-optimal.
The Hamiltonian of Eq.~(\ref{eq:relaxed-hamiltonian-1}) becomes
\begin{gather*}
H(y,z,\phi)=y\left[r+(\mu-r)\phi-\frac{1}{2}\sigma^{2}\phi^{2}\right]+z\sigma\phi\quad\forall(y,z,\phi)\in\YY\times\ZZ\times\AA.
\end{gather*}
By Theorem~\ref{thm:ra-min-principle}, we know that there exists
a process $(y_{t}^{\pi},y_{t}^{\prime\,\pi},z_{t}^{\pi},z_{t}^{\prime\,\pi})_{t\in\TT}$
where by Eqs.~(\ref{eq:y-bsde}, \ref{eq:y-prime}, \ref{eq:y-prime-bsde}),
$y_{t}^{\prime\,\pi}=\EE\bigl[\DF\rho(\lL(-x_{T}^{\pi}))(-x_{T}^{\pi})\bigm|\Ff_{t}\bigr]$,
$\D y_{t}^{\prime\,\pi}=z_{t}^{\prime\,\pi}\,\D w_{t}$, and $\D y_{t}^{\pi}=z_{t}^{\pi}\,\D w_{t}$,
$y_{T}=-y_{T}^{\prime}=-\DF\rho(\lL(-x_{T}^{\pi}))(-x_{T}^{\pi})$.
This yields Eqs.~(\ref{eq:phi-y-prime}) and~(\ref{eq:phi-y-prime-bsde}).
By the uniqueness of solutions of Lipschitz backward differential
equations, see e.g. \cite[Theorem 5.17]{Pardoux2014}, we have that
$y_{t}^{\pi}=-y_{t}^{\prime\,\pi}$ and $z_{t}^{\pi}=-z_{t}^{\prime\,\pi}$
for all $t\in\TT$, $\PP$-almost surely. From the assumption of positivity
of $\DF\rho(\cdot)(\cdot)$ we infer that $y_{t}^{\prime\,\pi}>0$
and $y_{t}^{\pi}<0$ for all $t\in\TT$. We can thus assume in the
following that $\YY=\RR_{<0}$ and $\YY^{\prime}=\RR_{>0}$. 

Since $H$ as a function of the control, $\phi\to H(y,z,\phi)$ is
strictly convex for all $(y,z)\in\YY\times\ZZ$, by an elementary
application of Jensen's inequality we see that a minimizer of Eq.~(\ref{eq:H-inf-1})
is found in Dirac measures: For any $\pi\in\Pp(\AA)$, we have that
\begin{gather*}
H(y,z,\bar{\phi})\leq\int_{\AA}H(y,z,\phi)\,\pi(\D\phi),
\end{gather*}
for all $(y,z)\in\YY\times\ZZ$, where $\bar{\phi}=\int_{\AA}\phi\,\pi(\D\phi)$.
If $\pi^{\ast}\in\Pp(\AA)$ is a minimizer of $\pi\to H(y,z,\pi)=\int_{\AA}H(y,z,\phi)\,\pi(\D\phi)$
and $\bar{\phi}^{\ast}=\int_{\AA}\phi\,\pi^{\ast}(\D\phi)$, then
using the convexity of $\AA$, $\bar{\phi}^{\ast}\in\AA$ we have
the Dirac measure $\delta_{\bar{\phi}^{\ast}}$ such that $H(y,z,\delta_{\bar{\phi}^{\ast}})\leq H(y,z,\pi^{\ast})$.
Therefore a minimizer is always found within the set of Dirac measures,
and we can consider the problem
\begin{gather}
\inf_{\phi\in\AA}H(y,z,\phi).\label{eq:infH-temp}
\end{gather}
This immediately yields as a minimizer the function $\phi^{\ast}:$$\YY\times\ZZ\to\AA$
such that
\begin{gather*}
\phi^{\ast}(y,z)=\underline{\phi}\vee\frac{\mu-r+\sigma z/y}{\sigma^{2}}\wedge\bar{\phi}\quad\forall(y,z)\in\YY\times\ZZ.
\end{gather*}
Since $H$ and the terminal cost function are convex in $(x,\phi)$
and $x$, respectively, and $\rho$ is $\Ll$-convex by the assumption
of the proposition, Assumption~\ref{assu:sufrel} holds and by Theorem~\ref{thm:ra-min-principle}(\emph{ii})
the above properties are also sufficient for $\phi_{t}=\phi^{\ast}(y_{t}^{\pi},z_{t}^{\pi})$
to be $\pP_{\phi}$-optimal. 
\end{proof}
\bibliographystyle{siam}
\bibliography{Lit/rafbsde}

\end{document}